\documentclass[11pt,a4paper]{article}
\usepackage[utf8]{inputenc}

\usepackage{mathrsfs,amsthm,xcolor,verbatim,bbm,amsmath,amsfonts,amssymb,nicefrac,enumitem,hyperref,bm,mathtools,xparse,etoolbox}
\usepackage[margin=1in]{geometry}
\usepackage[capitalise]{cleveref} 

\usepackage{stmaryrd} 

\crefname{enumi}{item}{items}
\crefname{equation}{}{}
\crefname{subsection}{Subsection}{Subsections}

\usepackage{xurl}

\hypersetup{
    colorlinks,
    linkcolor={red!80!black},
    citecolor={green},
    urlcolor={blue!80!black}
}

\theoremstyle{plain}
\newtheorem{theorem}{Theorem}[section]
\newtheorem{lemma}[theorem]{Lemma}
\newtheorem{prop}[theorem]{Proposition}
\newtheorem{cor}[theorem]{Corollary}

\newtheorem{setting}[theorem]{Setting}
\newtheorem{conjecture}[theorem]{Conjecture}

\theoremstyle{definition}
\newtheorem{definition}[theorem]{Definition}

\newtheorem{remark}[theorem]{Remark}

\DeclareMathAlphabet{\mathpzc}{OT1}{pzc}{m}{it}

\DeclareFontEncoding{LS1}{}{}
\DeclareFontSubstitution{LS1}{stix}{m}{n}
\DeclareMathAlphabet{\mathscr}{LS1}{stixscr}{m}{n}

\newcommand{\E}{\mathbb{E}}
\renewcommand{\P}{\mathbb{P}}

\newcommand{\R}{\mathbb{R}}
\newcommand{\N}{\mathbb{N}}

\newcommand{\bbL}{\mathbb{L}}
\newcommand{\dens}{\mathfrak{p}}
\newcommand{\Dens}{\mathscr{p}}

\newcommand{\Rect}{\mathfrak{R}}

\newcommand{\nsign}{\cS}
\newcommand{\inact}{\cI}
\newcommand{\globinf}{\mathbb{I}}

\renewcommand{\d}{ \mathrm{d}}

\newcommand{\ssum}{\textstyle\sum}
\newcommand{\ssuml}{\textstyle\sum\limits}
\newcommand{\tint}{\textstyle\int}

\newcommand{\cA}{\mathcal{A}}
\newcommand{\cB}{\mathcal{B}}

\newcommand{\cF}{\mathcal{F}}
\newcommand{\cG}{\mathcal{G}}

\newcommand{\cI}{\mathcal{I}}

\newcommand{\cL}{\mathcal{L}}

\newcommand{\cN}{\mathcal{N}}

\newcommand{\cP}{\mathcal{P}}

\newcommand{\cS}{\mathcal{S}}

\newcommand{\bfa}{\mathbf{a}}

\newcommand{\bfd}{\mathbf{d}}
\newcommand{\bfe}{\mathbf{e}}

\newcommand{\bfl}{\mathbf{l}}
\newcommand{\bfm}{\mathbf{m}}

\newcommand{\scrA}{\mathscr{A}}
\newcommand{\scrB}{\mathscr{B}}

\newcommand{\scrK}{\mathscr{K}}

\newcommand{\fC}{\mathfrak{C}}

\newcommand{\fG}{\mathfrak{G}}
\newcommand{\fH}{\mathfrak{H}}

\newcommand{\fL}{\mathfrak{L}}
\newcommand{\fM}{\mathfrak{M}}

\newcommand{\fU}{\mathfrak{U}}

\newcommand{\fb}{\mathfrak{b}}

\newcommand{\fd}{\mathfrak{d}}

\newcommand{\fw}{\mathfrak{w}}

\renewcommand{\emptyset}{\varnothing}

\DeclarePairedDelimiter{\norm}{\lVert}{\rVert}
\DeclarePairedDelimiter{\abs}{\lvert}{\rvert}
\DeclarePairedDelimiter{\rbr}{(}{)}
\DeclarePairedDelimiter{\br}{[}{]}
\DeclarePairedDelimiter{\cu}{\{}{\}}
\DeclarePairedDelimiter{\spro}{\langle}{\rangle}
\DeclarePairedDelimiter{\bbrr}{\llbracket}{\rrbracket}

\newcommand{\qandq}{\quad \text{and} \quad }
\newcommand{\qqandqq}{\qquad\text{and}\qquad}
\newcommand{\andq}{\text{and}\qquad}

\newcommand{\num}[1] { \bbrr{ #1 } }

\newcommand{\indicator}[1]{\mathbbm{1}_{\smash{#1}}}

\newcommand{\realization}[1]{\mathscr{N}^{ #1 }}

\newcommand{\width}{\mathfrak{h}}

\newcommand{\bbM}{\mathbb{M}}

\ExplSyntaxOn

\bool_new:N \g_noteobserve

\NewDocumentCommand{\setnote}{}{
  \bool_gset_true:N \g_noteobserve
}

\NewDocumentCommand{\setobserve}{}{
  \bool_gset_false:N \g_noteobserve
}

\NewDocumentCommand{\nobs}{ o }{
  \IfValueT{#1}{
    \str_if_eq:noTF {note} {#1} {
      \bool_gset_true:N \g_noteobserve
    } {
      \str_if_eq:noTF {Note} {#1} {
        \bool_gset_true:N \g_noteobserve
      } {
        \bool_gset_false:N \g_noteobserve
      }
    }
  }
  \bool_if:nTF { \g_noteobserve } {
    \bool_gset_false:N \g_noteobserve
    note
  } {
    \bool_gset_true:N \g_noteobserve
    observe
  }
  \IfValueF{#1}{~}
}

\NewDocumentCommand{\Nobs}{ o }{
  \IfValueT{#1}{
    \str_if_eq:noTF {note} {#1} {
      \bool_gset_true:N \g_noteobserve
    } {
      \str_if_eq:noTF {Note} {#1} {
        \bool_gset_true:N \g_noteobserve
      } {
        \bool_gset_false:N \g_noteobserve
      }
    }
  }
  \bool_if:nTF { \g_noteobserve } {
    \bool_gset_false:N \g_noteobserve
    Note
  } {
    \bool_gset_true:N \g_noteobserve
    Observe
  }
  \IfValueF{#1}{~}
}

\int_new:N \g_furthermore

\NewDocumentCommand{\Moreover}{ o o }{
  \IfValueT{#1}{
    \str_case:nn {#1} {
      {Furthermore} {\int_set:Nn {\g_furthermore} {0}}
      {Moreover} {\int_set:Nn {\g_furthermore} {1}}
      {In~addition} {\int_set:Nn {\g_furthermore} {2}}
      {note} {\bool_gset_true:N \g_noteobserve}
      {observe} {\bool_gset_false:N \g_noteobserve}
    }
    \IfValueT{#2}{
      \str_case:nn {#2} {
        {Furthermore} {\int_set:Nn {\g_furthermore} {0}}
        {Moreover} {\int_set:Nn {\g_furthermore} {1}}
        {In~addition} {\int_set:Nn {\g_furthermore} {2}}
        {note} {\bool_gset_true:N \g_noteobserve}
        {observe} {\bool_gset_false:N \g_noteobserve}
      }
    }
  }
  \int_case:nn { \int_mod:nn {\g_furthermore} {3} } {
    { 0 } { Furthermore,~\nobs that}
    { 1 } { Moreover,~\nobs that}
    { 2 } { In~addition,~\nobs that}
  }
  \int_incr:N \g_furthermore
  \IfValueF{#1}{~}
}

\bool_new:N \g_hencetherefore

\NewDocumentCommand{\hence}{}{
  \bool_if:nTF { \g_hencetherefore } {
    \bool_gset_false:N \g_hencetherefore
    hence~
  } {
    \bool_gset_true:N \g_hencetherefore
    therefore~
  }
}

\NewDocumentCommand{\Hence}{}{
  \bool_if:nTF { \g_hencetherefore } {
    \bool_gset_false:N \g_hencetherefore
    Hence,~we~obtain~
  } {
    \bool_gset_true:N \g_hencetherefore
    Therefore,~we~obtain~
  }
}

\seq_new:N \g_cflist_loaded
\seq_new:N \g_cflist_pending

\NewDocumentCommand{\cfadd}{ m }
{
  \seq_if_in:NnF \g_cflist_loaded { #1 } {
    \seq_if_in:NnF \g_cflist_pending { #1 } {
      \seq_gput_right:Nn \g_cflist_pending { #1 }
    }
  }
}

\NewDocumentCommand{\cfconsiderloaded}{ m }{
  \seq_gput_right:Nn \g_cflist_loaded {#1}
}

\NewDocumentCommand{\cfremove}{ m }
{
  \seq_gremove_all:Nn \g_cflist_pending { #1 }
}

\NewDocumentCommand{\cfload}{ o }
{
  \seq_if_empty:NTF \g_cflist_pending {\unskip} {
    (cf.\ \cref{\seq_use:Nn \g_cflist_pending {,}})\IfValueTF{#1}{#1~}{\unskip}
    \seq_gconcat:NNN \g_cflist_loaded \g_cflist_loaded \g_cflist_pending
    \seq_gclear:N \g_cflist_pending
  }
}

\NewDocumentCommand{\cfclear} {} {
  \seq_gclear:N \g_cflist_loaded
  \seq_gclear:N \g_cflist_pending
}

\NewDocumentCommand{\cfout}{ o }
{
  \seq_if_empty:NTF \g_cflist_pending {\unskip} {
    (cf.\ \cref{\seq_use:Nn \g_cflist_pending {,}})\IfValueTF{#1}{#1~}{\unskip}
    \seq_gclear:N \g_cflist_pending
  }
}

\NewDocumentCommand{\ifnocf} { m } {
  \seq_if_empty:NT \g_cflist_pending { #1 }
}

\ExplSyntaxOff

\NewDocumentEnvironment{cproof}{m}
{\begin{proof}[Proof of \cref{#1}]}%
{\noindent The proof of \cref{#1} is thus complete.
\end{proof}}

\NewDocumentEnvironment{cproof2}{m}
{\begin{proof}[Proof of \cref{#1}]}%
{\noindent This completes the proof of \cref{#1}.
\end{proof}}

\title{Non-convergence to global minimizers 
for Adam and\\ 
stochastic gradient descent optimization 
and constructions of local minimizers 
in the training of artificial neural networks}

\author{Arnulf Jentzen$^{1,2}$ and Adrian Riekert$^{3}$\medskip\\
\small{$^1$ School of Data Science and Shenzhen Research Institute of Big Data,} \vspace{-0.1cm}\\
\small{The Chinese University of Hong Kong, Shenzhen (CUHK-Shenzhen),}\vspace{-0.1cm}\\
\small{China; e-mail: \texttt{ajentzen}\textcircled{\texttt{a}}\texttt{cuhk.edu.cn}}\smallskip\\
\small{$^2$ Applied Mathematics: Institute for Analysis and Numerics,}\vspace{-0.1cm}\\
\small{University of M\"unster, Germany; e-mail: \texttt{ajentzen}\textcircled{\texttt{a}}\texttt{uni-muenster.de}}\smallskip\\
\small{$^3$ Applied Mathematics: Institute for Analysis and Numerics,}\vspace{-0.1cm}\\
\small{University of M\"unster, Germany; e-mail: \texttt{ariekert}\textcircled{\texttt{a}}\texttt{uni-muenster.de}}}

\date{\today}

\begin{document}

\maketitle

\begin{abstract}
Stochastic gradient descent (SGD) optimization methods such as the plain vanilla SGD 
method and the popular Adam optimizer are nowadays the method of choice 
in the training of artificial neural networks (ANNs).  
Despite the remarkable success of SGD methods in the ANN training in numerical simulations, 
it remains in essentially all practical relevant scenarios an open problem 
to rigorously explain why SGD methods seem to succeed to train ANNs. 
In particular, in most practically relevant supervised learning problems, it seems that SGD methods do with high probability not converge to global minimizers in the optimization landscape of the ANN training problem. 
Nevertheless, it remains an open problem of research to disprove the convergence of SGD methods to global minimizers.
In this work we solve this research problem in the situation of shallow ANNs with the rectified linear unit (ReLU) and related activations with the standard mean square error loss by \emph{disproving} in the training of such ANNs that SGD methods 
(such as the plain vanilla SGD, the momentum SGD, the AdaGrad, 
the RMSprop, and the Adam optimizers) 
can find a global minimizer with high probability. 
Even stronger, we reveal in the training of such ANNs 
that SGD methods do with high probability \emph{fail to converge} 
to global minimizers in the optimization landscape. 
The findings of this work do, however, not disprove that SGD methods 
succeed to train ANNs since they do not exclude the possibility that SGD methods find \emph{good local minimizers} whose risk values are close to the risk values of the global minimizers.
In this context, another key contribution of this work is to establish the existence of a hierarchical structure of \emph{local minimizers} with distinct risk values in the optimization landscape of ANN training problems with ReLU and related activations. 
\end{abstract}


\tableofcontents

\section{Introduction}

Stochastic gradient descent (SGD) optimization methods such as 
the plain vanilla SGD, 
the momentum SGD (cf.\ \cite{POLYAK19641}), 
the RMSprop (cf.\ \cite{HintonSlides}), 
the AdaGrad (cf.\ \cite{MR2825422}), 
and the Adam (cf.\ \cite{KingmaBa2015}) optimizers 
(cf., for example, also the monographs 
\cite{Goodfellow-et-al-2016,JentzenKuckuckvonWurstemberger2023arXiv}
and 
survey articles \cite{Ruder2017overview,Sun2019overview} 
and the references therein) 
are nowadays the method of choice in the training of artificial neural networks (ANNs), 
with overwhelming success in various applications, ranging from language modeling and image processing to game intelligence 
(cf., for instance, \cite{Goodfellow-et-al-2016,LeCun2015,Schmidhuber2015} 
and the references therein).
Despite the remarkable success of SGD optimization 
methods in the ANN training in numerical simulations, 
it remains in essentially all practical relevant scenarios 
an open problem of research to rigorously explain why SGD optimization methods 
seem to succeed to train ANNs.

In particular, in nearly all practically relevant supervised learning problems  
it remains an open problem of research to prove 
or disprove that the considered SGD optimization 
method finds with high probability 
global minimizers in the optimization landscape associated 
to the ANN training problem. 
In fact, in certain so-called \emph{overparameterized} settings 
where there are much more parameters of the ANNs 
than input-output data pairs 
and where the considered ANNs can exactly 
match/represent all input-output data pairs in the learning problem,
there are already such results which prove that 
with high probability the SGD optimization method succeeds 
to globally minimize the objective function (see, e.g., \cite{AllenZhu2019,Arora2019,DuLeeLiWang2019,Du2019,JentzenKroeger2021}). 
Beyond this overparameterized regime 
where the class of considered ANNs cannot exactly 
represent the relationship between input and output data, 
it remained an open problem of research in the training of ANNs 
to prove or disprove that SGD optimization 
methods can find global minimizers with high probability.

In this work we solve this research problem 
in the situation of 
shallow ANNs with ReLU and related activations 
and the standard mean square error loss 
by disproving 
in the training of such ANNs
that SGD optimization methods 
(such as the plain vanilla SGD, the momentum SGD, the AdaGrad, 
the RMSprop, and the Adam optimizer) 
can find a global minimizer with high probability. 
Even stronger, in this work 
(see \cref{prop:sgd:shallow:nonc},
\cref{cor:sgd:shallow:data_loss}, 
\cref{cor:sgd:shallow:multi_dim}, 
and \cref{cor:sgd:shallow:1d} 
in \cref{sec:nonc} below) 
we reveal in the training of such ANNs 
that beyond the overparameterized regime 
we have that SGD methods 
do with high probability \emph{not} converge 
to global minimizers in the optimization landscape.

To simplify the notation, we now illustrate this contribution 
in \cref{thm:sgd:shallow:nonc:multi_d_A_data_and_scientific} below 
in this introductory section 
in the special case where the employed SGD method 
is the plain vanilla SGD method 
and we refer to \cref{sec:nonc} below for our more general results 
covering, for instance, 
the momentum SGD (see \cref{lem:sgd:property}), 
the AdaGrad, 
the RMSprop (see \cref{lem:adam:property} below), and 
the Adam (see \cref{lem:adam:property} below) optimizers.

In \cref{thm:sgd:shallow:nonc:multi_d_A_data_and_scientific} 
the \emph{realization functions} of the considered \emph{fully-connected feedforward ANNs} 
are introduced in \cref{eq:realization_ANNs}, 
the \emph{risk} function (the objective function, the function which we intend to minimize)
is presented in \cref{eq:objective_function}, 
the \emph{empirical risks} for the mini-batches 
used in the SGD steps are provided in \cref{eq:empirical_risk_for_mini_batch}, 
the \emph{gradients} in the SGD steps are described through \cref{eq:gradients_SGD}, 
the \emph{SGD method} is specified in \cref{eq:SGD_method}, 
and the conclusion of 
\cref{thm:sgd:shallow:nonc:multi_d_A_data_and_scientific} 
is given in \cref{eq:conclusion_non_convergence}. 
We now present \cref{thm:sgd:shallow:nonc:multi_d_A_data_and_scientific} 
in all mathematical details and, thereafter, we provide further explanations 
on the mathematical objects appearing in \cref{thm:sgd:shallow:nonc:multi_d_A_data_and_scientific} 
and on the conclusion and the relevance of \cref{thm:sgd:shallow:nonc:multi_d_A_data_and_scientific}.

\begin{samepage}
\begin{theorem}[Non-convergence to minimizers in the training of ANNs]
\label{thm:sgd:shallow:nonc:multi_d_A_data_and_scientific}
Let 
$ 
  d \in \N 
  = \{ 1, 2, \dots \} 
$, 
for every $ \width \in \N $ 
let 
$ 
  \fd_{ \width } 
  = d \width + 2 \width + 1 
$, 
for every $ \width \in \N $, 
$ 
  \theta = ( \theta_1, \dots, \theta_{ \fd_{ \width } } ) 
  \in \R^{ \fd_{ \width } } 
$
let 
$
  \realization{ \theta } \colon \R^d \to \R
$
satisfy for all $ x = ( x_1, \dots, x_d ) \in \R^d $ that 
\begin{equation}
\label{eq:realization_ANNs}
  \realization{\theta}( x ) 
  =
  \theta_{ d \width + 2 \width + 1 } 
  +
  \ssum\limits_{ i = 1 }^{ \width } 
  \theta_{ d \width + \width + i } 
  \max\biggl\{ 
    \theta_{ d \width + i } 
    \allowbreak
    + 
    \sum\limits_{ j = 1 }^d
    \theta_{ (i - 1) d + j } x_j 
    , 0 
  \biggr\}
  ,
\end{equation}
let 
$ a \in \R $, 
$ b \in (a, \infty) $, 
let $ ( \Omega, \cF, \P ) $ be a probability space, 
for every $ m, n \in \N_0 $ 
let 
$ X^m_n \colon \Omega \to [a,b]^d $
and 
$ Y^m_n \colon \Omega \to \R $
be random variables,
for every $ \width \in \N $ 
let
$ \cL^{ \width } \colon \R^{ \fd_{ \width } } \to \R $
satisfy for all 
$ 
  \theta \in \R^{ \fd_{ \width } } 
$
that
\begin{equation}
\label{eq:objective_function}
\textstyle 
  \cL^{ \width }( \theta ) 
  = 
  \E\bigl[
    | 
      \realization{\theta}( X^0_0 )
      -
      Y^0_0
    |^2
  \bigr]
  ,
\end{equation}
let $ f \colon \R^d \to \R $ 
and 
$
  \dens \colon \R^d \to \R 
$
be continuous, 
assume $ \P $-a.s.\ that 
$
  \E[ Y^0_0 | X^0_0 ] = f( X^0_0 )
$, 
assume 
$
  \dens^{ - 1 }( \R \backslash \{ 0 \} ) = ( a, b )^d 
$,
assume for all $ A \in \mathcal{B}( \R^d ) $ that 
$
  \P( X^0_0 \in A ) =
  \int_A \dens( x ) \, \d x
$, 
for every $ \width \in \N $, $ n \in \N_0 $ 
let $ M^{ \width }_n \in \N $, 
let $ \gamma^{ \width }_n \in \R $, 
let 
$ 
  \fL_n^{ \width } \colon \R^{ \fd_{ \width } } \times \Omega \to \R 
$ 
satisfy for all 
$ \theta \in \R^{ \fd_{ \width } } $ 
that
\begin{equation}
\label{eq:empirical_risk_for_mini_batch}
\textstyle
  \fL_n^{ \width }( \theta ) 
  = 
  \frac{ 1 }{ M^{ \width }_n } 
  \biggl[ 
  \sum\limits_{ m = 1 }^{ M^{ \width }_n } 
    \abs{
      \realization{ \theta }( X^m_n ) 
      - 
      Y^m_n  
    }^2
  \biggr]
  ,
\end{equation}
let 
$ 
  \fG_n^{ \width } 
  = ( \fG_n^{ \width, 1 }, \dots, \fG_n^{ \width, \fd_{ \width } } ) 
  \colon \R^{ \fd_{ \width } } \times \Omega \to \R^{ \fd_{ \width } } 
$ 
satisfy for all
$ i \in \{ 1, 2, \dots, \width \} $,
$ j \in ( \cup_{ k = 1 }^d \{ (i - 1 ) d + k \} ) \cup \cu{\width d + i } $, 
$
  \theta = ( \theta_1, \dots, \theta_{ \fd_{ \width } } ) 
  \in 
  \R^{ \fd_{ \width } }
$, 
$ 
  \omega \in 
  \bigl\{ w \in \Omega 
    \colon
        \R^{ d + 1 } \ni ( \psi_1, \dots, \psi_{d+1} ) \mapsto 
        \fL_n^{ \width }(
            \theta_1 , \dots, \theta_{( i - 1 ) d}, 
            \allowbreak 
            \psi_1, \dots, \psi_d,
            \theta_{i d + 1}, 
            \dots, 
            \allowbreak 
            \theta_{d \width + i - 1 }, 
            \allowbreak
            \psi_{d+1}, \allowbreak
            \theta_{d \width + i + 1 } ,
            \dots, 
            \theta_{\fd_{ \width } } ,
            w
        ) \in \R 
        \allowbreak 
        \text{ is differentiable at } 
        ( 
          \theta_{ ( i - 1 ) d + 1 }, 
          \dots, 
          \allowbreak 
          \theta_{ i d }, 
          \allowbreak 
          \theta_{ \width d + i } 
        ) 
    \} 
$
that
\begin{equation}
\label{eq:gradients_SGD}
  \fG^{\width}_j( \theta, \omega ) 
  = 
  \bigl(
    \tfrac{ \partial }{ \partial \theta_j } 
    \fL_n^{ \width }
  \bigr)( \theta, \omega ) 
  ,
\end{equation}
and let 
$ 
  \Theta^{ \width }_n 
  = ( \Theta^{ \width, 1 }_n, \dots, \Theta^{ \width, \fd_{ \width } }_n ) 
  \colon \Omega \to \R^{ \fd_{ \width } } 
$ 
be a random variable, 
assume for all $ \width \in \N $, $ n \in \N_0 $ that 
\begin{equation}
\label{eq:SGD_method}
  \Theta^{ \width }_{ n + 1 } 
  = \Theta^{ \width }_n - \gamma^{ \width }_n \fG^{ \width }_n( \Theta^{ \width }_n ) 
  ,
\end{equation}
assume for all $ \width \in \N $ 
that 
$
  \Theta^{ \width, i }_0 
$, 
$ i \in \{ 1, 2, \dots, \width d + \width \} $, 
are independent, 
let $ \Dens \colon \R \to [0,\infty) $ be measurable, 
let $ \eta \in (0, \infty) $ 
satisfy 
$
  \Dens( ( -\eta, \eta ) ) \subseteq (0, \infty) 
$,
and let $ \kappa \in \R $ satisfy for all 
$
  \width \in \N
$,
$
  i \in \{ 1, 2, \dots, \width d + \width \}
$,
$ x \in \R $ 
that
$
  \P( 
    \width^{ \kappa }
    \Theta^{ \width, i }_0 
    < x
  )
  =
  \int_{ - \infty }^x \Dens(y) \, \d y
$
and 
$ 
  \E[ | f( X^0_0 ) - Y^0_0 |^2 ]
  \notin 
  \cL^{ \width }( \R^{ \fd_{ \width } } )
$.
Then
\begin{equation} 
\label{eq:conclusion_non_convergence}
  \liminf_{
    \width \to \infty 
  } 
  \P\biggl(
    \liminf_{ n \to \infty } 
    \cL^{ \width }( \Theta_n^{ \width } ) 
    > 
    \inf_{ \theta \in \R^{ \fd_{ \width } } }
    \cL^{ \width }( \theta )
  \biggr) 
  = 1 .
\end{equation}
\end{theorem}
\end{samepage}

\cref{thm:sgd:shallow:nonc:multi_d_A_data_and_scientific} follows from \cref{cor:sgd:shallow:multi_dim} and \cref{lem:sgd:property} in \cref{sec:nonc} below. 
\cref{cor:sgd:shallow:multi_dim} establishes 
with high probability the non-convergence to global minima 
also for more general gradient-based optimization methods including, in particular,
the momentum SGD method, 
the AdaGrad method, the RMSprop method, and the popular Adam optimizer 
as we show in detail in \cref{lem:adam:property,lem:sgd:property} below.
However, for the sake of simplicity we restrict ourselves in \cref{thm:sgd:shallow:nonc:multi_d_A_data_and_scientific} above 
to the standard SGD method.

In \cref{eq:conclusion_non_convergence} 
in \cref{thm:sgd:shallow:nonc:multi_d_A_data_and_scientific} 
we show that the probability that 
the limit of the risk of 
the SGD process 
$ ( \Theta_n^{ \width } )_{ n \in \N_0 } $ 
is strictly larger than 
the optimal risk value 
$
  \inf_{ \theta \in \R^{ \fd_{ \width } } }
  \cL^{ \width }( \theta )
$
converges to one 
as the number of neurons on the hidden layer 
$ \width \in \N $ 
(the width of the ANN) 
converges  
to infinity. 
In particular, \cref{eq:conclusion_non_convergence} 
reveals that the probability that
the limit of the risk of 
the SGD process 
$ ( \Theta_n^{ \width } )_{ n \in \N_0 } $ 
reaches the risk value 
$
  \inf_{ \theta \in \R^{ \fd_{ \width } } }
  \cL^{ \width }( \theta )
$
of a \emph{global minimizer} vanishes 
as the number of neurons on the hidden layer 
$ \width $
converges to infinity 
(see \cref{prop:global:exist} in \cref{subsec:global_min}) 
and, as a consequence of this, 
we have that the probability that the SGD process 
$ ( \Theta_n^{ \width } )_{ n \in \N_0 } $ 
converges to a global minimizer of the risk landscape 
must converge to zero as the number of neurons 
on the hidden layer $ \width $ increases to infinity.

We now comment on the conditions we require in the main theorem.
The assumption that $ 
\E[ | f( X^0_0 ) - Y^0_0 |^2 ]
\notin 
\cL^{ \width }( \R^{ \fd_{ \width } } )
$ in \cref{thm:sgd:shallow:nonc:multi_d_A_data_and_scientific} essentially means that there does not exist a parameter vector $\theta \in \R^{ \fd_\width }$ which satisfies $f(X_0^0 ) = \realization{\theta} ( X_0^0 )$ $ \P $-a.s.~(see also \cref{lem:l2:decomp} in \cref{sec:nonc} below). 
This excludes the overparameterized regime, but holds for 
many practical (scientific) machine learning problems where the distribution of the input data $X_0^0$ is absolutely continuous and the target function $f$ is not piecewise linear so that it cannot be represented exactly by a ReLU ANN.
Furthermore, the assumptions on the distribution of the initial ANN parameters 
$
  \Theta_0^{ \width , i }
$, 
$ i \in \{ 1, 2, \dots, \width d + \width \} $, 
are satisfied by many practically used initialization schemes where the parameters of the first layer are i.i.d.~with a certain scaled absolutely continuous distribution and the scaling factor depends on the number of neurons $\width$, for instance the standard Kaiming He initialization used in the machine learning library \textsc{PyTorch}.

We point out that the result in 
\cref{thm:sgd:shallow:nonc:multi_d_A_data_and_scientific} 
and its extensions 
in \cref{sec:nonc} below 
do not prove that SGD methods 
fail to train ANNs. 
\cref{thm:sgd:shallow:nonc:multi_d_A_data_and_scientific} 
and its extensions in \cref{sec:nonc} do just prove 
that ANNs are not trained through reaching \emph{global minimizers} 
but it does not exclude the possibility that 
SGD optimization methods find good \emph{local minimizers} 
in the optimization landscape 
which are close to global minimizers 
in the sense that risk values of such good local minimizers 
converge to the optimal risk value as the number of neurons 
on the hidden layer $ \width $ (the width of the ANN) 
converges to infinity. We also refer to the works \cite{GentileWelper2022,IbragimovJentzenRiekert2022arXiv,Welper2023arXiv} 
in which precisely such statements -- overall convergence through \emph{convergence to good local minimizers} -- 
have been rigorously verified in simplified ANN training situations 
(such as, for example, one-dimensional input of the ANN 
and training of only bias parameters instead of weight 
and bias parameters). 
Moreover, another key contribution of this work is to establish the existence 
of a \emph{hierarchical structure of local minimizers} with distinct risk values 
in the optimization landscape of ANN training problems with ReLU and related activations; 
see 
\cref{cor:different:risk:discr} and \cref{cor:different:risk:examples} in 
\cref{sec:different_risk_levels_of_local_minima}
and 
\cref{cor:diff:risk:relu} in 
\cref{sec:different_risk_levels_of_local_minima_for_ReLU_ANNs} 
below. 
Further results related to our findings in \cref{sec:different_risk_levels_of_local_minima_for_ReLU_ANNs} which show existence of non-global local minimizers in the ANN optimization landscape can be found, e.g., in \cite{ChristofKowalczyk2023,He2020,Liu2022} and the references mentioned therein.

We now provide a short literature review on research findings 
related to \cref{thm:sgd:shallow:nonc:multi_d_A_data_and_scientific} above 
and the mathematical analysis 
of gradient based optimization methods 
in general, respectively. 
For convex objective functions, classical results show convergence of gradient methods to a global minimizer; see, for example, \cite{BachMoulines2013,JentzenKuckuckNeufeld2021,BachMoulines2011,Nesterov2004}.
A weaker condition than convexity which nevertheless implies convergence to a global minimum is the Polyak-\L ojasiewicz inequality, and abstract convergence results for SGD methods under this assumption have been established, 
for instance, in \cite{Karimi2016linear,LeiHuLiTang2020,Polyak1963}.
Without such assumptions, it is still possible to show convergence of the gradients of the iterates to zero, which has been carried out, for example, in \cite{BertsekasTsitsiklis2000,GhadimiLan2013,ReddiHefnySra2016} for SGD and in \cite{DefossezBach2020} for Adam and AdaGrad.

In the training of ANNs, the risk landscape is typically highly non-convex and global convergence to a global minimizer is therefore -- except for the overparameterized regime -- usually not expected.
Indeed, in the situation of ReLU ANNs there are several results which show that the probability of reaching a global minimizer is \emph{strictly smaller than one} or, in other words, which show that the probability of not converging to a global minimizer is \emph{strictly larger than zero} 
(cf., for instance, \cite{CheriditoJentzenRossmannek2020,LuShinSu2020,ShinKarniadakis2020}). 
The contribution of \cref{eq:conclusion_non_convergence} in \cref{thm:sgd:shallow:nonc:multi_d_A_data_and_scientific} 
and its extensions in \cref{sec:nonc} below is, however, substantially different 
as \cref{eq:conclusion_non_convergence} does not only imply that the probability of 
not converging to a global minimizer is strictly bigger than zero but in fact shows 
that this probability converges \emph{to one} as the number of neurons on the hidden layer 
$ \width $ (the width of the ANN) goes to infinity. 
Nevertheless, in some cases it is possible to show \emph{local} convergence of SGD methods towards global minimizers; see, for instance, \cite{Chatterjee2022,FehrmanGessJentzen2020,JentzenRiekert2022_JML,JentzenRiekert2022piecewise}.
Additional results show convergence of gradient methods in the training of ANNs to a \emph{good critical point}, i.e., with a risk value only slightly larger than the global minimum; as mentioned above \cite{GentileWelper2022,IbragimovJentzenRiekert2022arXiv,Welper2023arXiv}.

The remainder of this work is organized as follows. 
In \cref{sec:properties_generalized_critical_points} 
we establish some properties for appropriately generalized critical points 
in the optimization risk landscape associated to the training 
of possibly deep ANNs with an arbitrarily large number of hidden layers. 
Specifically, in \cref{prop:clarke:crit} 
in \cref{sec:properties_generalized_critical_points} 
we establish in the case of a general Lipschitz continuous activation function 
a general upper bound for the risk value of all such generalized critical points, that is, 
we show that the risk value of each such a generalized critical point 
is lower or equal than the smallest risk value of all parameter vectors 
with constant ANN realization functions.
Additionally, in \cref{prop:local:min:ext} we show
in the case of an activation function that is constant on some interval 
(such as the ReLU and the clipping activation) 
that for every ANN architecture and every local minimizer 
of the considered risk function we have that there exist local minimizers 
with the same ANN realization function in all larger ANN architectures. 

In \cref{sec:locmin} we establish several key properties 
for local and global minimizers in the optimization risk landscape associated 
to the training of shallow ANNs. 
In particular, in \cref{cor:different:risk:discr} in \cref{sec:locmin} 
we establish a lower bound on the number of distinct values 
of the risks of local minimum points in the optimization landscape 
and in \cref{prop:all:neurons:active} in \cref{sec:locmin} 
we show that every parameter vector with a risk sufficiently close 
to the minimal risk cannot have any \emph{inactive} neurons in a certain sense
(see \cref{sec:setting_section} for our precise definition of the set of inactive neurons). 
In \cref{sec:nonc} we employ the findings from \cref{sec:locmin} 
to show under suitable assumptions that the risk value of gradient-based optimization methods in the training of shallow ANNs does with high probability not converge 
to the risk value of global minimizers and, thereby, we 
prove \cref{thm:sgd:shallow:nonc:multi_d_A_data_and_scientific} 
in this introductory section.

Roughly speaking, in \cref{sec:locmin} we studied 
the optimization landscapes associated to ANN training problems 
in the regime where the risk value of the parameter vector 
is lower or equal than the smallest risk value of 
all parameter vectors with constant ANN realization functions. 
To complement these findings, 
in the final section of this work 
(see \cref{sec:a_priori_bounds} below) 
we study gradient-based optimization procedures in the training of ANNs 
in the regime where the risk value of the parameter vector 
is strictly larger than the smallest risk value of 
all parameter vectors with constant ANN realization functions. 
The findings in \cref{sec:a_priori_bounds} apply to 
deep ReLU ANNs with an arbitrary number of hidden layers.

\section{Properties of generalized critical points of the risk function}
\label{sec:properties_generalized_critical_points}

In this section we consider ANNs with an arbitrarily large number of hidden layers 
and a general activation function. We will prove some useful statements about 
the appearance of critical points in the corresponding risk landscape.
In particular, in \cref{prop:clarke:crit} we establish an upper estimate on the risk of critical points in the sense of the Clarke subddifferential; a notion we formally introduce in \cref{def:clarke:subdiff}.
In \cref{prop:local:min:ext} we show that, roughly speaking, for every ANN architecture and every local minimum point of the considered risk function there exist local minimum points with the same ANN realization function in all larger ANN architectures.
Similar results regarding the embedding of critical points in the risk landscape of deep ANNs have been obtained by Zhang et al.~\cite{ZhangLi2022embedding,zhang2021embedding}.

\subsection{Description of deep artificial neural networks (ANNs) with general activations}

We first introduce in \cref{setting:dnn:gen} our notation for deep ANNs with variable architectures.

\begin{setting}
	\label{setting:dnn:gen}
	Let $ d_I, d_O \in \N $,
	$ a \in \R $, $ b \in (a, \infty) $,
	for every 
	$L \in \N$,
	$\ell = (\ell_0, \ldots, \ell_L ) \in \N^{ L + 1 }$
	define
	$ \fd _\ell = \sum_{ k = 1 }^L \ell_k( \ell_{k-1} + 1) $,
	for\footnote{For every $ n \in \N $ 
	we denote by $ \num{n} \subseteq \N $ 
	the set given by $ \num{n} = \cu{1, 2, \ldots, n } $.} 
	every $L \in \N$,
	$\ell = (\ell_0, \ldots, \ell_L ) \in \N^{ L + 1 }$,
	$k \in \num{L}$,
	$ \theta = ( \theta_1, \ldots, \theta_{ \fd _\ell}) \in \R^{ \fd_\ell } $
	let $ \fw^{ \ell, k, \theta } = ( \fw^{\ell , k, \theta }_{ i,j} )_{ (i,j) \in \num{ \ell_k } \times \num{ \ell_{k-1} } }
	\in \R^{ \ell_k \times \ell_{k-1} } $
	and $ \fb^{\ell , k, \theta } = ( \fb^{ \ell , k, \theta }_i )_{i \in \num{\ell_k } }
	\in \R^{ \ell_k} $
	satisfy for all
	$ i \in \num{ \ell_k } $,
	$ j \in \num{ \ell_{k-1} } $ that
	\begin{equation}
	\fw^{ \ell , k, \theta }_{ i, j } 
	=
	\theta_{ (i-1)\ell_{k-1} + j+\sum_{h= 1 }^{ k-1} \ell_h( \ell_{h-1} + 1)}
	\qqandqq
	\fb^{ \ell , k, \theta }_i =
	\theta_{\ell_k\ell_{k-1} + i+\sum_{h= 1 }^{ k-1} \ell_h( \ell_{h-1} + 1)} 
	\, ,
	\end{equation}
	for every 
	$L \in \N$,
	$\ell = (\ell_0, \ldots, \ell_L ) \in \N^{ L + 1 }$,
	$ k \in \num{ L } $, $ \theta \in \R^{ \fd_\ell } $
	let $ \cA_k^{\ell , \theta } = ( \cA_{k,1}^{\ell , \theta } , \ldots, \cA_{k, \ell_k}^{\ell , \theta } )
	\colon \R^{ \ell_{k-1} } \to \R^{ \ell_k} $
	satisfy 
	for all $ x \in \R^{ \ell_{ k - 1 } } $ that 
	$ 
	\cA_k^{ \ell , \theta }( x ) 
	=
	\fb^{\ell ,  k, \theta } + \fw^{\ell ,  k, \theta } x 
	$,
	let $\sigma  \colon \R \to \R$ be locally bounded and measurable,
	let 
	$
	\mathfrak{M}
	\colon( \cup_{ n \in \N } \R^n ) \to ( \cup_{ n \in \N } \R^n )
	$
	satisfy for all $ n \in \N $,
	$ x = (x_1, \ldots, x_n ) \in \R^n $
	that 
	\begin{equation}
	\mathfrak{M} ( x ) = ( \sigma  ( x_1) , \ldots, \sigma  ( x_n ) ),
	\end{equation}
	for every 
	$L \in \N$,
	$\ell = (\ell_0, \ldots, \ell_L ) \in \N^{ L + 1 }$,
	$ \theta \in \R^{ \fd_\ell } $
	let $ \realization{\ell , k, \theta } = ( \realization{\ell , k, \theta }_{1}, \ldots, \realization{ \ell , k, \theta }_{\ell_k} )
	\colon \R^{ \ell_0 } \to \R^{ \ell_k} $, $ k \in \num{L}$,
	satisfy 
	for all $ k \in \num{L-1} $,
	 $ x \in \R^{ \ell_0 } $
	that
	\begin{equation}
		\label{setting:dnn:eq_realization}
	\begin{split}
	&
	\realization{\ell , 1, \theta } (x) = \cA^{\ell , \theta} _1 (x) \qqandqq
	\realization{ \ell , k + 1, \theta } ( x ) 
	=
	\cA_{ k + 1}^{ \ell , \theta }(
	\mathfrak{M} (
	\realization{\ell , k, \theta } ( x )
	)
	),
	\end{split}
	\end{equation}
	let $ \mu \colon \mathcal{B} ( [a,b]^{ d_I } ) \to [0, \infty] $ be a finite measure,
	let $ f = (f_1, \ldots, f_{\ell_L}) \colon [a,b]^{ d_I } \to \R^{ d_O } $ be measurable,
	and
	for every 
	$L \in \N$,
	$\ell = (\ell_0, \ldots, \ell_L ) \in \N^{ L + 1 }$
	with $\ell_0 = d_I$ and $\ell_L = d_O$
	let $ \cL_\ell  \colon \R^{ \fd_\ell }\to \R $ 
	satisfy\footnote{We denote by  $\norm{\cdot}\colon(\cup_{n\in\N}\R^{n})\to\R$ and 
		$\spro{\cdot , \cdot } \colon(\cup_{n\in\N}(\R^n\times\R^n))\to\R$ the functions which
		satisfy for all $n\in\N$,
		$x=(x_1,\ldots,x_n)$, $y=(y_1,\ldots,y_n)\in\R^n$
		that $\norm{x}=(\sum_{i=1}^n|x_{i}|^2)^{1/2}$
		and $\spro{x , y } =\sum_{i=1}^n x_i y_i$.} 
    for all $ \theta \in \R^{ \fd_\ell } $
	that
	\begin{equation}
	\cL_\ell ( \theta)=\int_{[a,b]^{ d_I } }
	\norm{ \realization{\ell , L, \theta }(x)-f(x) }^2\, \mu( \d x).
	\end{equation}
\end{setting}

In \cref{setting:dnn:gen} we think for every $L \in \N$, $\ell = (\ell_0 , \ldots, \ell_L ) \in \N^{L+1}$ of an ANN architecture with $L$ layers,
input dimension $\ell_0$,
hidden layers of dimensions $\ell_1, \ell_2, \ldots, \ell_{L-1}$,
and output dimension $\ell_L$.
The locally bounded and measurable function $\sigma$ is the considered ANN activation function,
the function $f \colon [a,b]^{d_I} \to \R^{d_O}$ is the target function which we intend to approximate with respect to a measure $ \mu \colon \mathcal{B} ( [a,b]^{ d_I } ) \to [0, \infty] $ describing an unnormalized input distribution,
and the function $\cL_\ell$ is the corresponding $L^2$-risk function that one seeks to minimize.

\subsection{Clarke subdifferentials}
\label{subsec:clarke_subdiff}

In this part we introduce several notions for subdifferentials, which can be viewed as generalizations of the concept of gradients for not necessarily differentiable functions.
For references on this topic we refer, e.g., to Clarke~\cite{Clarke1975}, 
Clarke et al.~\cite{ClarkeBook1998},
and Rockafellar \& Wets~\cite{RockafellarWetsBook}.
These notions are needed since we also allow non-differentiable activation functions.

\begin{definition} [Clarke subdifferential] \label{def:clarke:subdiff}
Let $n \in \N$, $x \in \R^ n$,
let $f \colon \R^n \to \R$ be locally Lipschitz continuous,
and let $\Omega = \cu{y \in \R^n \colon f \text{ is differentiable at } y }$.
Then the Clarke subdifferential $\partial_C f ( x ) \subseteq \R^n$ is defined as the set given by\footnote{For a set $A \subseteq \R^n$ we denote by $\operatorname{conv} A $ the convex hull of $A$.}
\begin{equation} \label{def:clarke:subdiff:eq}
    ( \partial_C f ) ( x ) = \operatorname{conv} \cu[\big]{y \in \R^n \colon \br[\big]{ \exists \, (z_k )_{k \in \N} \subseteq \Omega \colon \rbr[\big]{ \lim\nolimits_{k \to \infty} z_k = x, \, \lim\nolimits_{k \to \infty} \nabla f ( z_k ) = y } } }.
\end{equation}
\end{definition}

\begin{remark}
There is a more general definition of the Clarke subdifferential for lower semi-continuous functions in terms of the normal cone of the epigraph of $f$; see \cite[Definition 3.18]{Clarke1975}. 
\end{remark}

\begin{definition}[cf.~{\cite[Definition 8.3]{RockafellarWetsBook}} and {\cite[Definition 2.10]{BolteDaniilidis2006}}]
\label{def:subdifferential}
Let $n \in \N$, $f \in C(\R^n, \R)$, $x \in \R^n$.
\begin{enumerate} [ label = (\roman*)]
    \item The Fr\'{e}chet subdifferential $( \hat{\partial} f ) (x) \subseteq \R^n$ is the set given by
    \begin{equation}
    	\label{def:subdifferential:eq_frechet}
        ( \hat{\partial} f ) ( x ) = \cu*{ y \in \R^n \colon \liminf_{\R^n \backslash \cu{  0 } \ni h \to 0 } \frac{f(x + h ) - f ( x ) - \spro{y , h } }{\norm{h}} \geq 0  }.
    \end{equation}
    \item The limiting Fr\'{e}chet subdifferential $( \partial f ) (x) \subseteq \R^n $ is the set given by
    \begin{multline}
        ( \partial f ) ( x ) = \big\{ v \in \R^n \colon \big[ \exists \, (y_k)_{k \in \N} ,(z_k )_{k \in \N} \subseteq \R^n \colon \\
        \rbr[\big]{ \lim\nolimits_{k \to \infty} z_k = x , \, \lim\nolimits_{k \to \infty} y_k = v , \, \forall \, k \in \N \colon y_k \in \hat{\partial} f ( z_k ) } \big] \big\}.
    \end{multline}
\end{enumerate}
\end{definition}

\begin{remark}
	\label{remark:subdiff:inclusion}
If $f \colon \R^n \to \R$ is locally Lipschitz continuous and $x \in \R^n$ then one has $( \hat{\partial} f ) ( x ) \subseteq ( \partial f ) ( x ) \subseteq ( \partial_C f ) ( x ) $; see \cite[Theorem 8.9]{RockafellarWetsBook} and \cite[Proposition 3.17]{Clarke1975}. Furthermore, $( \partial_C f ) ( x )$ is convex, compact, and non-empty.
\end{remark}

\begin{remark} \label{rem:local:min}
If $f \colon \R ^n \to \R$ is locally Lipschitz continuous and $x \in \R^n$ is a local minimum of $f$ then by \cref{def:subdifferential:eq_frechet} and \cref{remark:subdiff:inclusion} one clearly has $0 \in (\hat{\partial} f ) ( x ) \subseteq (\partial f ) ( x ) \subseteq (\partial_C f ) ( x ) $.
\end{remark}

\cfclear
\begin{lemma} \label{lem:clarke:diff}
     Let $n \in \N $, $i \in \num{n}$,
     let $f \colon \R^n \to \R$ be locally Lipschitz continuous,
     assume for all $x = (x_1, \ldots, x_n) \in \R^n$ that $f(x_1, \ldots, x_{i-1}, (\cdot) , x_{i+1}, \ldots, x_n )$ is differentiable at $x_i$,
     assume that $\R^n \ni x \mapsto (\frac{\partial}{\partial x_i } f ) ( x) \in \R$ is continuous,
     and let $x \in \R^n$.
     Then it holds for all $y = (y_1, \ldots, y_n) \in ( \partial _C f ) ( x )$ that 
     $y_i = ( \frac{\partial}{\partial x_i } f ) ( x)$ \cfadd{def:clarke:subdiff}\cfload.
\end{lemma}
\begin{cproof}{lem:clarke:diff}
Throughout this proof let $\Omega = \cu{y \in \R^n \colon f \text{ is differentiable at } y }$
and let $z = (z_k )_{k \in \N} \subseteq \Omega$ satisfy $\lim_{k \to \infty} z_k = x$.
\Nobs that the assumption that
for all $y = (y_1, \ldots, y_n) \in \R^n$ it holds that $f(y_1, \ldots, y_{i-1}, (\cdot) , y_{i+1}, \ldots, y_n )$ is differentiable at $y_i$ and the assumption that $\R^n \ni y \mapsto (\frac{\partial}{\partial x_i } f ) ( y ) \in \R$ is continuous demonstrate that $\lim_{k \to \infty} (\frac{\partial}{\partial x_i } f ) ( z_k ) = (\frac{\partial}{\partial x_i } f ) ( x)$.
Combining this with \cref{def:clarke:subdiff:eq} shows for all $y = (y_1, \ldots, y_n) \in ( \partial _C f ) ( x )$ that 
     $y_i = ( \frac{\partial}{\partial x_i } f ) ( x)$.
\end{cproof}

\subsection{Upper bound on the risk level of Clarke critical points}

We next state in \cref{lem:convex:global:min} the elementary fact that every critical point of a convex function is a global minimum.
\cref{lem:convex:global:min} is well-known; see, e.g., \cite[Section 4.2.3]{Boyd2004}.
The proof is only included for completeness.

\begin{lemma} \label{lem:convex:global:min}
     Let $n \in \N$,
      let $f \colon \R^n \to \R$ be convex,
      let $x \in \R^n$,
      and assume that $f$ is differentiable at $x$ with $( \nabla f ) ( x ) = 0$.
      Then $f(x) = \inf_{y \in \R^n } f(y)$.
\end{lemma}

\begin{cproof}{lem:convex:global:min}
	\Nobs that the fact that $f$ is convex implies for all $y \in \R^n $, $t \in [0 , 1 ]$ that
	\begin{equation}
		f(x + t(y-x) ) = f(ty + (1-t)x) \le tf(y) + (1-t)f(x).
	\end{equation}
Hence, we obtain for all $y \in \R^n $, $t \in (0 , 1 ]$ that
\begin{equation}
	f(y) \ge f(x) + \frac{f ( x + t ( y - x ) ) - f ( x ) } {t}.
\end{equation}
Furthermore, the assumption that $f$ is differentiable at $x$ with $(\nabla f ) ( x ) = 0$ ensures that $\lim_{t \searrow 0 } \frac{f ( x + t ( y - x ) ) - f ( x ) } {t} = 0$.
Therefore, we obtain for all $y \in \R^n$ that $f(y) \ge f(x)$.
\end{cproof}

We next show in \cref{prop:clarke:crit} one of the main results of this section: 
In the deep ANN framework in \cref{setting:dnn:gen}, the risk of every Clarke critical point is bounded from above by the error of the best constant approximation of the considered target function.
As one can see from \cref{remark:subdiff:inclusion}, the Clarke subdifferential is a very general notion including many other definitions of generalized gradients, and hence this result gives a meaningful characterization of generalized critical points of the risk function.
By remark \cref{rem:local:min}, the Clarke critical points include in particular all local minima of the risk function.

To apply the results from \cref{subsec:clarke_subdiff} we first verify in \cref{lem:risk:lip}
that the considered risk function is locally Lipschitz continuous, provided that the activation function $\sigma$ admits this property.
In \cite[Lemma 2.10]{HutzenthalerJentzenPohlRiekertScarpa2021} this statement has been established for ANNs with the ReLU activation function, the proof for the general case is very similar.

\begin{lemma} \label{lem:risk:lip}
		Assume \cref{setting:dnn:gen},
	assume that $\sigma$ is locally Lipschitz continuous,
	and let $L \in \N$,
	$\ell = (\ell_0, \ldots, \ell_L ) \in \N^{ L + 1 }$
	satisfy $\ell_0 = d_I$ and $\ell_L = d_O$.
	Then $\cL_\ell$ is locally Lipschitz continuous.
\end{lemma}

\begin{cproof}{lem:risk:lip}
	Throughout this proof let $K \subseteq \R^{ \fd_\ell}$ be compact.
	\Nobs that the assumption that $\sigma$ is locally Lipschitz continuous and \cref{setting:dnn:eq_realization} ensure that $(\R^{ \fd_\ell} \times \R^{ d_I} \ni (\theta , x ) \mapsto \realization{\ell, L, \theta} ( x ) \in \R^{ d_O} )$ is locally Lipschitz continuous.
	In particular, there exists $c \in \R$ which satisfies for all $\theta , \vartheta \in K$, $x \in [a , b ] ^{ d_I }$ that
	\begin{equation}
		\norm{ \realization{\ell, L, \theta} ( x ) - \realization{\ell, L, \vartheta} ( x ) } \le c \norm{\theta - \vartheta }
		\qqandqq
		\norm{\realization{\ell , L , \theta} ( x ) } \le c .
	\end{equation}
Combining this with the triangle inequality and the Cauchy-Schwarz inequality demonstrates for all $\theta , \vartheta \in K$ that
\begin{equation}
	\label{lem:risk:lip:eq_main}
	\begin{split}
		\abs{ \cL_\ell ( \theta ) - \cL_\ell ( \vartheta) }
		& \le \int_{[a,b] ^{ d_I } } \abs[\big]{ \norm{\realization{\ell , L , \theta } ( x ) - f ( x ) } ^2 - \norm{\realization{\ell , L , \vartheta } ( x ) - f ( x ) } ^2 } \, \mu ( \d x ) \\
		&= \int_{[a,b] ^{ d_I } } \abs[\big]{\spro[\big]{\realization{\ell , L , \theta } ( x ) - \realization{\ell , L , \vartheta } ( x ) ,
			\realization{\ell , L , \theta } ( x ) + \realization{\ell , L , \vartheta } ( x ) - 2 f ( x ) } } \, \mu ( \d x ) \\
		& \le \int_{[a,b] ^{ d_I } }
			\norm[\big]{\realization{\ell , L , \theta } ( x ) - \realization{\ell , L , \vartheta } ( x ) }
			\norm[\big]{\realization{\ell , L , \theta } ( x ) + \realization{\ell , L , \vartheta } ( x ) - 2 f ( x ) } \, \mu ( \d x ) \\
		& \le c \norm{\theta - \vartheta } \int_{[a,b] ^{ d_I } } \rbr*{2 c + 2 \norm{ f ( x ) } } \, \mu ( \d x ).
	\end{split}
\end{equation}
Moreover, the fact that $\mu ( [ a , b ] ^{ d_I } ) < \infty $
and the fact that $\int_{[a,b] ^{ d_I } } \norm{f ( x ) } ^2 \, \mu ( \d x ) < \infty $
ensure that $\int_{[a,b] ^{ d_I } } \rbr*{2 c + 2 \norm{ f ( x ) } } \, \mu ( \d x ) < \infty $.
\cref{lem:risk:lip:eq_main} therefore shows that $\cL_\ell | _K$ is Lipschitz continuous.
\end{cproof}

\cfclear
\begin{prop} \label{prop:clarke:crit}
	Assume \cref{setting:dnn:gen},
	assume that $\sigma$ is locally Lipschitz continuous,
	let $L \in \N$,
	$\ell = (\ell_0, \ldots, \ell_L ) \in \N^{ L + 1 }$
	satisfy $\ell_0 = d_I$ and $\ell_L = d_O$,
	and let $\theta \in \R^{\fd_\ell }$ satisfy $0 \in (\partial_C \cL_\ell ) ( \theta )$ \cfadd{def:clarke:subdiff}\cfload.
	Then
	\begin{equation}
	\cL_\ell ( \theta ) \le \inf\nolimits_{\xi \in \R^{d_O } } \int_{[a,b]^{ d_I } } \norm{f ( x ) - \xi } ^2 \, \mu ( \d x ) .
	\end{equation}
\end{prop}

\begin{cproof}{prop:clarke:crit}
	Throughout this proof let $\bfd \in \N_0$ satisfy $\bfd = \sum_{k=1}^{L - 1 } \ell_k( \ell_{k-1} + 1 )$.
	\Nobs that \cref{lem:risk:lip} ensures that $\cL_\ell$ is locally Lipschitz continuous.
	In addition, the fact that for all $\vartheta \in \R^{\fd_\ell}$ it holds that
	\begin{equation}
		\cL _ \ell ( \vartheta ) = \int_{ [ a , b ] ^{ d_I} } \norm{ \cA_L^{ \ell , \vartheta } ( \fM ( \realization{\ell , L - 1 , \vartheta} ( x ) ) ) - f ( x )  } ^2 \, \mu ( \d x )
	\end{equation} implies for all 
	$\vartheta \in \R^\fd$,
	$i \in \N \cap ( \bfd , \fd _\ell ]$
	that
	$\cL_\ell ( \vartheta_1, \ldots, \vartheta_{i-1} , (\cdot), \vartheta_{i+1}, \ldots , \vartheta_\fd )$
	is differentiable at $\vartheta_i$ and that $\R^\fd \ni \vartheta \mapsto ( \frac{\partial}{\partial \theta_i } \cL_\ell ) ( \vartheta ) \in \R$ is continuous.
	Combining this with \cref{lem:clarke:diff} and
	the assumption that
	$0 \in (\partial_C \cL ) ( \theta )$
	assures for all
	$i \in \N \cap ( \bfd , \fd_\ell ]$
	that
	$( \frac{\partial}{\partial \theta_i } \cL_\ell ) ( \theta ) = 0$.
	Furthermore, \nobs that the same argument as in \cite[Proposition 2.15]{HutzenthalerJentzenPohlRiekertScarpa2021} ensures that
	$\R^{\fd_\ell  - \bfd } \ni ( \psi_1, \ldots, \psi_{\fd _\ell - \bfd } ) \mapsto \cL_\ell  ( \theta_1, \ldots, \theta_\bfd, \psi_1, \ldots, \psi_{\fd _\ell - \bfd } ) \in \R$ is convex.
	\cref{lem:convex:global:min} therefore shows that
	\begin{equation}
	\begin{split}
	\cL_\ell( \theta ) &= \inf\nolimits_{\psi \in \R^{\fd _\ell- \bfd } } \cL_\ell ( \theta_1, \ldots, \theta_\bfd, \psi_1, \ldots, \psi_{\fd _\ell - \bfd } ) \\
	& \le \inf\nolimits_{\xi \in \R^{d_O } } \cL_\ell ( \theta_1, \ldots, \theta_\bfd, \underbrace{0, \ldots, 0 ,}_{\ell_L\ell_{L-1}} \xi_1, \ldots, \xi_{d_O} ) \\
	&= \inf\nolimits_{\xi \in \R^{d_O} }
	\int_{[a,b]^{ d_I } } \norm{f ( x ) - \xi } ^2 \, \mu ( \d x ) .
	\end{split}
	\end{equation}
\end{cproof}

\subsection{Embeddings of local minima for deep ANNs}

We now show that local minima can always be embedded into larger ANN architectures 
while preserving the ANN realization function. We first establish this result in 
\cref{lem:local:min:ext} for the case that only one hidden layer dimension is increased by 1 
and then apply induction to show in \cref{prop:local:min:ext} the general case.
A key assumption in our results is that the activation function $\sigma$ is constant on some open interval $(\alpha , \beta )$, which is in particular satisfied by the popular ReLU activation function.

\begin{lemma} \label{lem:local:min:ext} 
	Assume \cref{setting:dnn:gen}, 
	let $ \alpha \in \R $, $ \beta \in ( \alpha, \infty ) $ satisfy 
	for all $ x \in ( \alpha, \beta ) $ that $ \sigma( x ) = \sigma( \alpha ) $,
	let $L \in \N$,
	$\ell = (\ell_0, \ldots, \ell_L ) $, $\bfl = (\bfl_0, \ldots, \bfl_L ) \in \N^{ L + 1 }$
	satisfy $\ell_0 = \bfl_0 = d_I$ and $\ell_L = \bfl_L = d_O$,
	let $m \in \num{L - 1 }$ satisfy $\bfl_m = \ell_m + 1 $,
	assume for all $k \in \num{L - 1 } \backslash \cu{m }$
	that $\ell_k = \bfl_k$,
	and let $ \theta \in \R^{ \fd_{ \ell } } $ 
	be a local minimum point of $ \cL_{ \ell } $. 
	Then there exists $ \vartheta \in \R^{ \fd_{ \bfl } } $ 
	such that
	\begin{enumerate}[label=(\roman*)]
		\item \label{lem:local:min:ext:item1} 
		it holds for all $x \in \R ^{ d_I }$
		that 
		$ 
		\realization{ \bfl , L , \vartheta  } ( x )
		= \realization{\ell , L , \theta }
		$ 
		and 
		\item \label{lem:local:min:ext:item2} 
		it holds that 
		$ \vartheta $ 
		is a local minimum point 
		of $ \cL_{ \bfl } $.
	\end{enumerate}  
\end{lemma}

\begin{cproof}{lem:local:min:ext}
	Throughout this proof let $\vartheta \in \R^{ \fd_\bfl}$
	satisfy
	\begin{equation} \label{lem:local:min:ext:eq:vartheta} 
	\begin{split}
	&\rbr*{ \forall \, k \in \num{ L } \backslash \cu{m , m +1 } \colon \fw^{ \bfl , k, \vartheta } = \fw^{ \ell , k , \theta } \wedge \fb^{ \bfl , k, \vartheta } = \fb^{ \ell , k , \theta  } } \wedge \rbr*{
		\fw^{ \bfl , m , \vartheta } = \rbr*{ \begin{array}{c}
			\fw^{ \ell , m , \theta }  \\
			\hline 0
			\end{array} } } \\
	&\wedge \rbr*{ \fb^{ \bfl , m , \vartheta } = \rbr*{ \begin{array}{c}
			\fb^{ \ell , m , \theta } \\
			\hline \tfrac{\alpha + \beta }{2}
			\end{array} } }
	\wedge \rbr*{
		\fw^{ \bfl ,  m+1 , \vartheta }
		= \rbr*{ \begin{array}{c | c}
			\fw^{ \ell ,  m+1 , \vartheta } &  0
			\end{array} } }
	\wedge  \rbr*{ \fb^{\bfl , m + 1 , \vartheta} = \fb ^{ \ell , m +1 , \vartheta } } .
	\end{split}
	\end{equation}
	\Nobs that \cref{lem:local:min:ext:eq:vartheta}
	assures for all
	$x \in \R ^{d_I}$
	that
	\begin{equation}
	\realization{\bfl , m , \vartheta } ( x )
	= \rbr[\big]{ \realization{\bfl , m , \theta} _1 ( x ) , \ldots , \realization{\bfl , m , \theta} _{\ell_m } ( x ) , \sigma ( \tfrac{\alpha + \beta}{2} ) }
	= \rbr[\big]{ \realization{\bfl , m , \theta} _1 ( x ) , \ldots , \realization{\bfl , m , \theta} _{\ell_m } ( x ) , \sigma ( \alpha ) } 
	\end{equation}
	and $ \realization{\bfl , m + 1  , \vartheta } ( x ) =  \realization{\ell , m + 1, \theta } ( x ) $.
	This and \cref{lem:local:min:ext:eq:vartheta} establish \cref{lem:local:min:ext:item1}.
	
	Next, for every $\psi \in \R^{ \fd_\bfl}$
	let $\cP ( \psi ) \in \R^{ \fd_\ell}$ 
	satisfy
	\begin{equation} \label{lem:local:min:ext:eq:cp}
	\begin{split}
	& \rbr*{ \forall \, k \in \num{ L } \backslash \cu{m+1 },
		i \in \num{\ell_k},
		j \in \num{\ell_{k - 1 } } \colon  \fw^{ \ell , k , \cP ( \psi ) }_{i,j} = \fw^{ \bfl , k , \psi } _{i,j} \wedge \fb^{ \ell , k , \cP ( \psi ) } _{i} = \fb^{ \bfl , k , \psi }_{i}} \\
	& \wedge \rbr*{ \forall \, i \in \num{ \ell_{m+1} },
		j \in \num{ \ell_{m} } \colon \fw^{ \ell , m + 1 , \cP ( \psi ) }_{i,j} = \fw^{ \bfl , m + 1 , \psi } _{i,j}} \\
	& \wedge
	\rbr*{ \forall \, i \in \num{ \ell_{m+1} } \colon \fb^{ \ell , m + 1 , \cP ( \psi ) }_i = \fb^{ \bfl , m + 1 , \psi }_i + \fw^{ \bfl , m + 1 , \psi}_{i , \bfl_m  } \sigma ( \alpha )},
	\end{split}
	\end{equation}
	and let $\psi^{(n)} \in \R^{ \fd_\bfl}$, $n \in \N$,
	satisfy $\lim_{n \to \infty} \psi ^{ ( n ) }  = \vartheta$.
	To show that $\vartheta$ is a local minimum of $\cL_\bfl$ we will prove that $\cL_\bfl ( \psi ^{( n ) } ) \ge \cL_\bfl ( \vartheta ) $ for $n$ sufficiently large.
	\Nobs that 
	\cref{lem:local:min:ext:eq:cp}
	and the fact that for all $i \in \num{\ell_{m + 1 } }$ it holds that $\fw^{ \bfl , m + 1 , \vartheta}_{i , \ell_m + 1 } = \fw^{ \bfl , m + 1 , \vartheta }_{i , \bfl_m  } = 0$
	show that $\lim_{n \to \infty} \cP ( \psi ^{ ( n ) } ) = \cP ( \vartheta ) = \theta$.
	This and the assumption that $\theta \in \R^{ \fd_\ell }$ is a local minimum point of $\cL_\ell$ imply that there exists $N \in \N$ which satisfies
	\begin{equation} \label{lem:local:min:ext:eq:min}
	\forall \, n \in \N \cap [ N , \infty ) \colon \cL_\ell ( \cP ( \psi ^{ ( n ) } ) ) \ge \cL_\ell ( \theta ).
	\end{equation}
	Furthermore, 
	\cref{lem:local:min:ext:eq:vartheta} ensures that there exists $M \in \N$
	which satisfies
	for all $n \in \N \cap [ M , \infty )$,
	$x \in [ a ,b ] ^{ d_I }$ that $\realization{\bfl , m , \psi^{ ( n ) } }_{\bfl_m} ( x ) \in ( \alpha , \beta ) $.
	Hence, we obtain for all $n \in \N \cap [ M , \infty )$,
	$x \in [ a , b ] ^{ d_I } $,
	$i \in \num{\ell_m}$
	that
	\begin{equation}
	\sigma \rbr[\big]{ \realization{\bfl , m , \psi^{ ( n ) } }_{\bfl_m} ( x ) } = \sigma ( \alpha )
	\qqandqq
	\sigma \rbr[\big]{ \realization{\bfl , m , \psi^{ ( n ) } }_{i} ( x ) } = \sigma \rbr[\big]{ \realization{\ell , m , \cP ( \psi ^{ ( n ) } ) } _ i ( x ) } .
	\end{equation}
	Combining this with \cref{lem:local:min:ext:eq:cp}
	implies for all $n \in \N \cap [ M , \infty )$,
	$x \in [ a , b ] ^{ d_I} $,
	$k \in \N \cap [ m + 1 , L ]$
	that
	$\realization{\bfl , k , \psi ^{ ( n ) } } ( x ) = \realization{\ell , k , \cP ( \psi ^{ ( n ) } ) } ( x )$.
	This and \cref{lem:local:min:ext:eq:min} ensure for all $n \in \N \cap [ \max \cu{M , N } , \infty )$
	that $\cL_\bfl ( \psi ^{( n ) } ) = \cL_\ell ( \cP ( \psi ^{( n ) } ) ) \ge \cL_\ell ( \theta ) = \cL_\bfl ( \vartheta )$.
	Therefore, we obtain that $\vartheta $ is a local minimum of $\cL_\bfl$.
	This establishes \cref{lem:local:min:ext:item2}.
\end{cproof}

\begin{prop} 
	\label{prop:local:min:ext} 
	Assume \cref{setting:dnn:gen}, 
	let $ \alpha \in \R $, $ \beta \in ( \alpha, \infty ) $ satisfy 
	for all $ x \in ( \alpha, \beta ) $ that $ \sigma( x ) = \sigma( \alpha ) $,
	let $L \in \N$,
	$\ell = (\ell_0, \ldots, \ell_L ) $, $\bfl = (\bfl_0, \ldots, \bfl_L ) \in \N^{ L + 1 }$
	satisfy $\ell_0 = \bfl_0 = d_I$ and $\ell_L = \bfl_L = d_O$,
	assume for all $i \in \num{L - 1 }$
	that $\ell_i \le \bfl_i$,
	and let $ \theta \in \R^{ \fd_{ \ell } } $ 
	be a local minimum point of $ \cL_{ \ell } $. 
	Then there exists $ \vartheta \in \R^{ \fd_{ \bfl } } $ 
	such that
	\begin{enumerate}[label=(\roman*)]
		\item \label{prop:local:min:ext:item1} 
		it holds for all $x \in \R ^{ d_I }$
		that 
		$ 
		\realization{ \bfl , L , \vartheta  } ( x )
		= \realization{\ell , L , \theta }
		$ 
		and 
		\item \label{prop:local:min:ext:item2} 
		it holds that 
		$ \vartheta $ 
		is a local minimum point 
		of $ \cL_{ \bfl } $.
	\end{enumerate}
\end{prop}
\begin{cproof}{prop:local:min:ext}
	\Nobs that \cref{lem:local:min:ext} and induction establish that there exists $ \vartheta \in \R^{ \fd_{ \bfl } } $ which satisfies \ref{prop:local:min:ext:item1} and \ref{prop:local:min:ext:item2}.
\end{cproof}

\section{Local and global minima in the training of ANNs}
\label{sec:locmin}

In this section we establish our main results on the properties of local and global minima in shallow ANN risk landscapes.
In particular, in \cref{cor:different:risk:discr} we establish a lower bound on the number of distinct risk values achieved by local minimum points in the optimization landscape.
Furthermore, in \cref{prop:all:neurons:active} we show that for every parameter vector with a risk sufficiently close to the minimal risk every neuron has to be active (a notion we define in \cref{setting:snn}).

\subsection{Description of shallow ANNs with general activations}
\label{sec:setting_section}

We first introduce in \cref{setting:snn} our framework for shallow ANNs with a general activation function $\sigma$, which will be employed during this and the next section.

\begin{setting}
\label{setting:snn}
Let $ d \in \N $, $ a \in \R $, $ b \in [a, \infty) $, 
for every $ \width \in \N_0 $
let $ \fd_{ \width } \in \N $
satisfy
$
  \fd_{ \width } = d \width + 2 \width + 1
$,
let 
$
  \mu \colon \cB( [ a , b ] ^d ) \to [0, \infty] 
$ 
be a measure,
let $ f \colon [a,b]^d \to \R $ be bounded and measurable,
and let $ \sigma \colon \R \to \R $ 
be locally bounded and measurable.
For every 
$ \width \in \N_0 $, 
$ 
  \theta = 
  ( \theta_1, \dots, \theta_{ \fd_{ \width } } ) 
  \in \R^{ \fd_{ \width } } 
$
let $ \realization{ \theta } \colon \R^d \to \R $
satisfy for all
$x \in \R^d$
that
\begin{equation}
    \realization{ \theta }( x ) 
    = 
    \theta_{ d \width + 2 \width + 1 } 
    + \ssum_{i=1}^\width \theta_{ d \width + \width + i } \sigma \rbr[\big]{ \theta_{d \width + i } + \ssum_{j=1}^d \theta_{(i - 1 ) d + j } x_j },
\end{equation}
for every
$\width \in \N_0$ let
$\cL_\width \colon \R^{ \fd_{ \width } } \to \R$
satisfy for all 
$\theta \in \R^{ \fd_{ \width } }$
that
\begin{equation}
\label{eq:setting_definition_of_risk}
  \cL_{ \width }( \theta ) 
  = \int_{ [a,b]^d } 
  ( 
    \realization{ \theta }( x ) - f( x ) 
  )^2 
  \, \mu( \d x ) ,
\end{equation}
for every $ \width \in \N_0 $, 
$ \theta \in \R^{ \fd_{ \width } } $
let 
$
  \inact_{ \width }^{ \theta } 
  \subseteq \N 
$
satisfy
\begin{equation}
\label{eq:setting_definition_of_I_set}
  \inact_{ \width }^{ \theta }
  = 
  \bigl\{ 
    i \in \N \cap [0,\width] 
    \colon 
    \bigl(
      \forall \, x \in [a,b]^d \colon 
      \sigma( 
        \theta_{ d \width + i } 
        + 
        \ssum_{ j = 1 }^d 
        \theta_{ (i - 1 ) d + j } x_j 
      ) 
      = 0 
    \bigr)
  \bigr\} ,
\end{equation}
for every $ \width \in \N_0 $ let $ \globinf_{ \width } \in \R $
satisfy
$
  \globinf_{ \width } = 
  \inf\nolimits_{  \theta \in \R^{ \fd_{ \width } } } 
  \cL_{ \width }( \theta )
$,
and let 
$
  \scrK \in \N_0 \cup \cu{ \infty } 
$ 
satisfy
$
  \scrK 
  = 
  \inf\rbr*{ 
    \cu*{ 
      \width \in \N_0 \colon 
      \globinf_{ \width } = 0 
    } 
    \cup \cu{ \infty } 
  }
$.
\end{setting}

In \cref{setting:snn} we consider shallow ANNs with input dimension $d \in \N$,
a variable number $\width \in \N_0$ of neurons on the hidden layer,
and output dimension $1$ and the activation function $\sigma \colon \R \to \R$.
Similarly to \cref{setting:dnn:gen} we introduce the ANN realization functions $\realization{\theta}$ and the risk functions $\cL_\width \colon \R^{ \fd_\width} \to \R$, $\width \in \N_0$, as the squared $L^2$-distance between the ANN realization and the target function.
In addition, for every $\width \in \N_0$ and every ANN parameter vector $ \theta \in \R^{ \fd_{ \width } } $ we denote by $\inact_\width^\theta \subseteq \N$ the set of \emph{inactive} neurons,
i.e., the set of neurons that have zero output value for every input vector $x \in [a,b] ^d$.

\subsection{Embeddings of local minima into larger architectures}

In the next auxiliary result, \cref{lem:mu:finite}, 
we show that under the assumptions of \cref{setting:snn} the measure $\mu$, 
which describes the distribution of the considered input data, is necessarily finite.

\begin{lemma} \label{lem:mu:finite}
	Assume \cref{setting:snn} and let $ \width \in \N_0 $. 
	Then it holds for all 
	$ 
	\theta 
	= ( \theta_1, \dots, \theta_{ \fd_{ \width } } ) 
	$, 
	$ 
	\vartheta 
	= ( \vartheta_1, \dots, \vartheta_{ \fd_{ \width } } ) 
	\in \R^{ \fd_{ \width } } 
	$
	with 
	$
	\sum_{ i = 1 }^{ d \width + 2 \width }
	(
	\abs{ \theta_i }
	+
	\abs{ \vartheta_i }
	)
	= 0 
	\neq 
	\theta_{ \fd_{ \width } } 
	- 
	\vartheta_{ \fd_{ \width } }   
	$
	that
	\begin{equation}
		\label{lem:mu:finite:eq_claim}
		0 \le \mu ( [ a , b ]^d ) 
		\le \abs{ \theta_{\fd_\width } - \vartheta_{\fd_\width } }^{ - 2 } 
		\bigl[
		  | \cL_{ \width }( \theta ) |^{ 1 / 2 } + 
		  | \cL_{ \width }( \vartheta ) |^{ 1 / 2 } 
        \bigr]^2 < \infty.
	\end{equation}
\end{lemma}

\begin{cproof}{lem:mu:finite}
	\Nobs that Minkowski's inequality demonstrates for all
	$\theta , \vartheta \in \R^{ \fd_\width }$
	that
	\begin{equation}
		\begin{split}
			& \rbr*{ \int_{[a,b] ^d } \abs{\realization{\theta} ( x ) - \realization{\vartheta} ( x ) } ^2 \, \mu ( \d x ) } ^{ \! 1/2} \\
			&\le \rbr*{ \int_{[a,b] ^d } \abs{\realization{\theta} ( x ) - f ( x ) } ^2 \, \mu ( \d x ) } ^{ \! 1/2}
			+\rbr*{ \int_{[a,b] ^d } \abs{f ( x ) - \realization{\vartheta} ( x ) } ^2 \, \mu ( \d x ) } ^{ \! 1/2} \\
			&= (\cL_\width ( \theta ) ) ^{ 1/2} + ( \cL_\width ( \vartheta ) ) ^{ 1/2}.
		\end{split}
	\end{equation}
This shows for all $ 
\theta 
= ( \theta_1, \dots, \theta_{ \fd_{ \width } } ) 
$, 
$ 
\vartheta 
= ( \vartheta_1, \dots, \vartheta_{ \fd_{ \width } } ) 
\in \R^{ \fd_{ \width } } 
$
with 
$
\sum_{ i = 1 }^{ d \width + 2 \width }
(
\abs{ \theta_i } +
\abs{ \vartheta_i } )
= 0 $
that
\begin{equation}
	(\cL_\width ( \theta ) ) ^{ 1/2} + ( \cL_\width ( \vartheta ) ) ^{ 1/2} 
\ge 
  \rbr*{ \int_{ [ a,b] ^d } \abs{ \theta_{\fd_\width } - \vartheta_{\fd_\width } }^2 \, \mu ( \d x ) }^{ \! 1/2} 
= 
  \abs{ \theta_{\fd_\width } - \vartheta_{\fd_\width } } ( \mu ( [ a , b ] ^d ) ) ^{ 1/2}.
\end{equation}
This establishes \cref{lem:mu:finite:eq_claim}.
\end{cproof}

We now restate the embedding result for local minima into larger architectures from \cref{prop:local:min:ext} in the shallow ANN setting considered in this section.

\begin{prop} 
\label{prop:local:min:ext:shallow} 
Assume \cref{setting:snn}, 
let $ \alpha \in \R $, $ \beta \in ( \alpha, \infty ) $ satisfy 
for all $ x \in ( \alpha, \beta ) $ that $ \sigma( x ) = \sigma( \alpha ) $,
let $\width, \mathfrak{H} \in \N_0$
satisfy
$\width \le \mathfrak{H}$,
and let $ \theta \in \R^{ \fd_{ \width } } $ 
be a local minimum point of $ \cL_{ \width } $. 
Then there exists $ \vartheta \in \R^{ \fd_{ \mathfrak{H} } } $ 
such that
\begin{enumerate}[label=(\roman*)]
\item 
it holds that 
$ 
  \realization{ \vartheta }
  = \realization{ \theta }
$ 
and 
\item 
it holds that 
$ \vartheta $ 
is a local minimum point 
of $ \cL_{ \mathfrak{H} } $.
\end{enumerate}
\end{prop}

It is possible to deduce \Cref{prop:local:min:ext:shallow} from \cref{prop:local:min:ext}, but only in the case $\width > 0$. 
For this reason we give a self-contained proof which relies on similar ideas as in the deep ANN case.

\begin{cproof}{prop:local:min:ext:shallow}
 Assume without loss of generality that $\fH > \width$
 and let $\vartheta \in \R^{ \fd_\fH}$ satisfy
 \begin{equation}
 	\label{prop:local:min:ext:shallow:eq1}
 	\begin{split}
 		& \rbr*{ \forall \, i \in \num{\width}, j \in \num{d} \colon \vartheta_{(i - 1)d + j } = \theta_{ ( i - 1 ) d + j } }
 		\wedge \rbr*{ \forall \, i \in \num{\fH} \backslash \num{\width},  j \in \num{d} \colon \vartheta_{(i - 1)d + j } = 0 }
 		\\ & \wedge \rbr*{ \forall \, i \in \num{\width} \colon \vartheta_{d \fH + i } = \theta_{d \width + i } }
 		\wedge \rbr*{  \forall \, i \in \num{\fH} \backslash \num{\width} \colon \vartheta_{d \fH + i } = \tfrac{\alpha + \beta}{2} } \\ 
 		& \wedge \rbr*{ \forall \, i \in \num{\width} \colon \vartheta_{(d + 1) \fH + i } = \theta_{ ( d + 1) \width + i } }
 		\wedge \rbr*{ \forall \, i \in \num{\fH} \backslash \num{\width} \colon \vartheta_{(d + 1 ) \fH + i } = 0 }
 		\wedge \rbr*{  \vartheta_{\fd_\fH} = \theta_{\fd_\width } } .
 	\end{split} 
 \end{equation} 
\Nobs that \cref{prop:local:min:ext:shallow:eq1} implies for all $x \in \R^d$ that $\realization{\vartheta} ( x ) = \realization{\theta} ( x )$.
To show that $\vartheta$ is a local minimum point let $\cP \colon \R^{  \fd_\fH } \to \R ^{ \fd_\width }$ satisfy for all
$\psi = (\psi_1, \ldots , \psi_{\fd_\fH} ) \in \R^{ \fd_\fH}$ that
\begin{equation}
	\label{prop:local:min:ext:shallow:eq2}
	\begin{split}
		& \rbr*{ \forall \, i \in \num{\width}, j \in \num{d} \colon ( \cP ( \psi ) )_{(i - 1)d + j } = \psi_{ ( i - 1 ) d + j } } \\
		& \wedge \rbr*{ \forall \, i \in \num{\width} \colon (( \cP ( \psi ) )_{d \width + i } = \psi_{d \fH + i } ) \wedge ( ( \cP ( \psi ) )_{(d+1) \width + i } = \psi_{ ( d+1) \fH + i } ) } \\
		& \wedge ( \cP ( \psi ) )_{\fd_\width } = \psi_{\fd_\fH} + \textstyle \sigma(\alpha) \sum_{k=\width + 1 } ^\fH \psi_{ ( d+1) \fH + k }
	\end{split}
\end{equation}
and let $\psi^{ ( n ) } \in \R^{ \fd_\fH } $, $n \in \N$,
satisfy $\lim_{n \to \infty } \psi ^{( n ) } = \vartheta $.
\Nobs that \cref{prop:local:min:ext:shallow:eq1,prop:local:min:ext:shallow:eq2} ensure that $\lim_{n \to \infty} \cP ( \psi ^{ ( n ) } ) = \cP ( \vartheta ) = \theta $.
This and the assumption that $\theta \in \R^{ \fd_\width }$ is a local minimum point of $\cL_\width$ imply that there exists $N \in \N$ which satisfies
\begin{equation} \label{prop:local:min:ext:shallow:eq3}
	\forall \, n \in \N \cap [ N , \infty ) \colon \cL_\width ( \cP ( \psi ^{ ( n ) } ) ) \ge \cL_\width ( \theta ).
\end{equation}
Furthermore, 
\cref{prop:local:min:ext:shallow:eq1} ensures that there exists $M \in \N$
which satisfies
for all $n \in \N \cap [ M , \infty )$,
$i \in \num{\fH} \backslash \num{\width}$,
$x \in [ a ,b ] ^{ d }$ that $\psi^{ ( n ) }_{ d \width + i } + \sum_{j=1}^d \psi^{ ( n ) }_{(i-1) d + j } x_j \in ( \alpha , \beta ) $.
Hence, we obtain for all $n \in \N \cap [ M , \infty )$,
$x \in [ a , b ] ^{ d } $,
$i \in \num{\fH}$
that
\begin{equation}
	\sigma \rbr[\big]{ \psi^{ ( n ) }_{ d \width + i } + \ssum_{j=1}^d \psi^{ ( n ) }_{(i-1) d + j } x_j } = 
	\begin{cases}
		\sigma ( \alpha ) & \colon i > \width \\
		\sigma \rbr[\big]{\cP ( \psi^{ ( n ) } )_{ d \width + i } + \ssum_{j=1}^d \cP( \psi^{ ( n ) } ) _{(i-1) d + j } x_j } & \colon i \le \width .
	\end{cases}
\end{equation}
Combining this with \cref{prop:local:min:ext:shallow:eq2}
implies for all $n \in \N \cap [ M , \infty )$,
$x \in [ a , b ] ^{ d} $
that
$\realization{\psi ^{ ( n ) } } ( x ) = \realization{ \cP ( \psi ^{ ( n ) } ) } ( x )$.
This and \cref{prop:local:min:ext:shallow:eq3} ensure for all $n \in \N \cap [ \max \cu{M , N } , \infty )$
that $\cL_\fH ( \psi ^{( n ) } ) = \cL_\width ( \cP ( \psi ^{( n ) } ) ) \ge \cL_\width ( \theta ) = \cL_\fH ( \vartheta )$.
Therefore, we obtain that $\vartheta $ is a local minimum of $\cL_\fH$.
\end{cproof}

\subsection{Existence of global minima for ReLU ANNs}
\label{subsec:global_min}

In this part we recall the known existence results for global minima of the risk function which have been established by Dereich \& Kassing~\cite{DereichKassing2023} and in our previous article \cite{JentzenRiekert2022_JML}.
To transfer these results to our framework we require the elementary \cref{lem:continuous:extension}, which is a special case of the Tietze extension theorem (cf., e.g., Munkres~\cite[Theorem 35.1]{Munkres2000topology}).
The proof is only included for completeness.

\begin{lemma}
	\label{lem:continuous:extension}
	Let $ a \in \R $, $ b \in (a,\infty) $, $ f \in C( [a,b]^d, \R ) $. 
	Then there exists $ g \in C( \R^d, \R ) $ such that 
	$
	g|_{ [a,b]^d } = f
	$.
\end{lemma}

\begin{cproof}{lem:continuous:extension}
\Nobs that the fact that $[a,b]^d \subseteq \R^d$ is closed and convex and the projection theorem for Hilbert spaces (cf., e.g., Rudin~\cite[Theorem 12.3]{Rudin1991_FA}) ensure that there exists a unique function $P \colon \R^d \to [a,b]^d$ which satisfies for all $x \in [a,b] ^d$ that $\norm{x-P(x)} = \inf_{y \in [a,b]^d } \norm{x-y}$.

We next show that $P$ is continuous.
For this assume that $(x_n)_{n \in \N} \subseteq \R^d$, $y \in \R^d$ satisfy $\limsup_{n \to \infty} \norm{x_n - y } = 0$.
\Nobs that the fact that $[a,b] ^d$ is compact demonstrates that there exist $z \in [a,b]^d$ and a strictly increasing $k \colon \N \to \N$ such that $\limsup_{n \to \infty} \norm{ P ( x_{k(n)} ) - z } = 0$.
The fact that the distance function $\R^d \ni u \mapsto \inf_{w \in [a,b]^d } \norm{u - w} \in \R$ is continuous therefore assures that
\begin{equation}
	\begin{split}
		\norm{ y - z } = \lim_{n \to \infty} \norm{ x_{ k(n) } - P ( x_{k(n)} ) }
		= \lim_{n \to \infty} \br*{\inf_{w \in [a,b]^d } \norm{x_{k(n) } - w} } = \inf_{w \in [a,b]^d } \norm{y - w}.
	\end{split}
\end{equation}
By uniqueness of the projection we hence obtain that $P(y)=z$.
Since the same holds for every subsequence of $(x_n)$ we conclude that $\limsup_{n \to \infty} \norm{ P ( x_{k(n)} ) - P(y) } = 0$. 
Hence, $P$ is continuous.

	Now define the function $g \colon \R^d \to \R$ as the composition $g = f \circ P$.
	\Nobs that the fact that both $P$ and $f$ are continuous ensures that $g$ is continuous.
	Moreover, the fact that for all $x \in [a,b]^d$ we have that $P(x)=x $ demonstrates that
	$ g|_{ [a,b]^d } = f$.
\end{cproof}

\begin{prop}
	\label{prop:global:exist}
	Assume \cref{setting:snn},
	assume that $f$ is continuous,
	assume for all $x \in \R$
	that
	$\sigma ( x ) = \max \cu{ x , 0 }$,
	let $\dens \colon \R^d \to [0 , \infty )$ be continuous,
	assume $\dens^{-1} ( ( 0 , \infty ) ) = ( a,b) ^d$,
	assume for all $B \in \cB ( [ a , b ]^d )$ that $\mu ( B ) = \int_B \dens ( x ) \, \d x$,
	and let $\width \in \N_0$.
	Then there exists $\theta \in \R^{ \fd_{ \width } }$
	such that
	$\cL _\width ( \theta ) = \globinf_{ \width }$.
\end{prop}

\begin{cproof}{prop:global:exist}
	\Nobs that \cref{lem:continuous:extension} ensures that there exists $g \in C(\R^d , \R$) which satisfies $g |_{[a,b] ^d } = f$.
	The assumption that $\dens^{-1} ( ( 0 , \infty ) ) = ( a,b) ^d$ therefore shows for all $\theta \in \R^{ \fd_\width }$ that
	\begin{equation}
		\begin{split}
			\cL_\width ( \theta ) 
			&= \int_{[a,b]^d } ( \realization{\theta} ( x ) - f ( x ) ) ^2 \, \mu ( \d x )
			=  \int_{[a,b]^d } ( \realization{\theta} ( x ) - g ( x ) ) ^2 \dens ( x ) \, \d x \\
			& =  \int_{\R^d } ( \realization{\theta} ( x ) - g ( x ) ) ^2 \dens ( x ) \, \d x .
		\end{split}
	\end{equation}
	In the case $\width > 0$ it therefore follows from Dereich \& Kassing~\cite[Theorem 1.1]{DereichKassing2023} that there exists $\theta \in \R^{ \fd_{ \width } }$
	such that
	$\cL _\width ( \theta ) = \globinf_{ \width }$.
	In the case $\width = 0$ the claim is a direct consequence of the fact that $\cL_\width$ is a quadratic function of $\theta_{\fd_\width } = \theta_1$.
\end{cproof}

\begin{prop}
\label{prop:global:exist:1d}
Assume \cref{setting:snn}, 
assume $d=1$,
assume that $f$ is Lipschitz continuous,
assume for all $x \in \R$
that
$\sigma ( x ) = \max \cu{ x , 0 }$,
and let $\width \in \N_0$.
Then there exists $\theta \in \R^{ \fd_{ \width } }$
such that
$\cL _\width ( \theta ) = \globinf_{ \width }$.
\end{prop}
\begin{cproof}{prop:global:exist:1d}
In the case $\width > 0$ this statement follows directly from \cite[Theorem 1.1]{JentzenRiekert2022_JML}. 
In the case $\width = 0$ the claim follows in the same way as in \cref{prop:global:exist}.
\end{cproof}

\subsection{Discriminatory functions and density of ANNs}

We next define in \cref{def:disc} our notion of discriminatory functions, which is strongly related to the universal approximation theorem for shallow ANNs.
\cref{def:disc} is a slight adaptation of known definitions in the literature; see, e.g., Cybenko~\cite{Cybenko1989}, 
Gühring et al.~\cite[Definition 2.1]{Guehring2022},
and
Capel \& Oc{\'a}riz~\cite[Definition 4.1]{Capel2020_Approx}.

\begin{definition}[Discriminatory function] 
\label{def:disc}
Let $ d \in \N $, 
let $ K \subseteq \R^d $ be compact, 
let $ \mu \colon \cB( K ) \to [0,\infty] $ 
be a measure, 
and let 
$
  \sigma \colon \R \to \R 
$ 
be locally bounded and measurable. 
We say that $ \sigma $ is discriminatory for $ \mu $ 
if and only if it holds for all 
measurable $ \varphi \colon K \to \R $ 
with 
$
  \int_K | \varphi(x) | \, \mu( \d x ) < \infty
$
and 
\begin{equation}
  \sup_{
    w \in \R^d 
  } 
  \sup_{ 
    b \in \R 
  }
  \abs*{
    \int_K 
    \sigma( \spro{ w, x } + b ) 
    \, 
    \varphi( x ) 
    \,
    \mu( \d x ) }
  = 0
\end{equation}
that $ \int_K | \varphi( x ) | \, \mu( \d x ) = 0 $.
\end{definition}


\cfclear
\begin{lemma} \label{lem:discriminatory:equiv}
Let $ d \in \N $,
let $ K \subseteq \R^d $ be compact,
let $ \sigma \colon \R \to \R $ be locally bounded and measurable,
let $\cN \subseteq \cu{f \colon K \to \R \colon f \text{ is measurable}}$
satisfy
\begin{equation}
    \cN = \operatorname{span}_\R  \cu{ (K \ni x \mapsto \sigma ( \spro{ w , x } + b ) \in \R ) \colon w \in \R^d , \, b \in \R },
\end{equation}
assume that 
\begin{equation} \label{lem:discr:equiv:eq:assumption}
\forall \, \varepsilon \in (0, \infty ), f \in C(K) \colon \exists \, g \in \cN \colon \sup\nolimits_{x \in K} \abs{ f(x) - g(x) } < \varepsilon,
 \end{equation}
and let $\mu \colon \cB ( K ) \to [0 , \infty ]$ be a measure.
Then $\sigma$ is discriminatory for $\mu$ \cfadd{def:disc}\cfload.
\end{lemma}
\begin{cproof}{lem:discriminatory:equiv}
	Let $ \varphi \colon K \to \R $ 
	be measurable and satisfy
	$
	\int_K | \varphi(x) | \, \mu( \d x ) < \infty
	$
	and 
	\begin{equation}
		\label{lem:discr:phi:assumption}
		\sup_{
			w \in \R^d 
		} 
		\sup_{ 
			b \in \R 
		}
		\abs*{ \int_K 
		\sigma( \spro{ w, x } + b ) 
		\, 
		\varphi( x ) 
		\,
		\mu( \d x ) }
		= 0
	\end{equation}
	and let $\nu \colon \cB ( K ) \to \R$ be the finite signed measure which satisfies for all $E \in \cB ( K ) $ that $\nu ( E ) = \int_E \varphi ( x ) \, \mu ( \d x )$.
	\Nobs that \cref{lem:discr:phi:assumption}
implies for all $g \in \cN$ that 
\begin{equation}
	\int_K g ( x ) \, \nu ( \d x ) = \int_K g ( x ) \varphi ( x ) \, \mu ( \d x ) = 0. 
\end{equation}
Combining this with \cref{lem:discr:equiv:eq:assumption} demonstrates for all $f \in C(K)$ that $\int_K f ( x ) \, \nu ( \d x ) = 0$. 
Hence, we obtain that $\nu = 0$,
which shows that $\int_K \abs{ \varphi ( x ) } \, \mu ( \d x ) = 0$.
Consequently, $\sigma$ is discriminatory for $\mu$.
\end{cproof}

\begin{remark} 
If $\sigma$ is continuous and discriminatory for the Lebesgue measure on $K$ then it follows conversely that \cref{lem:discr:equiv:eq:assumption} holds; cf., e.g., \cite[Theorem 2]{Cybenko1989} or \cite[Theorem 4.2]{Capel2020_Approx}.
\end{remark}

\cfclear
\begin{lemma} \label{lem:discriminatory:charact}
Let $d \in \N$,
let $K \subseteq \R^d$ be compact,
let $\sigma \colon \R \to \R$ be locally bounded and measurable,
assume that the closure of the set of points of discontinuity of $\sigma$ has zero Lebesgue measure,
assume that there is no polynomial which agrees with $\sigma$ almost everywhere,
and let $\mu \colon \cB ( K ) \to [0 , \infty ]$ be a measure.
Then $\sigma$ is discriminatory for $ \mu $ \cfadd{def:disc}\cfload.
\end{lemma}
\begin{cproof}{lem:discriminatory:charact}
Combining \cref{lem:discriminatory:equiv} and the universal approximation theorem in \cite[Theorem 1]{LeshnoLinPinkusSchocken1993} 
establishes that $ \sigma $ is discriminatory for $ \mu $.
\end{cproof}

\subsection{Improvement of risk levels of local minima}

We now show in \cref{prop:extra:neuron:improve} that, roughly speaking, increasing the number of hidden neurons from $\width \in \N_0$ to $\width + 1$ always yields a smaller risk value for a suitably chosen parameter vector, unless it is already possible to achieve a risk value $0$ with $\width$ hidden neurons.

Similar ideas as in the proof of \cref{prop:extra:neuron:improve}
have been used in Dereich et al.~\cite[Proposition 3.5]{DereichJentzenKassing2023},
which considers only the ReLU activation function but more general loss functions.

\cfclear
\begin{prop} \label{prop:extra:neuron:improve}
Assume \cref{setting:snn}, 
assume that $ \sigma $ is discriminatory for $ \mu $, 
and let $ \width \in \N_0 $, $ \theta \in \R^{ \fd_{ \width } } $ 
satisfy
$ \cL_{ \width }( \theta ) > 0 $ \cfadd{def:disc}\cfload.
Then there exists $ \vartheta \in \R^{ \fd_{ \width + 1 } } $
such that 
$
  \cL_{ \width + 1 }( \vartheta ) 
  < \cL_{ \width }( \theta )
$.
\end{prop}
\begin{cproof}{prop:extra:neuron:improve}
We prove \cref{prop:extra:neuron:improve} by contradiction.
We thus assume for all 
$\vartheta \in \R ^{ \fd_{\width + 1 } }$
that 
\begin{equation} \label{prop:extra:neuron:eq:contr}
    \cL_{\width + 1 } ( \vartheta ) \ge \cL _ \width ( \theta ).
\end{equation}
\Nobs that for all $v , b \in \R$, $w \in \R^d$ there exists
$\vartheta \in \R ^{ \fd_{\width + 1 } }$
which satisfies for all
$x \in \R^d$
that
$\realization{\vartheta } ( x ) = \realization{\theta } ( x ) + v \sigma ( \spro{ w , x } + b ) $.
In the following let $\varphi \colon [ a,b]^d  \to \R$ satisfy for all $x \in [ a,b ] ^d $ that $\varphi ( x ) =  \realization{\theta} ( x ) - f ( x ) $
and for every $w \in \R^d$,
$b \in \R$
let $L_{w,b} \colon \R \to \R$ satisfy for all
$v \in \R$
that
\begin{equation}
    L_{w,b} ( v ) = \int_{[a,b]^d } ( \realization{\theta } ( x ) + v \sigma ( \spro{ w , x } + b ) - f ( x ) ) ^2 \, \mu ( \d x ).
\end{equation}
\Nobs that the fact that $\mu$ is a finite measure (cf.~\cref{lem:mu:finite}) and the fact that $\varphi$ is locally bounded assure that $\int_{[a,b] ^d } \abs{ \varphi ( x ) } \, \mu ( \d x ) < \infty $.
Next, \nobs that for all $w \in \R^d$, $b \in \R$
it holds that $L_{w,b}$ is differentiable and satisfies for all $v \in \R$ that
\begin{equation}
    (L_{w,b} ) ' ( v ) = 2 \int_{[a,b]^d } \sigma ( \spro{ w , x } + b )  ( \realization{\theta } ( x ) + v \sigma ( \spro{ w , x } + b ) - f ( x ) )  \, \mu ( \d x ).
\end{equation}
Furthermore, \cref{prop:extra:neuron:eq:contr} ensures for all $w \in \R^d$, $b , v \in \R$ that $L_{w,b} ( v ) \ge L_{w,b} ( 0 ) $. Hence, we obtain for all $w \in \R^d$, $b \in \R$ that
\begin{equation}
	\begin{split}
		    0 
		    &= (L_{w,b} ) ' ( 0 ) 
		    = 2 \int_{[a,b]^d } \sigma ( \spro{ w , x } + b )  ( \realization{\theta } ( x ) - f ( x ) )  \, \mu ( \d x ) \\
		    &= 2 \int_{[a,b]^d } \sigma ( \spro{ w , x } + b ) \varphi ( x ) \, \mu ( \d x ) .
	\end{split}
\end{equation}
Combining this with the assumption that $\sigma$ is discriminatory for $\mu$ shows that $\int_{[a,b] ^d } \abs{ \varphi ( x)} \, \mu ( \d x ) = 0$. 
This demonstrates that for $\mu$-almost every $x \in [a,b]^d$ it holds that $\realization{\theta} ( x ) = f ( x )$.
Hence, we obtain that $\cL_ \width ( \theta ) = 0$, which is a contradiction.
\end{cproof}

\cfclear
\begin{cor} 
\label{cor:extra:neuron:improve}
Assume \cref{setting:snn}, 
assume that $ \sigma $ is discriminatory for $ \mu $, 
and let 
$ \width \in \N_0 \cap [ 0, \scrK ) $, 
$ \theta \in \R^{ \fd_{ \width } } $
satisfy 
$
  \cL_{ \width }( \theta ) = \globinf_{ \width }
$ \cfadd{def:disc}\cfload.
Then 
\begin{equation}
  \globinf_{ \width + 1 } 
  <
  \globinf_{ \width } 
  .
\end{equation}
\end{cor}
\begin{cproof}{cor:extra:neuron:improve}
\Nobs that the assumption that 
$\width  < \scrK $
ensures that 
$
  \globinf_{ \width } > 0
$. 
Combining this with 
\cref{prop:extra:neuron:improve} 
demonstrates that there exists 
$
  \vartheta \in \R^{ \fd_{ \width + 1 } }
$
such that
$
  \globinf_{ \width + 1 } \leq 
  \cL_{ \width + 1 }( \vartheta ) 
  < \cL_{ \width }( \theta ) 
$. 
The assumption that 
$
  \cL_{ \width }( \theta ) = \globinf_{ \width }
$ 
hence shows that
$
  \globinf_{ \width + 1 } 
  < \cL_{ \width }( \theta ) = \globinf_{ \width }
$. 
\end{cproof}

The previous result leads us to the following conjecture.

\begin{conjecture}
Assume \cref{setting:snn} and assume that $\sigma$ is discriminatory for $ \mu $.
Then it holds for all $\width  \in \N_0 \cap [ 0 , \scrK ) $ that
$ \globinf_{\width + 1 } < \globinf_{ \width } $.
\end{conjecture}

We know that $\N_0 \ni \width \mapsto \globinf_{ \width } \in [0, \infty )$ is non-increasing and tends to zero by the universal approximation theorem. 
The question is: 
Does it decrease strictly (until it reaches zero for $\width = \scrK$ in the case $\scrK < \infty$)?
So far we can only show this under the assumption that global minima exist, by using \cref{prop:extra:neuron:improve}.
For the risk function $\theta \mapsto \int_{[a,b]^d } \abs{ \realization{\theta} ( x ) - f ( x ) } \, \mu ( \d x )$, which is defined via a non-differentiable loss function, the corresponding statement is false as shown in 
Dereich et al.~\cite[Example 3.7]{DereichJentzenKassing2023}.

\subsection{Different risk levels of local minima}
\label{sec:different_risk_levels_of_local_minima}

We now employ \cref{prop:extra:neuron:improve} to establish in \cref{cor:different:risk:discr} one of the main results of this section: Under the assumption that global minima exist, there are always at least $\min \cu{\scrK , \fH } + 1 $ local minimum points with distinct risk values where $\scrK$ is the minimum number of hidden neurons needed to achieve zero approximation error as defined in \cref{setting:snn} and $\fH$ denotes the total number of hidden neurons in the considered ANN architecture.

\cfclear
\begin{cor}
\label{cor:different:risk:discr}
Assume \cref{setting:snn},
assume that $ \sigma $ is locally Lipschitz continuous and discriminatory for $ \mu $, 
let $ \alpha \in \R $, $ \beta \in ( \alpha, \infty ) $ 
satisfy for all $ x \in ( \alpha, \beta ) $ 
that $ \sigma( x ) = \sigma( \alpha ) $,
let $ \fH, n \in \N_0 $ 
satisfy $ n = \min\cu{ \fH, \scrK } $, 
and assume for all 
$ \width \in \{ 0, 1, \dots, n \} $ 
that 
\begin{equation}
  ( \cL_{ \width } )^{ - 1 }( \{ \globinf_{ \width } \} )
  \neq \emptyset 
\end{equation}
\cfadd{def:disc}\cfload.
Then there exist 
local minimum points 
$
  \vartheta_0, \vartheta_1, \dots, \vartheta_n \in \R^{ \fd_{ \fH } }
$ 
of $ \cL_{ \fH } $ such that 
\begin{equation}
  \inf_{ \xi \in \R } 
  \int_{ [a,b]^d } 
    \abs{ f(x) - \xi }^2 \, 
  \mu( \d x ) 
\ge 
  \globinf_0
  =
  \cL_{ \fH }( \vartheta_0 ) 
  > 
  \globinf_1
  =
  \cL_{ \fH }( \vartheta_1 ) 
  > 
  \dots 
  > 
  \globinf_n
  =
  \cL_{ \fH }( \vartheta_n ) 
  \geq 0
  .
\end{equation}
\end{cor}
\begin{cproof}{cor:different:risk:discr}
	Throughout this proof for every $\width \in \cu{0, 1, \ldots, n }$ let $\theta_\width \in \R^{ \fd_\width }$
	satisfy $\cL_\width ( \theta_\width ) = \globinf_\width $.
\Nobs that for all 
$ \width \in \cu{ 0, 1, \dots,  n } $ we have that 
$ \theta_{ \width } $ is a local minimum point 
of $ \cL_{ \width } $. 
Combining this with \cref{prop:local:min:ext} 
demonstrates that for all $\width \in \cu{0, 1, \ldots, n }$ there exists $\vartheta_\width \in \R^{\fd_\fH}$ which satisfies that $\vartheta_\width$ is a local minimum of $\cL_\fH$ and which satisfies for all $x \in \R^d$ that $\realization{\vartheta_\width} ( x ) = \realization{\theta_\width } ( x )$.
Hence, we obtain for all $\width \in \cu{0, 1, \ldots, n }$
that $\cL_\fH ( \vartheta_\width ) = \cL_\width ( \theta_\width ) = \globinf_{ \width }$.
Furthermore,
\nobs that the fact that $n \le \scrK$ assures for all $\width \in \cu{0, 1, \ldots, n - 1 } $ that $\globinf_{ \width } > 0$.
Combining this with \cref{prop:extra:neuron:improve} shows for all $\width \in \cu{0, 1, \ldots, n - 1 } $
that
\begin{equation}
    \cL_\fH ( \vartheta_{\width + 1 } ) 
    = \globinf_{\width + 1 }
    < \globinf_{ \width } = \cL_\fH ( \vartheta_{\width } ).
\end{equation}
Finally, \nobs that the fact that $\vartheta_0$ is a local minimum of $\cL_\fH$,
\cref{rem:local:min},
and \cref{prop:clarke:crit}
assure that 
\begin{equation}
      \inf\nolimits_{\xi \in \R } \int_{[a,b]^d } \abs{f(x) - \xi } ^2 \, \mu ( \d x ) \ge \cL_\fH ( \vartheta_0 ) .
\end{equation}
\end{cproof}

Next, we specialize \cref{cor:different:risk:discr} for multiple activation functions 
which satisfy the required assumption including, in particular, the highly popular ReLU activation.

\cfclear
\begin{cor}[Clipping and power ReLU/RePU activations] 
\label{cor:different:risk:examples}
Assume \cref{setting:snn}, 
let $ k \in \N $, $ c \in (0,\infty] $ 
satisfy for all $ x \in \R $ that 
\begin{equation}
  \sigma(x) = ( \max\{ \min\{ x, c \}, 0 \} )^k 
  ,
\end{equation}
let $\fH, n \in \N_0$
satisfy $n = \min \cu{ \fH , \scrK }$,
and assume for all $\width \in \cu{0, 1, \ldots, n}$ that there exists $\theta_\width \in \R^{\fd_{ \width }}$ such that $\cL_\width ( \theta_\width ) = \globinf_{ \width } $.
Then there exist $\vartheta_0, \vartheta_1, \ldots, \vartheta_n \in \R^{\fd_\fH}$ 
such that
\begin{enumerate}[label=(\roman*)]
\item 
it holds for all $ \width \in \cu{ 0, 1, \dots, n } $ that $ \vartheta_{ \width } $ is a local minimum of $ \cL_{ \fH } $
and 
\item 
it holds that 
\begin{equation}
  \inf_{ \xi \in \R } 
  \int_{ [a,b]^d } \abs{ f(x) - \xi }^2 \, \mu( \d x ) 
\ge
  \cL_{ \fH }( \vartheta_0 ) 
  > \cL_{ \fH }( \vartheta_1 ) > \cdots > \cL_{ \fH }( \vartheta_n ) 
  .
\end{equation}
\end{enumerate}
\end{cor}

\begin{proof}[Proof of \cref{cor:different:risk:examples}]
\Nobs that \cref{lem:discriminatory:charact} assures that $\sigma$ is discriminatory for $\mu$ \cfadd{def:disc}\cfload.
Combining this with \cref{cor:different:risk:discr} completes the proof of \cref{cor:different:risk:examples}.
\end{proof}

\subsection{Different risk levels of local minima for ReLU ANNs}
\label{sec:different_risk_levels_of_local_minima_for_ReLU_ANNs}

We now specialize \cref{cor:different:risk:examples} to the two particular situations 
considered in \cref{subsec:global_min} where the existence of global minima is known.

\begin{cor}[ReLU ANNs with one-dimensional input] 
\label{cor:diff:risk:relu:1d}
Assume \cref{setting:snn},
assume $d=1$,
assume that $f$ is Lipschitz continuous,
assume for all $x \in \R$ that $\sigma ( x ) = \max \cu{x,0}$,
and let $\fH , n \in \N_0$ satisfy $n = \min \cu{ \fH , \scrK }$.
Then there exist $\vartheta_0, \vartheta_1, \ldots, \vartheta_n \in \R^{\fd_\fH}$ 
such that
\begin{enumerate}[label=(\roman*)]
\item 
it holds for all $\width \in \cu{0, 1, \ldots, n}$ that $\vartheta_\width$ is a local minimum of $\cL_\fH$ and 
\item 
it holds that
\begin{equation}
       \inf\limits_{\xi \in \R } \int_a^b \abs{f(x) - \xi } ^2 \, \mu ( \d x ) \ge
       \cL_\fH ( \vartheta_0 ) > \cL_\fH ( \vartheta_1 ) > \ldots > \cL_\fH ( \vartheta _ n ) .
\end{equation}
\end{enumerate}
\end{cor}
\begin{proof}[Proof of \cref{cor:diff:risk:relu:1d}]
\Nobs that \cref{prop:global:exist:1d} ensures for every $\width \in \cu{0, 1, \ldots, \fH}$ that there exists $\theta_\width \in \R^{\fd_{ \width }}$ such that $\cL_\width ( \theta_\width ) = \globinf_{ \width } $.
Combining this with \cref{cor:different:risk:examples} completes the proof of \cref{cor:diff:risk:relu:1d}.
\end{proof}

\begin{cor}[ReLU ANNs with multi-dimensional input]
\label{cor:diff:risk:relu}
	Assume \cref{setting:snn},
	assume that $f$ is continuous,
	assume for all $x \in \R$ that $\sigma ( x ) = \max \cu{x,0}$,
	let $\dens \colon \R^d \to [0 , \infty )$ be continuous,
	assume $\dens^{-1} ( ( 0 , \infty ) ) = ( a , b ) ^d$,
	assume for all $B \in \cB ( [a,b] ^d )$ that $\mu ( B ) = \int_B \dens ( x ) \, \d x$,
	and let $\fH , n \in \N_0$ satisfy $n = \min \cu{ \fH , \scrK }$.
	Then there exist $\vartheta_0, \vartheta_1, \ldots, \vartheta_n \in \R^{\fd_\fH}$ 
	such that 
\begin{enumerate}[label=(\roman*)]
\item 
it holds for all $\width \in \cu{0, 1, \ldots, n}$ that $\vartheta_\width$ 
is a local minimum of $\cL_\fH$ and 
\item 
it holds that 
\begin{equation}
  \inf\limits_{\xi \in \R } \int_{[ a , b ] ^d } \abs{f(x) - \xi } ^2 \, \mu ( \d x ) \ge
  \cL_{ \fH }( \vartheta_0 ) > \cL_\fH ( \vartheta_1 ) > \cdots > \cL_\fH ( \vartheta _ n ) 
  .
\end{equation}
\end{enumerate} 
\end{cor}
\begin{proof}[Proof of \cref{cor:diff:risk:relu}]
	\Nobs that \cref{prop:global:exist} ensures for every $\width \in \cu{0, 1, \ldots, \fH}$ that there exists $\theta_\width \in \R^{\fd_{ \width }}$ which satisfies that $\cL_\width ( \theta_\width ) = \globinf_{ \width } $.
	Combining this with \cref{cor:different:risk:examples} completes the proof of \cref{cor:diff:risk:relu}.
\end{proof}

\subsection{No inactive neurons for small risk}

To conclude this section we establish as a consequence of the previous results that every parameter vector with a risk value sufficiently close to the global infimum value $\globinf_{\width}$ cannot have any inactive neurons.

\begin{prop} \label{prop:all:neurons:active}
	Assume \cref{setting:snn},
	assume that $\sigma$ is discriminatory for $ \mu $,
	and let $\width \in \N$,
	$\theta \in \R^{\fd_{\width - 1 } } $
	satisfy
	$\cL_{\width -1} ( \theta ) = \globinf_{\width - 1} > 0$.
	Then there exists $\varepsilon \in (0, \infty )$
	such that for all 
	$\vartheta \in \R^{\fd_{ \width } } $
	with
	$\cL_\width ( \vartheta ) < \varepsilon + \globinf_{ \width } $
	it holds that $\inact_\width ^\vartheta = \emptyset$.
\end{prop}

\begin{cproof} {prop:all:neurons:active}
	First, \nobs that \cref{prop:extra:neuron:improve} assures that $\globinf_{ \width } < \cL_{\width -1} ( \theta ) = \globinf_{\width - 1 } $.
	In the following let $\varepsilon \in (0, \infty )$ satisfy
	$\varepsilon = \globinf_{\width - 1 } - \globinf_{ \width }$.
	\Nobs that the fact that for all $\vartheta \in \R^{\fd_{ \width } }$ with $\inact_\width ^\vartheta \not= \emptyset$ there exists $\psi \in \R^{\fd_{\width - 1 } }$ which satisfies for all $x \in [a,b] ^d $ that $\realization{\vartheta} ( x ) = \realization{\psi} ( x ) $ demonstrates for all
	$\vartheta \in \R^{\fd_{ \width } }$ with $\inact_\width ^\vartheta \not= \emptyset$
	that
	$\cL_\width ( \vartheta ) \ge \globinf_{\width - 1 }$. Hence, we obtain
	for all 
	$\vartheta \in \R^{\fd_{ \width } } $
	with
	$\cL_\width ( \vartheta ) < \varepsilon + \globinf_{ \width }$
	 that $\inact_\width ^\vartheta = \emptyset$.
\end{cproof}

\begin{cor}
\label{cor:all:neurons:active:1d}
Assume \cref{setting:snn}, assume $ d = 1 $,
assume that $ f $ is Lipschitz continuous,
assume for all $ x \in \R $ that 
$ \sigma( x ) = \max\cu{ x, 0 } $, 
and let $ \width \in \N $ 
satisfy $ \globinf_{ \width - 1 } > 0 $.
Then there exists $ \varepsilon \in (0, \infty) $
such that for all 
$ \vartheta \in \R^{\fd_{ \width } } $
with
$ \cL_{ \width }( \vartheta ) < \varepsilon + \globinf_{ \width } $
it holds that 
$
  \inact_{ \width }^{ \vartheta } = \emptyset 
$.
\end{cor}

\begin{proof} [Proof of \cref{cor:all:neurons:active:1d}]
	\Nobs that \cref{prop:global:exist:1d} implies that there exists $\theta \in \R^{\fd_{\width - 1 } } $
	such that
	$\cL_{\width -1} ( \theta ) = \globinf_{\width - 1 } $.
	Combining this with \cref{prop:all:neurons:active} completes the proof of \cref{cor:all:neurons:active:1d}.
\end{proof}

We now reformulate \cref{cor:all:neurons:active:1d} under the assumption that $\scrK = \infty $,
which means that it is not possible achieve zero risk with any finite number of neurons.

\begin{cor}
	\label{cor:all:neurons:active:1d:v2}
Assume \cref{setting:snn}, assume $ d = 1 $,
assume 
$ 
  \scrK = \infty 
$, 
assume that $ f $ is Lipschitz continuous, 
assume for all $ x \in \R $ that 
$ \sigma( x ) = \max\cu{ x, 0 } $, 
and let $ \width \in \N $. 
Then there exists $ \varepsilon \in (0, \infty) $
such that for all 
$ \vartheta \in \R^{\fd_{ \width } } $
with
$ \cL_{ \width }( \vartheta ) < \varepsilon + \globinf_{ \width } $
it holds that 
$
  \inact_{ \width }^{ \vartheta } = \emptyset 
$.
\end{cor}
\begin{cproof}{cor:all:neurons:active:1d:v2}
	\Nobs that the assumption that $\scrK = \infty $ ensures that $\globinf_{\width-1} > 0$. \cref{cor:all:neurons:active:1d} therefore implies that for all 
	$ \vartheta \in \R^{\fd_{ \width } } $
	with
	$ \cL_{ \width }( \vartheta ) < \varepsilon + \globinf_{ \width } $
	it holds that 
	$
	\inact_{ \width }^{ \vartheta } = \emptyset 
	$.
\end{cproof}

\section{Non-convergence of SGD methods for ReLU and power ReLU ANNs}
\label{sec:nonc}

In this section we employ the results from \cref{sec:locmin} to show under suitable assumptions that the risk value of gradient-based optimization methods in the training of shallow ANNs does with high probability not converge to the global minimum value and, thereby, establish \cref{thm:sgd:shallow:nonc:multi_d_A_data_and_scientific}.

We also refer to Cheridito et al.~\cite{CheriditoJentzenRossmannek2020} for further non-convergence results for SGD in the training of deep ANNs.
Their result relies on ideas that are in some sense related to ours, in particular, the consideration of inactive neurons which yield zero output for any input vector $ x $ in the considered domain $[a,b] ^d$ (cf.~\cref{eq:setting_definition_of_I_set} in \cref{setting:snn}).
Sometimes these neurons are also referred to as \emph{dead}; 
cf., for example, Lu et al.~\cite{LuShinSu2020} and Shin \& Karniadakis~\cite{ShinKarniadakis2020}.

\subsection{Non-convergence of generalized gradient methods}
\label{subsec:nonc}

We first introduce the framework of generalized stochastic gradient methods which we employ during this section and which uses the notation introduced in \cref{setting:snn}.

\begin{setting}
	\label{setting:snn:sgd}
	Assume \cref{setting:snn},
	let $ ( \Omega , \cF , \P ) $ 
	be a probability space, 
	and for every $ m, n \in \N_0 $ 
	let 
	$ 
	X^m_n \colon \Omega \to [a,b]^d 
	$
	and 
	$
	Y^m_n \colon \Omega \to \R 
	$
	be random variables.
	For every $\width , n \in \N_0 $ 
	let $ M_n^\width \in \N $, 
	let 
	$
	\fL_n ^\width \colon \R^{\fd_{ \width } } \times \Omega \to \R
	$ 
	satisfy for all 
	$ \theta \in \R^{\fd_{ \width } } $ 
	that
	\begin{equation}
		\textstyle 
		\fL_n ^\width ( \theta ) 
		= 
		\frac{ 1 }{ M_n ^\width  } 
		\bigl[
		\sum_{ m = 1 }^{ M_n ^\width } 
		\abs{
		\realization{ \theta }( 
		X^m_n 
		) 
		- 
		Y^m_n }^2 
		\bigr] 
		,
	\end{equation}
	let 
	$
	\fG_n^\width 
	=
	( \fG^{ \width , 1 }_n , \ldots, \fG^{ \width , \fd_{ \width } }_n ) 
	\colon 
	\R^{ \fd_{ \width } } \times \Omega 
	\to \R^{ \fd_{ \width } }
	$ 
	satisfy for all
	$ i \in \num{ \width } $,
	$ 
	j \in 
	( 
	\cup_{ k = 1 }^d 
	\cu{ ( i - 1 ) d + k } 
	) 
	\cup \cu{ \width d + i } 
	$,
	$ 
	\theta = ( \theta_1, \dots, \theta_{ \fd_{ \width } } ) 
	\in 
	\R^{ \fd_{ \width } } 
	$,
	$
	\omega \in 
	\{ w \in \Omega \colon
	\R^{ d + 1 } \ni ( \psi_1, \ldots, \psi_{d+1} ) 
	\mapsto 
	\fL_n^\width (
	\theta_1 , \dots, \theta_{( i - 1 ) d}, 
	\psi_1, \ldots, \psi_d,
	\allowbreak 
	\theta_{i d + 1}, \ldots, 
	\allowbreak 
	\theta_{d \width + i - 1 }, 
	\allowbreak
	\psi_{d+1}, \allowbreak
	\theta_{d \width + i + 1 },    
	\dots, 
	\theta_{\fd_{ \width } } , w 
	) 
	\in \R \text{ is differentiable at } 
	\allowbreak 
	( \theta_{ ( i - 1 ) d + 1 } , \dots, \theta_{ i d } , \theta_{\width d + i } ) 
	\} 
	$
	\allowbreak 
	that
	\begin{equation}
		\label{setting:eq_def_generalized_grad}
		\fG^{\width , j }_n( \theta , \omega ) 
		= \rbr[\big]{ 
			\tfrac{ \partial }{ \partial \theta_j } 
			\fL_n ^\width 
		}( \theta , \omega ) 
		,
	\end{equation}
	let 
	$
	\Theta_n^\width 
	= 
	( \Theta_n^{ \width , 1 }, \dots, \Theta_n^{ \width , \fd_{ \width } } ) 
	\colon \Omega  \to \R^{\fd_{ \width } }
	$
	be a random variable, 
	and let 
	$
	\Phi_n ^\width 
	= 
	( 
	\Phi^{\width , 1 }_n, \dots, 
	\Phi^{ \width , \fd_{ \width } }_n 
	)
	\colon 
	\allowbreak
	( \R^{ \fd_{ \width } } )^{ n + 1 }
	\allowbreak
	\to 
	\R^{ \fd_{ \width } }
	$ 
	satisfy 
	for all 
	$
	g =
	( 
	( g_{ i, j } )_{ j \in \{ 1, 2, \dots, \fd_{ \width } \} }
	)_{
		i \in \{ 0, 1, \dots, n \}
	}
	\in 
	(
	\R^{ 
		\fd_{ \width }
	}
	)^{ n + 1 }
	$, 
	$ 
	j \in \num{ \fd_{ \width } } 
	$
	with 
	$
	\sum_{ i = 0 }^n
	\abs{ g_{ i, j } }
	= 0
	$
	that 
	$
	\Phi^{\width , j }_n( g ) = 0 
	$.
	And assume for all 
	$ \width , n \in \N_0 $
	that 
	\begin{equation}
		\label{setting:eq_general_SGD}
		\Theta_{ n + 1 } ^\width 
		= 
		\Theta_n ^\width - 
		\Phi_n^\width \bigl(
		\fG_0^\width ( \Theta_0^\width  ) ,
		\fG_1^\width ( \Theta_1^\width  ) ,
		\dots ,
		\fG_n^\width ( \Theta_n^\width )
		\bigr)
		.
	\end{equation}
\end{setting}

In \cref{setting:snn:sgd} for every $n , \width \in \N_0$ the positive integer $M_n^\width$ is the batch size used at the $n$-th training step for an ANN with $\width$ hidden neurons,
the function $\fL_n^\width \colon \R^{ \fd_\width } \times \Omega \to \R$ is the corresponding empirical risk function computed as the average $L^2$-error over $M_n^\width $ input-output data pairs,
and $\fG_n^\width \colon \R^{ \fd_\width } \times \Omega \to \R^{ \fd_\width }$ is a suitably generalized gradient of $\fL_n^\width $.
The processes $(\Theta_n^\width)_{n \in \N_0 } \colon \N_0 \times \Omega \to \R^{ \fd_\width }$, $\width \in \N_0$,
are general stochastic gradient optimization methods.
As we will see in \cref{subsec:adam}, the considered definition \cref{setting:eq_general_SGD} in particular encompasses the popular SGD and Adam optimization methods.

We now state in \cref{prop:sgd:shallow:nonc} one of our central results, which establishes an upper estimate on the probability that the considered optimization method converges to a global minimum of the risk function.
Note that for each of the $\width$ neurons its input depends on $d+1$ ANN parameters and we assume in \cref{prop:sgd:shallow:nonc} that these $(d+1)$-dimensional parameter vectors are i.i.d.~at initialization.
The idea of proof of \cref{prop:sgd:shallow:nonc} is, roughly speaking, that the risk can only converge to its global minimum value if no neuron is inactive due to \cref{prop:all:neurons:active},
and a neuron that is inactive at initialization can never become active since the corresponding gradient components are always zero.

\begin{theorem}[Non-convergence to global minimizers]
\label{prop:sgd:shallow:nonc}
Assume \cref{setting:snn:sgd},
let $ \width , k \in \N $, 
assume for all $ x \in \R $ that 
$ 
  \sigma( x ) = ( \max\cu{ x, 0 } )^k 
$,
assume that
$
  ( 
    \Theta^{ \width , ( i - 1 ) d + 1 }_0, 
    \Theta^{ \width , ( i - 1 ) d + 2 }_0 , 
    \dots, 
    \allowbreak 
    \Theta^{ \width ,  ( i - 1 ) d + d }_0 , 
    \allowbreak
    \Theta^{ \width , \width d + i }_0 
  ) \colon \Omega \to \R^{ d+1 }
$, 
$ i \in \num{\width} $, 
are i.i.d., 
and assume 
$ 
  \{ 
    \theta \in \R^{ \fd_{ \width - 1 } } 
    \colon 
    \cL_{ \width - 1 }( \theta ) 
    = \globinf_{ \width - 1 } 
    > 0 
  \} 
  \neq \emptyset 
$.
Then 
\begin{equation}
\label{eq:in_statement_estimate_probability}
  \P\Bigl(
    \liminf\limits_{ n \to \infty } 
    \cL_{ \width }( \Theta_n^\width ) 
    \leq 
    \globinf_{ \width } 
  \Bigr)
  \le 
  \exp\Bigl(
    -
    \width 
    \,
    \P\bigl(
      \ssum_{j=1}^d 
      \max\{ 
        \Theta_0^{ \width , j } a
        ,
        \Theta_0^{ \width , j } b
      \} 
      < 
      - \Theta_0^{ \width, \width d + 1 }
    \bigr)
  \Bigr)
  .
\end{equation}
\end{theorem}

\begin{cproof}{prop:sgd:shallow:nonc}
Throughout this proof for every 
$ i \in \num{\width} $
let $ \fU_{ i } \subseteq \R^{\fd_{ \width } } $
satisfy
\begin{equation}
\label{eq:in_proof_introduction_of_U_set}
\begin{split}
  \fU_{ i } 
& = 
  \bigl\{  
    \theta = ( \theta_1, \dots \theta_{ \fd_{ \width } } ) 
    \in \R^{\fd_{ \width } } 
    \colon  
    \bigl(
      \forall \, x \in [a,b]^d \colon 
      \theta_{ \width d + i }  + \ssum_{j=1}^d \theta _{ ( i - 1 ) d + j } x_j < 0 
    \bigr)
  \bigr\} 
\\ & =
\textstyle 
  \bigl\{  
    \theta = ( \theta_1, \dots \theta_{ \fd_{ \width } } ) 
    \in \R^{\fd_{ \width } } 
    \colon  
    \sup_{ x \in [a,b]^d }
    \bigl(
      \theta_{ \width d + i }  
      + 
      \ssum_{j=1}^d \theta _{ ( i - 1 ) d + j } x_j 
    \bigr)
    < 0 
  \bigr\} 
\\ & =
\textstyle 
  \bigl\{  
    \theta = ( \theta_1, \dots \theta_{ \fd_{ \width } } ) 
    \in \R^{\fd_{ \width } } 
    \colon  
    \bigl(
      \theta_{ \width d + i }  
      + 
      \ssum_{ j = 1 }^d 
      \bigl[ 
        \sup_{ x \in [a,b] }
        \theta _{ ( i - 1 ) d + j } x 
      \bigr]
    \bigr)
    < 0 
  \bigr\} 
\\ & =
\textstyle 
  \bigl\{  
    \theta = ( \theta_1, \dots \theta_{ \fd_{ \width } } ) 
    \in \R^{\fd_{ \width } } 
    \colon  
    \bigl(
      \theta_{ \width d + i }  
      + 
      \ssum_{ j = 1 }^d 
      \max_{ x \in \{ a, b \} }
      \bigl(
        \theta _{ ( i - 1 ) d + j } x
      \bigr)
    \bigr)
    < 0 
  \bigr\} 
  .
\end{split}
\end{equation}
\Nobs that \cref{eq:in_proof_introduction_of_U_set} 
assures that for all 
$ i \in \num{\width} $
it holds that $ \fU_{ i } $ is open. 
\Moreover \cref{eq:in_proof_introduction_of_U_set} 
shows that for all 
$ i \in \num{ \width } $,
$ \theta = ( \theta_1, \dots, \theta_{ \fd_{ \width } } ) \in \fU_{ i } $, 
$ x = ( x_1, \dots, x_d ) \in [a,b]^d $
we have that
\begin{equation}
\label{eq:in_proof_activation_evaluation_zero}
  \sigma( \theta_{ \width d + i }  + \ssum_{j=1}^d \theta _{ ( i - 1 ) d + j } x_j ) = 0
  .
\end{equation}
\Hence that for all
$ \omega \in \Omega $,
$ n \in \N_0 $, 
$ i \in \num{\width} $,
$ j \in ( \cup_{ k = 1 }^d \{ ( i - 1 ) d + k \} ) \cup \cu{ \width d + i } $,
$ \theta \in \fU_{ i } $ 
it holds that 
\begin{multline}
  \R^{ d + 1 } \ni ( \psi_1, \dots, \psi_{d+1} ) 
  \mapsto 
\\
  \fL_n^\width ( 
    \theta_1, \dots, \theta_{ (i - 1) d }, 
    \psi_1, \dots, \psi_d,
    \theta_{ i d + 1 }, \dots, \theta_{ d \width + i - 1 }, 
    \allowbreak
    \psi_{ d + 1 }, 
    \allowbreak
    \theta_{ d \width + i + 1 },
    \dots, 
    \theta_{ \fd_{ \width } } , \omega 
  ) \in \R
\end{multline}
is differentiable at 
$
  ( \theta_{ ( i - 1 ) d + 1 } , \ldots, \theta_{ id } , \theta_{\width d + i } ) 
$ 
with
$ 
  \rbr{ 
    \tfrac{\partial}{\partial \theta_{ j } } \fL_n^\width 
  }( \theta, \omega ) 
  = 0
$. 
This and \cref{setting:eq_def_generalized_grad} 
show for all 
$ n \in \N_0 $,
$ i \in \num{\width} $, 
$ j \in ( \cup_{ k = 1 }^d \{ ( i - 1 ) d + k \} ) \cup \cu{ \width d + i } $,
$ \omega \in \Omega $, 
$
  \theta \in \fU_{ i }
$
that
\begin{equation}
  \fG^{\width , j }_n( 
    \theta, \omega 
  ) 
  = 
  0
  .
\end{equation}
Combining this with 
the assumption that for all 
$ n \in \N_0 $, 
$
  g 
  =
  ( 
    ( g_{ i, j } )_{ j \in \{ 1, 2, \dots, \fd_{ \width } \} }
  )_{
    i \in \{ 0, 1, \dots, n \}
  }
  \in 
  (
    \R^{ 
      \fd_{ \width }
    }
  )^{ n + 1 }
$, 
$ 
  j \in \{ 1, 2, \dots, \fd_{ \width } \} 
$
with 
$
  \sum_{ i = 0 }^n
  \abs{ g_{ i, j } }
  = 0
$
it holds that 
$
  \Phi^{\width , j }_n( g ) = 0 
$
implies that for all 
$ n \in \N_0 $,
$ i \in \num{\width} $, 
$ j \in ( \cup_{ k = 1 }^d \{ ( i - 1 ) d + k \} ) \cup \cu{ \width d + i } $,
$ \omega \in \Omega $
with 
$
  ( 
    \cup_{ k = 0 }^n \{ 
      \Theta^{\width}_k ( \omega )
    \} 
  )
  \subseteq 
  \fU_{ i }
$
we have that 
\begin{equation}
  \Phi^{ \width , j }_n \rbr[\big]{ \fG^\width _0 (  \Theta^\width _0 ( \omega ), \omega ) 
    ,
    \fG^\width _1( 
      \Theta^\width _1 ( \omega ), \omega 
    ) 
    ,
    \dots 
    ,
    \fG^\width _n( 
      \Theta^\width _n ( \omega ), \omega 
    ) }
  = 0 . 
\end{equation}
This and \cref{setting:eq_general_SGD} 
show 
for all 
$ n \in \N_0 $,
$ i \in \num{\width} $, 
$ j \in ( \cup_{ k = 1 }^d \{ ( i - 1 ) d + k \} ) \cup \cu{ \width d + i } $,
$ \omega \in \Omega $
with 
$
  ( 
    \cup_{ k = 0 }^n \{ 
      \Theta^\width_k ( \omega )
    \} 
  )
  \subseteq 
  \fU_{ i }
$
that 
\begin{equation}
  \Theta^{\width , j }_{ n + 1 }( \omega ) 
  =
  \Theta^{ \width , j }_n( \omega ) 
  .
\end{equation}
Combining this with \cref{eq:in_proof_introduction_of_U_set} 
demonstrates for all 
$ n \in \N_0 $,
$ i \in \num{\width} $, $ \omega \in \Omega $
with 
$
  ( 
    \cup_{ k = 0 }^n \{ 
      \Theta_k^\width ( \omega )
    \} 
  )
  \subseteq 
  \fU_{ i }
$
 that 
\begin{equation}
  ( 
    \cup_{ k = 0 }^{ n + 1 } \{ 
      \Theta_k^\width ( \omega )
    \} 
  )
  \subseteq 
  \fU_{ i }
  .
\end{equation}
Induction \hence proves that 
for all 
$ n \in \N_0 $, 
$ i \in \num{\width} $, 
$ \omega \in \Omega $ 
with 
$ 
  \Theta_0^\width ( \omega ) 
  \in \fU_{ i } 
$
it holds that 
$
  ( 
    \cup_{ k = 0 }^n 
    \{ 
      \Theta_k^\width ( \omega ) 
    \} 
  ) 
  \subseteq 
  \fU_{ i }
$. 
\Hence
for all 
$ i \in \num{\width} $, 
$ \omega \in \Omega $ 
with 
$ 
  \Theta_0^\width ( \omega ) 
  \in \fU_{ i } 
$
that 
\begin{equation}
  ( 
    \cup_{ k = 0 }^{ \infty }
    \{ 
      \Theta_k^\width ( \omega ) 
    \} 
  ) 
  \subseteq 
  \fU_{ i }
  .
\end{equation}
\Hence that 
\begin{equation}
\label{eq:in_proof_inclusion_Omega}
  \bigl\{ 
    \omega \in \Omega \colon
    \bigl(
      \Theta_0^\width ( \omega ) 
      \in 
      ( \cup_{ i = 1 }^{ \width } \fU_{ i } )
    \bigr)
  \bigr\} 
  \subseteq 
  \bigl\{ 
    \omega \in \Omega 
    \colon
    \bigl(
      \forall \, n \in \N_0 \colon
      \Theta_n^\width ( \omega ) 
      \in
      ( \cup_{ i = 1 }^{ \width } \fU_{ i } )
    \bigr)
  \bigr\}
  .
\end{equation}
\Moreover 
\cref{eq:setting_definition_of_I_set}
and 
\cref{eq:in_proof_activation_evaluation_zero} 
show that 
for all 
$ i \in \num{\width} $, 
$ \theta \in \fU_{ i } $
it holds that 
$
  i \in \inact_{ \width }^{ \theta }
$. 
\Hence that for all 
$ 
  \theta \in 
  ( \cup_{ i = 1 }^{ \width } \fU_{ i } ) 
$ 
we have that 
$
  \inact_{ \width }^{ \theta } \not= \emptyset
$. 
Combining this 
and \cref{eq:in_proof_inclusion_Omega} 
with \cref{prop:all:neurons:active} 
and the assumption that 
$
  \{ 
    \theta \in \R^{ \fd_{ \width - 1 } } 
    \colon 
    \cL_{ \width - 1 }( \theta ) 
    = \globinf_{ \width - 1 } 
    > 0 
  \} 
  \neq \emptyset 
$
assures that 
there exists $ \varepsilon \in (0, \infty) $ such that 
\begin{equation}
\label{eq:in_proof_probab_liminf_estimate}
\begin{split}
&
  \P\bigl(
    \liminf\nolimits_{ n \to \infty} \cL_{ \width }( \Theta_n^\width ) 
    > \globinf_{ \width } 
  \bigr)
\\ & 
\ge 
  \P\bigl(
    \liminf\nolimits_{ n \to \infty} 
    \cL_{ \width }( \Theta_n^\width ) 
    \ge \globinf_{ \width } + \varepsilon 
  \bigr)
\ge 
  \P\bigl(
    \forall \, n \in \N_0 \colon 
    \cL_{ \width }( \Theta_n^\width ) 
    \ge 
    \globinf_{ \width } + \varepsilon 
  \bigr)
\\ &
\ge 
  \P\bigl(
    \forall \, n \in \N_0 \colon 
    \inact_{ \width }^{ 
      \Theta_n ^\width
    } 
    \not= \emptyset 
  \bigr)
\ge 
  \P\bigl(
    \forall \, n \in \N_0 \colon 
    \Theta_n^\width \in 
    ( 
      \cup_{ i = 1 }^{ \width }
      \fU_{ i } 
    )
  \bigr)
\ge 
  \P\bigl(
    \Theta_0 ^\width
    \in 
    ( 
      \cup_{ i = 1 }^{ \width } 
      \fU_{ i } 
    )
  \bigr)
  .
\end{split}
\end{equation}
\Moreover 
\cref{eq:in_proof_introduction_of_U_set} and 
the assumption that 
$
  ( 
    \Theta_0^{ \width , ( i - 1 ) d + 1 }, 
    \allowbreak 
    \Theta_0^{ \width , ( i - 1 ) d + 2 } , \dots,
    \allowbreak 
    \Theta_0^{ \width , ( i - 1 ) d + d } , \allowbreak 
    \Theta_0^{ \width , \width d + i } 
  ) 
$,
$ i \in \num{\width} $,
are i.i.d.~implies 
that 
\begin{equation}
\begin{split}
&
  \P\bigl( 
    \Theta_0 ^\width
    \in 
    ( 
      \textstyle\cup_{ i = 1 }^{ \width } 
      \fU_{ i } 
    )
  \bigr) 
  =
  \P\bigl( 
    \cup_{ i = 1 }^{ \width } 
    \{ 
      \Theta_0 ^\width
      \in 
      \fU_{ i } 
    \}
  \bigr)
  =
  1 
  -
  \P\bigl( 
    \Omega 
    \backslash
    (
      \textstyle\cup_{ i = 1 }^{ \width } 
      \{ 
        \Theta_0 ^\width
        \in 
        \fU_{ i } 
      \}
    )
  \bigr)
\\ &
  =
  1 
  -
  \P\bigl( 
    \cap_{ i = 1 }^{ \width } 
    (
      \Omega 
      \backslash
      \{ 
        \Theta_0 ^\width
        \in 
        \fU_{ i } 
      \}
    )
  \bigr)
  =
  1 
  -
  \textstyle 
  \prod_{ i = 1 }^{ \width }
  \P\bigl( 
      \Omega 
      \backslash
      \{ 
        \Theta_0 ^\width
        \in 
        \fU_{ i } 
      \}
  \bigr)
\\ &
  =
  1 
  -
  \bigl[
    \P\bigl( 
      \Omega 
      \backslash
      \{ 
        \Theta_0 ^\width
        \in 
        \fU_{ 1 } 
      \}
    \bigr)
  \bigr]^{ \width }
  =
  1 
  -
  \bigl[
    1 
    -
    \P( 
      \Theta_0 ^\width
      \in 
      \fU_{ 1 } 
    )
  \bigr]^{ \width } .
\end{split}
\end{equation}
The fact that 
$
  1 
  -
  \P( 
    \Theta_0 ^\width
    \in 
    \fU_{ 1 } 
  )
  \ge 0
$
and the fact that for all $ x \in \R $
it holds that 
$
  1 - x \leq \exp( - x )
$
\hence 
show that 
\begin{equation}
\begin{split}
  \P\bigl( 
    \Theta_0 ^\width
    \in 
    ( 
      \textstyle\cup_{ i = 1 }^{ \width } 
      \fU_{ i } 
    )
  \bigr) 
  \geq
  1 
  -
  \bigl[
    \exp(
      -
      \P( 
        \Theta_0 ^\width
        \in 
        \fU_{ 1 } 
      )
    )
  \bigr]^{ \width }
  =
  1 
  -
  \exp\bigl(
    -
    \width 
    \,
    \P( 
      \Theta_0 ^\width
      \in 
      \fU_{ 1 } 
    )
  \bigr)
  .
\end{split}
\end{equation}
This and \cref{eq:in_proof_probab_liminf_estimate} demonstrate that 
\begin{equation}
\begin{split} 
  \P\bigl(
    \liminf\nolimits_{ n \to \infty} 
    \cL_{ \width }( \Theta_n^\width ) 
    \leq \globinf_{ \width } 
  \bigr)
& 
  =
  1 -
  \P\bigl(
    \liminf\nolimits_{ n \to \infty} 
    \cL_{ \width }( \Theta_n^\width ) 
    > \globinf_{ \width } 
  \bigr)
\leq 
  1 -
  \P\bigl( 
    \Theta_0 ^\width
    \in 
    ( 
      \textstyle\cup_{ i = 1 }^{ \width } 
      \fU_{ i } 
    )
  \bigr) 
\\ &
\leq 
  \exp\bigl(
    -
    \width 
    \,
    \P( 
      \Theta_0 ^\width
      \in 
      \fU_{ 1 } 
    )
  \bigr) ,
\end{split}
\end{equation}
which establishes \cref{eq:in_statement_estimate_probability}. 
\end{cproof}

As a consequence we now show in \cref{cor:sgd:shallow:nonc} under suitable assumptions on the random initialization that the probability of convergence to a global minimum tends to zero exponentially fast as the number of hidden neurons increases to infinity.

\begin{cor}
\label{cor:sgd:shallow:nonc}
Assume \cref{setting:snn:sgd},
let $ k \in \N $, 
assume for all $ x \in \R $ that 
$ 
  \sigma( x ) = ( \max\cu{ x, 0 } )^k 
$,
assume for all $ \width \in \N $ that 
$
  ( 
    \Theta^{ \width, ( i - 1 ) d + 1 }_0, 
    \Theta^{ \width, ( i - 1 ) d + 2 }_0 , 
    \dots, 
    \allowbreak 
    \Theta^{ \width, ( i - 1 ) d + d }_0 , 
    \Theta^{ \width, \width d + i }_0 
  ) 
$, 
$ i \in \num{\width} $, 
are i.i.d., 
assume 
\begin{equation}
	\label{prop:sgd:shallow:nonc:eq:init}
	\textstyle 
	\liminf_{ \width \to \infty }
	\P \rbr[\big]{
	\ssum_{j=1}^d 
	\max\{ 
	\Theta_0^{ \width, j } a
	,
	\Theta_0^{ \width, j } b
	\} 
	< 
	- \Theta_0^{ \width, \width d + 1 } }
> 0 ,
\end{equation} 
and assume for all $\width \in \N$ that
$ 
  \{ 
    \theta \in \R^{ \fd_{ \width - 1 } } 
    \colon 
    \cL_{ \width - 1 }( \theta ) 
    = \globinf_{ \width - 1 } 
    > 0 
  \} 
  \neq \emptyset 
$.
Then there exists $ \delta \in (0, \infty) $ 
which satisfies for all $ \width \in \N $ that
\begin{equation}
  \P\bigl(
    \liminf\nolimits_{ n \to \infty } 
    \cL_{ \width }( \Theta_n^{ \width } ) 
    \leq 
    \globinf_{ \width } 
  \bigr)
\le 
  \delta^{ - 1 } 
  \exp( - \delta \width ) 
  .
\end{equation}
\end{cor}

\begin{cproof}{cor:sgd:shallow:nonc}
	\Nobs that \cref{prop:sgd:shallow:nonc}
	demonstrates for all
	$\width \in \N$ that
	\begin{equation}
		\label{cor:sgd:shallow:eq:proof_1}
		\P \rbr[\big]{
		\liminf\nolimits_{ n \to \infty } 
		\cL_{ \width }( \Theta_n^{ \width } ) 
		\leq 
		\globinf_{ \width } }
		\le \exp \rbr*{ - \width \P \rbr[\big]{
				\ssum_{j=1}^d 
				\max\{ 
				\Theta_0^{ \width, j } a
				,
				\Theta_0^{ \width, j } b
				\} 
				< 
				- \Theta_0^{ \width, \width d + 1 } } } .
	\end{equation}
Furthermore, \cref{prop:sgd:shallow:nonc:eq:init} ensures that there exist $\eta \in (0 , \infty )$,
$H \in \N$
which satisfy for all $\width \in \N \cap [H , \infty )$ that
\begin{equation}
	\label{cor:sgd:shallow:eq:proof_2}
	\P \rbr[\big]{
		\ssum_{j=1}^d 
		\max\{ 
		\Theta_0^{ \width, j } a
		,
		\Theta_0^{ \width, j } b
		\} 
		< 
		- \Theta_0^{ \width, \width d + 1 } } \ge \eta .
\end{equation}
Next, let $\delta \in (0 , \infty )$ satisfy $\delta = \min \cu{1 , \eta , e^{ - \eta H } }$.
\Nobs that for all $\width \in \N \cap [ 1 , H )$ we have that
\begin{equation}
	\P \rbr[\big]{
		\liminf\nolimits_{ n \to \infty } 
		\cL_{ \width }( \Theta_n^{ \width } ) 
		\leq 
		\globinf_{ \width } } \le 1 \le \delta^{-1} e^{ - \eta H } \le \delta^{-1} e^{ - \delta \width } .
\end{equation}
Moreover,
\cref{cor:sgd:shallow:eq:proof_1} and \cref{cor:sgd:shallow:eq:proof_2} imply for all $\width \in \N \cap [H , \infty )$
that
\begin{equation}
		\P \rbr[\big]{
		\liminf\nolimits_{ n \to \infty } 
		\cL_{ \width }( \Theta_n^{ \width } ) 
		\leq 
		\globinf_{ \width } }
	\le e^{ - \eta \width } \le e^{ - \delta \width } \le \delta^{-1} e^{ - \delta \width }.
\end{equation}
\end{cproof}

The next goal is to establish in \cref{cor:sgd:shallow:data_loss} a more general version of \cref{cor:sgd:shallow:nonc}
where the risk function is defined as the expectation with respect to arbitrary input-output data pairs $(X_0^0 , Y_0^0)$. 
The previous result, \cref{cor:sgd:shallow:nonc}, is equivalent to the special case that $Y_0^0$ is equal to a deterministic function $f(X_0^0)$ of $X_0^0$.
Furthermore, \cref{cor:sgd:shallow:data_loss} will reveal that the assumption on the initialization in \cref{prop:sgd:shallow:nonc:eq:init} is, for instance, satisfied for standard normal or uniform initializations.

To prove \cref{cor:sgd:shallow:data_loss} we combine \cref{cor:sgd:shallow:nonc} with \cref{lem:l2:decomp}, which is essentially 
the well-known $L^2$-error decomposition (cf., e.g., Beck et al.~\cite[Lemma 4.1]{Beck2022} or Cucker \& Smale~\cite[Proposition 1]{Cucker2001}).

\begin{lemma}
	\label{lem:l2:decomp}
	Let $(\Omega , \cF , \P)$ be a probability space,
	let $(S, \cS)$ be a measurable space,
	let $\fd \in \N$,
	for every $\theta \in \R ^\fd$ let $N^\theta \colon S \to \R$ be a measurable function,
	let $X \colon \Omega \to S$ and $Y \colon \Omega \to \R$ be random variables,
	assume $\E [ \abs{ Y }^2 ] < \infty $,
	let $f \colon S \to \R$ be measurable,
	let $\cL \colon \R^\fd \to \R$ and $\bbL \colon \R^\fd \to \R$ satisfy for all $\theta \in \R^\fd$ that $\bbL ( \theta ) = \E [ \abs{N^\theta ( X ) - Y } ^2 ]$
	and $\cL ( \theta ) = \E [ \abs{ N^\theta ( X ) - f ( X ) } ^2 ]$,
	and assume $\P$-a.s.~that $\E [ Y | X ] = f(X)$.
	Then
	\begin{enumerate}[label = (\roman*)]
		\item \label{lem:l2:decomp:item1} 
		it holds for all $\theta \in \R^\fd$ that
		\begin{equation}
			\bbL ( \theta ) = \cL ( \theta ) + \E \br[\big]{\abs{f(X) - Y }^2}
		\end{equation}
	and
	\item \label{lem:l2:decomp:item2}
	it holds for all $\theta \in \R ^\fd$ that $( \cL ( \theta ) = \inf_{\vartheta \in \R ^\fd } \cL ( \vartheta ) \Longleftrightarrow \bbL ( \theta ) = \inf_{\vartheta \in \R^\fd } \bbL ( \vartheta ) )$.
	\end{enumerate}
\end{lemma}

\begin{cproof}{lem:l2:decomp}
	\Nobs that, e.g., Beck et al.~\cite[Lemma 4.1]{Beck2022}
	ensures for all
	$\theta \in \R^{ \fd }$
	that
	\begin{equation}
		\begin{split}
			\bbL ( \theta ) & 
			= \E \br[\big]{\abs{ N^{\theta} ( X ) - Y } ^2 } 
			= \E \br[\big]{\abs{ N^{\theta} ( X ) - \E [ Y | X ] } ^2 }
			+ \E \br[\big]{ \abs{ \E [ Y | X ] - Y } ^2} \\
			&= \E \br[\big]{\abs{ N^{\theta} ( X ) - f ( X ) } ^2 } + \E \br[\big]{\abs{f(X) - Y }^2}
			= \cL ( \theta ) + \E \br[\big]{\abs{f(X) - Y }^2} .
		\end{split}
	\end{equation}
This establishes \cref{lem:l2:decomp:item1}.
In addition, \nobs that \cref{lem:l2:decomp:item1} implies \cref{lem:l2:decomp:item2}.
\end{cproof}

\begin{cor}
	\label{cor:sgd:shallow:data_loss}
	Assume \cref{setting:snn:sgd},
	let $k \in \N$,
	assume for all $x \in \R$ that $\sigma ( x ) = ( \max \cu{x , 0 } ) ^k $,
	for every $ \width \in \N $ 
	let
	$ \bbL^{ \width } \colon \R^{ \fd_{ \width } } \to \R $
	satisfy for all 
	$ 
	\theta \in \R^{ \fd_{ \width } } 
	$
	that
	\begin{equation}
		\textstyle 
		\bbL^{ \width }( \theta ) 
		= 
		\E \br[\big]{ \abs{ 
		\realization{\theta}( X^0_0 )
		-
		Y^0_0
	} ^2 }
		,
	\end{equation}
	assume $ \P $-a.s.\ that 
	$
	\E[ Y^0_0 | X^0_0 ] = f( X^0_0 )
	$, 
	assume for all $B \in \cB ( [ a , b ] ^ d ) $ that $\P ( X_0^0 \in B ) = \mu ( B )$,
assume for all $ \width \in \N $ that 
$
\Theta^{ \width, i }_0
$, 
$ i \in \num{ \width d + \width } $, 
are independent,
let $ \Dens \colon \R \to [0,\infty) $ be measurable, 
let $\eta \in (0 , \infty )$ satisfy $ (- \eta , \eta ) \subseteq \Dens ^{-1} ( ( 0 , \infty ) )$,
let $ \kappa \in \R $ satisfy for all 
$
\width \in \N
$,
$
i \in \num{ \width d + \width }
$,
$ x \in \R $ 
that
\begin{equation}
	\label{cor:sgd:shallow:data:eq:init}
	\P \rbr[\big]{ \width^{ \kappa }
		\Theta^{ \width, i }_0 
		< x }
	=
	\int_{ - \infty }^x \Dens(y) \, \d y ,
\end{equation} 
and assume that for all $\width \in \N$ there exists $\vartheta \in \R^{ \fd_\width }$ such that
\begin{equation}
	\label{cor:sgd:shallow:data:eq:existence}
	\bbL ^\width ( \vartheta ) = \inf\nolimits_{\theta \in \R^{ \fd_\width } } \bbL^\width ( \theta ) 
		> \E \br[\big]{ \abs{ f( X^0_0 ) - Y^0_0 }^2 } .
\end{equation}
Then there exists $\delta \in (0 , \infty )$
which satisfies for all $\width \in \N$ that
\begin{equation}
	\P \rbr[\big]{ \liminf\nolimits_{ n \to \infty } \bbL^\width ( \Theta_n^\width ) = \inf\nolimits_{\theta \in \R^{ \fd_\width } } \bbL^\width ( \theta ) } \le \delta^{-1} e^{ - \delta \width } .
\end{equation}
\end{cor}

\begin{cproof}{cor:sgd:shallow:data_loss}
	First, \nobs that \cref{cor:sgd:shallow:data:eq:init} and the assumption that for all $ \width \in \N $ it holds that 
	$
	\Theta^{ \width, i }_0
	$, 
	$ i \in \num{\width ( d + 1 ) } $, 
	are independent ensures that for all $\width \in \N$ we have that $
	( 
	\Theta^{ \width, ( i - 1 ) d + 1 }_0, 
	\Theta^{ \width, ( i - 1 ) d + 2 }_0 , 
	\dots, 
	\allowbreak 
	\Theta^{ \width, ( i - 1 ) d + d }_0 , 
	\Theta^{ \width, \width d + i }_0 
	) 
	$, 
	$ i \in \num{\width} $, 
	are i.i.d.
	Furthermore, \nobs that \cref{cor:sgd:shallow:data:eq:init}
	demonstrates for all $\width \in \N$ that
	\begin{equation}
		\begin{split}
			& \P \rbr[\big]{\ssum_{j=1}^d \max \cu{\Theta_0^{ \width , j } a , \Theta_0^{ \width , j } b } < - \Theta_0 ^{ \width , \width d + 1 } } \\
			&=  \P \rbr[\big]{\ssum_{j=1}^d \max \cu{\width ^\kappa\Theta_0^{ \width , j } a , \width ^\kappa \Theta_0^{ \width , j } b } < - \width ^\kappa \Theta_0 ^{ \width , \width d + 1 } } \\
			& \ge \P \rbr[\big]{ \br[\big]{ \width ^\kappa \Theta_0^{ \width , \width d + 1 } < - \tfrac{\eta}{2} } \wedge \br[\big]{\abs[\big]{\ssum_{j=1}^d \max \cu{\width ^\kappa \Theta_0^{ \width , j } a , \width ^\kappa \Theta_0^{ \width , j } b } } < \tfrac{\eta}{2} } } \\
			& \ge \P \rbr[\big]{ \br[\big]{ \width ^\kappa \Theta_0^{ \width , \width d + 1 } \in (- \eta, - \tfrac{\eta}{2} ) } \wedge \br[\big]{\textstyle\max_{j=1}^d \max \cu{ \abs{ \width ^\kappa \Theta_0^{ \width , j } a  } , \abs{ \width ^\kappa \Theta_0^{ \width , j } b } } } < \tfrac{\eta}{2 d } } \\
			& \ge \P \rbr[\big]{ \width ^\kappa \Theta_0^{ \width , \width d + 1 } \in (- \eta, - \tfrac{\eta}{2} ) } \br*{ \textstyle\prod_{j=1}^d \P\rbr[\big]{\abs{ \width ^\kappa \Theta_0^{ \width , j } < \tfrac{\eta}{2 d \max \cu{\abs{a} , \abs{b} } } } } }
			\\
			&= \rbr*{ \int_{-\eta}^{ -\eta/2} \Dens ( x ) \, \d x }
			\rbr*{ \int_{- \tfrac{\eta}{2 d \max \cu{\abs{a} , \abs{b} } } }^{\tfrac{\eta}{2 d \max \cu{\abs{a} , \abs{b} } } } \Dens ( x ) \, \d x } ^d > 0.
		\end{split}
	\end{equation}
Therefore, we obtain that
\begin{equation}
	\label{cor:sgd:shallow:data:eq:proof_1}
	\liminf\nolimits_{ \width \to \infty } \P \rbr[\big]{\ssum_{j=1}^d \max \cu{\Theta_0^{ \width , j } a , \Theta_0^{ \width , j } b } < - \Theta_0 ^{ \width , \width d + 1 } } > 0.
\end{equation}
	Next, \nobs that \cref{lem:l2:decomp}
and \cref{cor:sgd:shallow:data:eq:existence} demonstrate that for all $\width \in \N$
there exists $\vartheta \in \R^{ \fd_\width }$
such that $\cL_\width ( \vartheta ) = \inf_{\theta \in \R^{\fd_\width } } \cL_\width ( \theta ) > 0$.
This, \cref{cor:sgd:shallow:data:eq:proof_1}, \cref{cor:sgd:shallow:nonc}, and \cref{lem:l2:decomp} show that there exists $\delta \in (0 , \infty )$
which satisfies for all $\width \in \N$ that
\begin{equation}
	\begin{split}
		&	\delta^{-1} e^{ - \delta \width }
		 \le \P \rbr[\big]{ \liminf\nolimits_{ n \to \infty } \cL_\width ( \Theta_n^\width ) = \inf\nolimits_{\theta \in \R^{ \fd_\width } } \cL_\width ( \theta ) } \\
		&= \P \rbr[\big]{ \liminf\nolimits_{ n \to \infty } \rbr[\big]{ \cL_\width ( \Theta_n^\width ) + \E \br[\big]{ \abs{ f( X^0_0 ) - Y^0_0 }^2 } } = \inf\nolimits_{\theta \in \R^{ \fd_\width } } \rbr[\big]{ \cL_\width ( \theta ) + \E \br[\big]{ \abs{ f( X^0_0 ) - Y^0_0 }^2 } } } \\
		&= \P \rbr[\big]{ \liminf\nolimits_{ n \to \infty } \bbL^\width ( \Theta_n^\width ) = \inf\nolimits_{\theta \in \R^{ \fd_\width } } \bbL^\width ( \theta ) }.
	\end{split}
\end{equation}
\end{cproof}

\subsection{Non-convergence of generalized gradient methods for ReLU ANNs}

In this subsection we verify the assumption in \cref{cor:sgd:shallow:data_loss} that global minima of the risk function exist for ReLU ANNs in two particular cases.
First, we employ \cref{prop:global:exist} to show this under a continuity assumption on the target function $f$ and suitable regularity assumptions on the distribution of the input data $X_0^0$.

\begin{cor}
	\label{cor:sgd:shallow:multi_dim}
	Assume \cref{setting:snn:sgd},
	assume for all $x \in \R$ that $\sigma ( x ) = \max \cu{x , 0 }$,
for every $ \width \in \N $ 
let
$ \bbL^{ \width } \colon \R^{ \fd_{ \width } } \to \R $
satisfy for all 
$ 
\theta \in \R^{ \fd_{ \width } } 
$
that
\begin{equation}
	\textstyle 
	\bbL^{ \width }( \theta ) 
	= 
	\E \br[\big]{ \abs{ 
			\realization{\theta}( X^0_0 )
			-
			Y^0_0
		} ^2 }
	,
\end{equation}
assume that $f$ is continuous,
assume $ \P $-a.s.\ that 
$
\E[ Y^0_0 | X^0_0 ] = f( X^0_0 )
$, 
let $\dens \colon \R^d  \to [0 , \infty )$ be continuous,
assume $\dens^{-1} ( ( 0 , \infty ) ) = ( a , b )^d $,
assume for all $B \in \cB ( [ a , b ]^d )$ that $\P ( X_0^0 \in B ) = \mu ( B ) = \int_B \dens ( x ) \, \d x $,
assume for all $ \width \in \N $ that 
$
\Theta^{ \width, i }_0
$, 
$ i \in \num{ \width d + \width } $, 
are independent,
let $ \Dens \colon \R \to [0,\infty) $ be measurable, 
let $\eta \in (0 , \infty )$ satisfy $ (- \eta , \eta ) \subseteq \Dens ^{-1} ( ( 0 , \infty ) )$,
let $ \kappa \in \R $ satisfy for all 
$
\width \in \N
$,
$
i \in \num{ \width d + \width }
$,
$ x \in \R $ 
that
\begin{equation}
	\label{cor:sgd:shallow:multidim:eq:init}
	\P \rbr[\big]{ \width^{ \kappa }
		\Theta^{ \width, i }_0 
		< x }
	=
	\int_{ - \infty }^x \Dens(y) \, \d y ,
\end{equation} 
and assume $ 
\E[ \abs{ f( X^0_0 ) - Y^0_0 }^2 ]
\notin 
\bbL^{ \width }( \R^{ \fd_{ \width } } )
$.
Then there exists $\delta \in (0 , \infty )$
which satisfies for all $\width \in \N$ that
\begin{equation}
	\label{cor:sgd:shallow:multidim:eq_claim}
	\P \rbr[\big]{ \liminf\nolimits_{ n \to \infty } \bbL^\width ( \Theta_n^\width ) = \inf\nolimits_{\theta \in \R^{ \fd_\width } } \bbL^\width ( \theta ) } \le \delta^{-1} e^{ - \delta \width } .
\end{equation}	
\end{cor}

\begin{cproof}{cor:sgd:shallow:multi_dim}
	\Nobs that \cref{prop:global:exist} demonstrates that for all $\width \in \N$
	there exists $\vartheta \in \R^{ \fd_\width }$
	such that $\cL_\width ( \vartheta ) = \inf_{\theta \in \R^{\fd_\width } } \cL_\width ( \theta )$.
	Combining this with \cref{lem:l2:decomp} shows that for all $\width \in \N$ there exists $\vartheta \in \R^{ \fd_\width }$
	such that $\bbL^\width ( \vartheta ) = \inf_{\theta \in \R^{ \fd_\width } } \bbL^\width ( \theta ) > \E[ \abs{ f( X^0_0 ) - Y^0_0 }^2 ] $.
	\cref{cor:sgd:shallow:data_loss} therefore establishes \cref{cor:sgd:shallow:multidim:eq_claim}.
\end{cproof}

Similarly, in the special case of a one-dimensional input the assumption that global minima of the risk function exist in \cref{cor:sgd:shallow:data_loss} is true whenever the function $f$ is Lipschitz due to \cref{prop:global:exist:1d}, as we show in the next result.

\begin{cor}
	\label{cor:sgd:shallow:1d}
	Assume \cref{setting:snn:sgd},
	assume $d=1$,
	assume for all $x \in \R$ that $\sigma ( x ) = \max \cu{x , 0 }$,
	for every $ \width \in \N $ 
	let
	$ \bbL^{ \width } \colon \R^{ \fd_{ \width } } \to \R $
	satisfy for all 
	$ 
	\theta \in \R^{ \fd_{ \width } } 
	$
	that
	\begin{equation}
		\textstyle 
		\bbL^{ \width }( \theta ) 
		= 
		\E \br[\big]{ \abs{ 
				\realization{\theta}( X^0_0 )
				-
				Y^0_0
			} ^2 }
		,
	\end{equation}
	assume that $f$ is Lipschitz continuous,
	assume $ \P $-a.s.\ that 
	$
	\E[ Y^0_0 | X^0_0 ] = f( X^0_0 )
	$, 
	assume for all $B \in \cB ( [ a , b ] )$ that $\P ( X_0^0 \in B ) = \mu ( B )$,
	assume for all $ \width \in \N $ that 
	$
	\Theta^{ \width, i }_0
	$, 
	$ i \in \num{ 2 \width } $, 
	are independent,
	let $ \Dens \colon \R \to [0,\infty) $ be measurable, 
	let $\eta \in (0 , \infty )$ satisfy $ (- \eta , \eta ) \subseteq \Dens ^{-1} ( ( 0 , \infty ) )$,
	let $ \kappa \in \R $ satisfy for all 
	$
	\width \in \N
	$,
	$
	i \in \num{ 2 \width }
	$,
	$ x \in \R $ 
	that
	\begin{equation}
		\label{cor:sgd:shallow:1d:eq:init}
		\P \rbr[\big]{ \width^{ \kappa }
			\Theta^{ \width, i }_0 
			< x }
		=
		\int_{ - \infty }^x \Dens(y) \, \d y ,
	\end{equation} 
	and assume $ 
	\E[ \abs{ f( X^0_0 ) - Y^0_0 }^2 ]
	\notin 
	\bbL^{ \width }( \R^{ \fd_{ \width } } )
	$.
	Then there exists $\delta \in (0 , \infty )$
	which satisfies for all $\width \in \N$ that
	\begin{equation}
		\label{cor:sgd:shallow:1d:eq_claim}
		\P \rbr[\big]{ \liminf\nolimits_{ n \to \infty } \bbL^\width ( \Theta_n^\width ) = \inf\nolimits_{\theta \in \R^{ \fd_\width } } \bbL^\width ( \theta ) } \le \delta^{-1} e^{ - \delta \width } .
	\end{equation}	
\end{cor}

\begin{cproof}{cor:sgd:shallow:1d}
	\Nobs that \cref{prop:global:exist:1d} demonstrates that for all $\width \in \N$
	there exists $\vartheta \in \R^{ \fd_\width }$
	such that $\cL_\width ( \vartheta ) = \inf_{\theta \in \R^{\fd_\width } } \cL_\width ( \theta )$.
	\cref{lem:l2:decomp} hence implies that for all $\width \in \N$ there exists $\vartheta \in \R^{ \fd_\width }$
	such that $\bbL^\width ( \vartheta ) = \inf_{\theta \in \R^{ \fd_\width } } \bbL^\width ( \theta ) > \E[ \abs{ f( X^0_0 ) - Y^0_0 }^2 ] $.
	Combining this with \cref{cor:sgd:shallow:data_loss} establishes \cref{cor:sgd:shallow:1d:eq_claim}.
\end{cproof}

\subsection{SGD and Adam optimizers as generalized gradient methods}
\label{subsec:adam}

In this subsection we show that the framework considered in \cref{setting:snn:sgd}
is sufficiently general to include in particular the standard SGD optimization method as well as the popular Adam optimizer.
We first establish this property for the standard SGD method with momentum.

\begin{lemma}[Momentum SGD method]
	\label{lem:sgd:property}
	Let $\fd \in \N$,
	$( \alpha_n ) _{n \in \N_0 } \subseteq  [0 , 1 ]$,
	$(\gamma_n)_{n \in \N _0 } \subseteq [0 , \infty )$,
	let $(\Omega , \cF , \P )$ be a probability space,
	for every $n \in \N_0$ let $\fG_n \colon \R^\fd \times \Omega \to \R^\fd$ be measurable,
	and let
	$\bfm \colon \N_0 \times \Omega \to \R^\fd$
	and $\Theta \colon \N_0 \times \Omega \to \R^\fd$
	satisfy for all
	$n \in \N_0$
	 that $\bfm_0 = 0$,
	 \begin{equation}
	 	\label{lem:sgd:theta:recursion}
	 	\begin{split}
	 			\bfm_{n+1} & = \alpha_n \bfm_n + (1 - \alpha_n ) \fG_n ( \Theta_n ), \\
	 			\andq 	\Theta_{n+1} &= \Theta_{n} - \gamma_n \bfm_{n+1}. 
	 	\end{split} 
	 \end{equation}
	Then there exist functions $\Phi_n = (\Phi_n^1 , \ldots , \Phi_n^\fd ) \colon (\R^\fd ) ^{ n + 1 } \to \R^\fd$, $n \in \N_0$,
	which satisfy for all $n \in \N_0$ that
	\begin{equation}
		\Theta_{n+1} = \Theta_n - \Phi_n ( \fG_0 ( \Theta_0 ) , \fG_1 ( \Theta_1 ) , \ldots, \fG_n ( \Theta_n ) )
	\end{equation}
	and
	which satisfy for all $n \in \N_0$,
	$G = ((G_{i,j} )_{j \in \num{\fd} })_{i \in \cu{0, 1, \ldots , n } } \in (\R^\fd ) ^{ n + 1 }$,
	$j \in \num{\fd}$ with
	$\sum_{i=0}^n \abs{G_{i,j} } = 0$ that $\Phi_n^j ( G ) = 0$.
\end{lemma}

\Nobs that \cref{lem:sgd:property} also includes the plain vanilla SGD method (by choosing $\forall \, n \in \N_0 \colon \alpha_n = 0$).

\begin{cproof}{lem:sgd:property}
	Throughout this proof for every $n \in \N_0$ let
		$\Phi_n = (\Phi_n^1 , \ldots, \Phi_n^\fd ) \colon (\R^\fd ) ^{ n + 1 } \to \R^\fd$
	satisfy for all
	$G = ((G_{i,j} )_{j \in \num{\fd} })_{i \in \cu{0, 1, \ldots , n } } \in (\R^\fd ) ^{ n + 1 }$
	that
	\begin{equation}
		\label{lem:sgd:def:phi}
		\Phi_n ( G ) = \gamma_n \ssum_{k=0}^{n} \br*{ (1 - \alpha_k ) \rbr*{ \prod_{l=k+1}^{n} \alpha_l } G_{ k } }  .
	\end{equation}
 \Nobs that \cref{lem:sgd:theta:recursion}, \cref{lem:sgd:def:phi}, and induction
 show for all $n \in \N_0$ that 
 $	\Theta_{n+1} = \Theta_n - \Phi_n ( \fG_0 ( \Theta_0 ) , \allowbreak \fG_1 ( \Theta_1 ) , \ldots, \fG_n ( \Theta_n ) )$.
 Furthermore, \cref{lem:sgd:def:phi} assures for all $n \in \N_0$,
 $G = ((G_{i,j} )_{j \in \num{\fd} })_{i \in \cu{0, 1, \ldots , n } } \in (\R^\fd ) ^{ n + 1 }$,
 $j \in \num{\fd}$ with
 $\sum_{i=0}^n \abs{G_{i,j} } = 0$ that $\Phi_n^j ( G ) = 0$.
\end{cproof}

Next, we show that the Adam optimizer also satisfies the assumptions of \cref{prop:sgd:shallow:nonc}.

\begin{lemma}[Adam optimizer] 
	\label{lem:adam:property}
	Let $\fd \in \N$,
	$(\gamma_n)_{n \in \N _0 } \subseteq [0 , \infty )$,
	$(\alpha_n)_{n \in \N_0 } \subseteq [0 , 1 ]$,
	$(\beta_n)_{n \in \N_0 } \subseteq [0 , 1 ]$,
	$\varepsilon \in (0 , \infty ) $ satisfy $\max \cu{ \alpha_1, \beta_1 } < 1$,
	let $(\Omega , \cF , \P )$ be a probability space,
	for every $n \in \N_0$ let $\fG_n = (\fG_n^1 , \ldots, \fG_n^{ \fd } ) \colon \R^\fd \times \Omega \to \R^\fd$ be measurable,
	and let $\Theta = ( \Theta^1 , \ldots, \Theta^\fd ) \colon \N_0 \times \Omega \to \R^\fd$,
	$\bfm = (\bfm^1 , \ldots , \bfm ^{ \fd } )  \colon \N_0 \times \Omega \to \R^\fd$,
	$\bbM = (\bbM^1 , \ldots , \bbM^\fd )  \colon \N_0 \times \Omega \to \R^\fd$
	satisfy for all
	$n \in \N_0$,
	$i \in \num{\fd}$ that
	$\bfm_0 = \bbM_0 = 0$,
	\begin{equation}
		\label{lem:adam:eq_recursion}
		\begin{split} 
		\bfm_{n+1} & = \alpha_n \bfm_{n} + (1 - \alpha_n ) \fG_{n} ( \Theta_{n} ) , \\
	\bbM^i_{n+1} & = \beta_n \bbM^i_{n} + (1 - \beta_n ) \abs{ \fG^i_{n} ( \Theta_{n} ) } ^2, \\
	\andq 
	\Theta_{n+1}^i & = \Theta_{n} ^i - \gamma_n \rbr*{ \varepsilon + \br*{ \tfrac{ \bbM_{n+1}^i }{1 - \prod_{l=0}^{n} \beta_l } } ^{ \nicefrac{1}{2}} } ^{ - 1 }
	\frac{\bfm_{n+1} ^i }{1 - \prod_{l=0}^{n} \alpha_l }.
	\end{split}
\end{equation}
Then there exist functions $\Phi_n = (\Phi_n^1 , \ldots , \Phi_n^\fd ) \colon (\R^\fd ) ^{ n + 1 } \to \R^\fd$, $n \in \N_0$,
 which satisfy for all $n \in \N_0$ that
 \begin{equation}
 	\Theta_{n+1} = \Theta_n - \Phi_n ( \fG_0 ( \Theta_0 ) , \fG_1 ( \Theta_1 ) , \ldots, \fG_n ( \Theta_n ) )
 \end{equation}
and
which satisfy for all $n \in \N_0$,
$G = ((G_{i,j} )_{j \in \num{\fd} })_{i \in \cu{0, 1, \ldots , n } } \in (\R^\fd ) ^{ n + 1 }$,
$j \in \num{\fd}$ with
$\sum_{i=0}^n \abs{G_{i,j} } = 0$ that $\Phi_n^j ( G ) = 0$.
\end{lemma}

\Nobs that the sequence $(\Theta_n)$ in \cref{lem:adam:property} describes the Adam optimizer  (cf.~Kingma \& Ba~\cite{KingmaBa2015} and Ruder~\cite{Ruder2017overview}) with learning rates $(\gamma_n)$.
 Typical default values for the parameters used in practice are $\forall \, n \in \N_0 \colon \alpha_n = 0.9$, $\beta_n = 0.999$, $\gamma_n = 10^{-3}$, and $\varepsilon = 10^{-8}$.

\begin{cproof}{lem:adam:property}
	Throughout this proof for every $n \in \N_0$ let
	$m_n = ( m_n^1 , \ldots, m_n^\fd ) \colon (\R^\fd ) ^{ n + 1 } \to \R^\fd$,
	$M_n = ( M_n^1 , \ldots, M_n^\fd ) \colon (\R^\fd ) ^{ n + 1 } \to \R^\fd$,
	$\Phi_n = (\Phi_n^1 , \ldots, \Phi_n^\fd ) \colon (\R^\fd ) ^{ n + 1 } \to \R^\fd$
	satisfy for all
	$G = ((G_{i,j} )_{j \in \num{\fd} })_{i \in \cu{0, 1, \ldots , n } } \in (\R^\fd ) ^{ n + 1 }$,
	$j \in \num{\fd} $
	that
	\begin{equation}
		\label{lem:adam:def:m}
		m_n ( G ) =  \ssum_{k=0}^{n} \br*{ (1 - \alpha_k ) \rbr*{ \prod_{l=k+1}^{n} \alpha_l } G_{ k } } ,
	\end{equation}
\begin{equation}
	\label{lem:adam:def:m2}
	M_n^j ( G ) =  \ssum_{k=0}^n \br*{ (1 - \beta_k ) \rbr*{ \prod_{l=k+1}^{n} \beta_l } \abs{ G_{k , j } } ^2 } ,
\end{equation}
and
\begin{equation}
	\label{lem:adam:def:phi}
	\Phi_n^j ( G ) = \gamma_n \rbr*{ \varepsilon + \br*{ \tfrac{ M_n^j ( G ) }{1 - \prod_{l=1}^n \beta_l } } ^{ \nicefrac{1}{2}} } ^{ - 1 }
	\frac{ m_n^j ( G ) }{1 - \prod_{l=1}^n \alpha_l }
\end{equation}
\Nobs that \cref{lem:adam:def:m}, \cref{lem:adam:eq_recursion}, and induction show for all
$n \in \N_0$
that
$\bfm_{n+1} = m_n ( \fG_0 ( \Theta_0 ) , \ldots, \fG_n ( \Theta_n ) )$.
In the same way, \cref{lem:adam:def:m2}, \cref{lem:adam:eq_recursion}, and induction 
prove that for all $n \in \N_0$ 
it holds that
$\bbM_{n+1} = M_n ( \fG_0 ( \Theta_0 ) , \ldots, \fG_n ( \Theta_n ) )$.
Combining this with \cref{lem:adam:eq_recursion} and \cref{lem:adam:def:phi}
establishes for all
$n \in \N_0$ that 
\begin{equation}
	\Theta_{n+1} = \Theta_n - \Phi_n ( \fG_0 ( \Theta_0 ) , \fG_1 ( \Theta_1 ) , \ldots, \fG_n ( \Theta_n ) ).
\end{equation}
Furthermore, \cref{lem:adam:def:m} ensures for all 
$n \in \N_0$,
$G = ((G_{i,j} )_{j \in \num{\fd} })_{i \in \cu{0, 1, \ldots , n } } \in (\R^\fd ) ^{ n + 1 }$,
$j \in \num{\fd}$ with
$\sum_{i=0}^n \abs{G_{i,j} } = 0$ that
$m_n^j  ( G ) = 0$.
Combining this with
\cref{lem:adam:def:phi}
implies for all 
$n \in \N_0$,
$G = ((G_{i,j} )_{j \in \num{\fd} })_{i \in \cu{0, 1, \ldots , n } } \in (\R^\fd ) ^{ n + 1 }$,
$j \in \num{\fd}$ with
$\sum_{i=0}^n \abs{G_{i,j} } = 0$ that
 $\Phi_n^j ( G ) = 0$.
\end{cproof}

It should be pointed out that \cref{lem:adam:property} includes 
the RMSProp algorithm as a special case (namely $\forall \, n \in \N_0 \colon \alpha_n = 0)$, and hence the assumptions of \cref{prop:sgd:shallow:nonc} also hold for this method. 
For other popular gradient methods such as SGD with Nesterov momentum, AdaDelta, 
and AdaGrad they can be shown to hold in a similar way 
(cf., e.g., Becker et al.~\cite[Section 6]{Becker2023_LRV}).

\subsection{Left-approximation gradients as generalized gradients}

The assumptions on the generalized gradient functions $\fG_n^\width$ considered in \cref{setting:snn:sgd} are fairly general. 
In this part we show in \cref{lem:left:approx:partial:deriv} in the case of the ReLU activation that they are in particular satisfied by a standard choice of generalized gradients that has been used, among others, in our previous articles \cite{JentzenRiekert2022_JML,JentzenRiekert2023_Flow}.
For this we require the auxiliary result in \cref{lem:derivative:help}.

\begin{lemma} \label{lem:derivative:help}
	Let $d \in \N$, $x \in \R^{d + 1 }$, $\fC \in (0 , \infty )$,
	let $g \colon \R^{d + 1 } \to \R$ be a function,
	assume that $g$ is differentiable at $x$,
	let $U = \cu{y \in \R^{d+1} \colon \max _{i \in \num{d} } \abs{y_i - x_i } < \fC ( x_{d+1} - y_{d+1} ) }$,
	and assume for all $y \in U$ that $g(x) = g(y)$.
	Then $(\nabla g ) ( x ) = 0$.
\end{lemma}
\begin{cproof}{lem:derivative:help}
	Let $ e_1, e_2, \dots, e_d \in \R^d$ be the standard basis of $\R^d$ and let $\bfe_i \in \R^{d+1}$, $i \in \num{d+1}$, satisfy $\forall \, j \in \num{d} \colon \bfe_i = (e_i , 1 + \fC^{-1} )$ and $\bfe_{d+1} = ( 0, \ldots, 0 , 1 )$.
	\Nobs that $\cu{\bfe_i} _{i \in \num{d + 1 } }$ is a basis of $\R^{d+1}$. Furthermore, \nobs that for all $i \in \num{d+1}$, $\lambda \in (0 , \infty )$ it holds that $x - \lambda  \bfe_i \in U$.
	This, the fact that $g$ is differentiable at $x$,
	and the fact that $\forall \, y \in U \colon g ( y ) = g ( x ) $
	assure for all $i \in \num{d+1}$
	that
	\begin{equation}
	\begin{split}
	0 & = \lim\nolimits_{\lambda \downarrow 0 } \br*{ \lambda^{-1}  \abs{ g ( x - \lambda \bfe_i ) - g ( x ) + \spro{ ( \nabla g ) ( x ) ,  \lambda \bfe_i} } } \\
	&=  \lim\nolimits_{\lambda \downarrow 0 } \br*{ \lambda^{-1}  \abs{ \spro{ ( \nabla g ) ( x ) ,  \lambda \bfe_i} } } = \abs{ \spro{ ( \nabla g ) ( x ) ,  \bfe_i } }.
	\end{split}
	\end{equation}
	Hence, we obtain that $( \nabla g ) ( x ) = 0$. 
\end{cproof}

We now establish the announced property of the generalized gradients defined in \cref{lem:left:approx:partial:deriv} as limits of the gradients of appropriate continuously differentiable approximations of the risk function.
The definition of the generalized gradients $\cG^\width$ in \cref{lem:left:approx:partial:deriv} is natural since it agrees with the standard implementation of the gradient in the widely used machine learning libraries \textsc{TensorFlow} and \textsc{PyTorch}
in the case of ANNs with the ReLU activation.

\begin{lemma}
	\label{lem:left:approx:partial:deriv}
	Assume \cref{setting:snn},
	assume $\forall \, x \in \R \colon \sigma ( x ) = \max \cu{ x , 0 }$,
	let $\Rect _r \colon \R \to \R$, $r \in \N $, satisfy for all  $x \in \R$ that $ \rbr{ \bigcup_{r \in \N } \{ \Rect _r \}  } \subseteq C^1 ( \R , \R)$,
	 $\sup_{r \in \N} \sup_{y \in [-\abs{x}  , \abs{ x } ]}  \abs{ (\Rect _r)'(y) }  < \infty$, and 
	\begin{equation} \label{setting:assumption:rect}
	\limsup\nolimits_{r \to \infty} \rbr*{ \abs { \Rect _r ( x ) - \max \cu{ x , 0 } } + \abs { (\Rect _r)' ( x ) - \indicator{(0, \infty)} ( x ) } } = 0,
	\end{equation}
	for every $\width , r \in \N$ let $\fL^\width_r \colon \R^{ \fd_\width} \to\R$
	satisfy for all
	$\theta = (\theta_1, \ldots, \theta_{\fd_\width  } ) \in \R^{ \fd_\width } $ that
	\begin{equation} \label{setting:snn:eq:realization}
	\fL_r ^\width ( \theta ) =
	 \int_{[a,b]^d } \rbr[\big]{  \theta_{ ( d + 2 ) \width + 1 } + \ssum_{i=1}^\width \theta_{ ( d + 1 ) \width + i } \Rect_r \rbr[\big]{ \theta_{d \width + i } + \ssum_{j=1}^d \theta_{(i - 1 ) d + j } x_j } - f ( x ) }^2 \, \mu ( \d x ),
	\end{equation}
	for every $\width \in \N$ let $\cG^\width  = (\cG^\width _1, \ldots, \cG^\width _{\fd_\width } ) \colon \R^{\fd_\width } \to \R^{\fd_\width }$ satisfy for all
	$\theta \in  \{ \vartheta \in \R^{\fd_\width } \colon   ((\nabla \cL_r ) ( \vartheta ) )_{r \in \N } \text{ is }\allowbreak\text{convergent} \}$ that $\cG^\width  ( \theta ) = \lim_{r \to \infty} (\nabla \fL_r^\width  ) ( \theta )$.
Then it holds for all
$ \width \in \N $,
$ i \in \num{\width} $,
$ j \in \cu{(i - 1 ) d + k \colon k \in \num{d} } \cup \cu{\width d + i } $,
$ 
    \theta \in \{ \vartheta \in \R^{\fd_{ \width } } \colon 
        \R^{ d + 1 } \ni ( \psi_1, \ldots, \psi_{d+1} ) \mapsto \cL_\width (\vartheta_1 , \ldots, \vartheta_{( i - 1 ) d}, 
        \allowbreak 
        \psi_1, \ldots, \psi_d,
        \vartheta_{i d + 1}, \ldots, 
        \vartheta_{d \width + i - 1 }, 
        \allowbreak
        \psi_{d+1}, \allowbreak
        \vartheta_{d \width + i + 1 } ,
        \ldots, \vartheta_{\fd_{ \width } } 
        ) \in \R 
        \text{ is differentiable at } 
        ( \vartheta_{ ( i - 1 ) d + 1 } , \ldots, \allowbreak \vartheta_{ id } , \allowbreak \vartheta_{\width d + i } ) 
    \} 
$
that
\begin{equation}
  \cG^{\width}_j( \theta ) 
  = 
  \bigl(
    \tfrac{ \partial }{ \partial \theta_j } \cL_{\width} 
  \bigr)( \theta ) 
  ,
\end{equation}
\end{lemma}

\begin{cproof}{lem:left:approx:partial:deriv}
	Throughout this proof let $\width \in \N$,
	$i \in \num{\width}$,
	$\theta \in \R^{ \fd_\width }$,
	that $	\R^{ d + 1 } \ni ( \psi_1, \ldots, \psi_{d+1} ) \mapsto \cL_\width (\theta_1 , \ldots, \theta_{( i - 1 ) d}, \psi_1, \ldots, \psi_d,
	\theta_{i d + 1}, \ldots, 
	\theta_{d \width + i - 1 }, \allowbreak
	\psi_{d+1}, \allowbreak
	\theta_{d \width + i + 1 },
	\ldots, \theta_{\fd_{ \width } } ) \in \R$ is differentiable at
	$( \theta_{ ( i - 1 ) d + 1 } , \ldots, \theta_{ id } , \theta_{\width d + i } )$
	and let $\nsign_s \subseteq [a,b]^d $,
	$s \in \cu{-1, 0, 1 }$,
	satisfy for all $s \in \cu{-1 , 1 }$ that
$
  \nsign_s = \cu{x \in [a,b]^d \colon s \rbr{\sum_{j=1}^d \theta_{(i - 1 ) d + j } x_j + \theta_{\width d + i } } > 0 } 
$ 
and 
\begin{equation}
\textstyle 
  \nsign_0 = 
  \bigl\{ 
    x \in [a,b]^d \colon 
    \sum_{j=1}^d \theta_{(i - 1 ) d + j } x_j + \theta_{\width d + i } = 0 
  \bigr\}
  .
\end{equation}
\Nobs that for all $x \in \nsign_{-1} \cup \nsign_1$ it holds that 
\begin{equation}
  \R^{ d + 1 } \ni 
  ( \psi_1, \ldots, \psi_{d+1} ) \mapsto 
  \realization{(\theta_1 , \ldots, \theta_{( i - 1 ) d}, \psi_1, \ldots, \psi_d,
  \theta_{i d + 1}, \ldots, 
  \theta_{d \width + i - 1 }, \allowbreak
  \psi_{d+1}, \allowbreak
  \theta_{d \width + i + 1 },
  \ldots, \theta_{\fd_{ \width } } } ( x )  \in \R
\end{equation}
is differentiable at
$( \theta_{ ( i - 1 ) d + 1 } , \ldots, \theta_{ id } , \theta_{\width d + i } )$ 
 and satisfies
 \begin{equation} 	\label{lem:left:approx:partial:deriv:eq1}
	 \tfrac{\partial}{\partial \theta_{\width d + i } } \realization{\theta} ( x ) = \theta_{ ( d + 1 ) \width + i } \indicator{\nsign_1 } ( x ) \qandq \forall \, j \in \num{d} \colon  \tfrac{\partial}{\partial \theta_{ ( i - 1 ) d + j } } \realization{\theta} ( x ) = \theta_{ ( d + 1 ) \width + i } x_j \indicator{\nsign_1 } ( x ).
 \end{equation}
 This and the same argument as in \cite[Corollary 2.3]{IbragimovJentzenKroger2022LocalMinima}
 assure that $	\R^{ d + 1 } \ni ( \psi_1, \ldots, \psi_{d+1} ) \mapsto \int_{ \nsign_{-1} \cup \nsign_1 } (\realization{(\theta_1 , \ldots, \theta_{( i - 1 ) d}, \psi_1, \ldots, \psi_d,
 	\theta_{i d + 1}, \ldots, 
 	\theta_{d \width + i - 1 }, \allowbreak
 	\psi_{d+1}, \allowbreak
 	\theta_{d \width + i + 1 },
 	\ldots, \theta_{\fd_{ \width } } } ( x ) - f ( x ) ) ^2 \, \mu ( \d x ) \in \R$
 is differentiable at
 $( \theta_{ ( i - 1 ) d + 1 } , \ldots, \theta_{ id } , \theta_{\width d + i } )$
 and satisfies
 \begin{equation} 
 	\label{lem:left:approx:partial:deriv:eq2}
 \begin{split}
  &\tfrac{\partial}{\partial \theta_{\width d + i } } \rbr*{\int_{ \nsign_{-1} \cup \nsign_1 } (\realization{ \theta  } ( x ) - f ( x ) ) ^2 \, \mu ( \d x ) }
 =  \theta_{ ( d + 1 ) \width + i }  \tint_{\nsign_1} ( \realization{\theta} ( x ) - f ( x ) ) 
 \, \mu( \d x ) \qandq \\
  & \forall \, j \in \num{d} \colon 
 \tfrac{\partial}{\partial \theta_{ ( i - 1 ) d + j } } \rbr*{\tint_{ \nsign_{-1} \cup \nsign_1 } (\realization{ \theta  } ( x ) - f ( x ) ) ^2 \, \mu ( \d x ) }
 =  \theta_{ ( d + 1 ) \width + i }  \tint_{\nsign_1} x_j ( \realization{\theta} ( x ) - f ( x ) ) 
 \, \mu( \d x ) .
  \end{split}
 \end{equation}
 Combining this with the assumption that $	\R^{ d + 1 } \ni ( \psi_1, \ldots, \psi_{d+1} ) \mapsto \cL_\width (\theta_1 , \ldots, \theta_{( i - 1 ) d}, \psi_1, \ldots, \psi_d, \allowbreak
 \theta_{i d + 1}, \ldots, 
 \theta_{d \width + i - 1 }, \allowbreak
 \psi_{d+1}, \allowbreak
 \theta_{d \width + i + 1 },
 \ldots, \theta_{\fd_{ \width } } ) \in \R$ is differentiable at
 $( \theta_{ ( i - 1 ) d + 1 } , \ldots, \theta_{ id } , \theta_{\width d + i } )$
 demonstrates that $\R^{ d + 1 } \ni ( \psi_1, \ldots, \psi_{d+1} ) \mapsto \int_{ \nsign_0 } (\realization{(\theta_1 , \ldots, \theta_{( i - 1 ) d}, \psi_1, \ldots, \psi_d,
 	\theta_{i d + 1}, \ldots, 
 	\theta_{d \width + i - 1 }, \allowbreak
 	\psi_{d+1}, \allowbreak
 	\theta_{d \width + i + 1 },
 	\ldots, \theta_{\fd_{ \width } } } ( x ) - f ( x ) ) ^2 \, \mu ( \d x ) \in \R$
 is differentiable at
 $( \theta_{ ( i - 1 ) d + 1 } , \ldots, \theta_{ id } , \theta_{\width d + i } ) 
 $.
 In the next step let 
 $ \fC \in \R $
 satisfy 
 $ \fC = d \max \cu{\abs{a} , \abs{b} } $ 
 and let 
 $
   U \subseteq \R^{ d + 1 }
 $
 satisfy
 \begin{multline}
   U = 
   \bigl\{  
     \psi = 
     ( \psi_1, \ldots, \psi_{d + 1 } ) \in \R^{ d + 1 } 
     \colon 
\\
\textstyle 
     \psi_{d+1} < \theta_{ \width d + i }, \, \max_{j \in \num{d} } \abs{ \psi_j - \theta_{ ( i - 1 ) d + j } } 
     < \fC ^{-1} ( \theta_{\width d + i } - \psi_{d + 1 } )  
  \bigr\} 
  .
\end{multline}
 \Nobs that for all $x \in \nsign_0$, $\psi \in U$ we have that
 \begin{equation}
 \begin{split}
 	&\ssum_{j=1}^d \psi_j x_j + \psi_{d+1} \\
 	& = \br*{ \ssum_{j=1}^d \theta_{ ( i - 1 ) d + j } x_j + \theta_{\width d + i } }
 	+ \br*{ \ssum_{j=1}^d (\psi_j - \theta_{ ( i - 1 ) d + j } ) x_j + ( \psi_{d+1} - \theta_{\width d + i } ) } \\
 	&= \ssum_{j=1}^d (\psi_j - \theta_{ ( i - 1 ) d + j } ) x_j + ( \psi_{d+1} - \theta_{\width d + i } ) \\
 	& \le \ssum_{j=1}^d \abs{\psi_j - \theta_{ ( i - 1 ) d + j } } \abs{ x_j } + ( \psi_{d+1} - \theta_{\width d + i } ) \\
 	& \le d \fC^{-1} \max \cu{\abs{a} , \abs{b}}  ( \theta_{\width d + i } - \psi_{d + 1 } )
 	+ ( \psi_{d+1} - \theta_{\width d + i } ) = 0.
 \end{split}
 \end{equation}
This shows for all $x \in \nsign_0$, $\psi \in U$ that 
\begin{equation}
   \sigma ( \ssum_{j=1}^d \psi_j x_j + \psi_{d+1} ) = 0 = \sigma ( \ssum_{j=1}^d \theta_{ ( i - 1 ) d + j } x_j + \theta_{\width d + i } )
   .
\end{equation}
Hence, we obtain for all $x \in \nsign_0$, $\psi \in U$ that 
\begin{equation}
  \realization{(\theta_1 , \ldots, \theta_{( i - 1 ) d}, \psi_1, \ldots, \psi_d,
 	\theta_{i d + 1}, \ldots, 
 	\theta_{d \width + i - 1 }, \allowbreak
 	\psi_{d+1}, \allowbreak
 	\theta_{d \width + i + 1 },
 	\ldots, \theta_{\fd_{ \width } } } ( x ) = \realization{\theta} ( x )
  .
\end{equation}
Combining this and \cref{lem:derivative:help}
with the fact that 
\begin{multline}
\textstyle 
  \R^{ d + 1 } 
  \ni ( \psi_1, \ldots, \psi_{d+1} ) \mapsto 
\\
\textstyle 
  \int_{ \nsign_0 } (\realization{(\theta_1 , \ldots, \theta_{( i - 1 ) d}, \psi_1, \ldots, \psi_d,
  \theta_{i d + 1}, \ldots, 
  \theta_{d \width + i - 1 }, \allowbreak
  \psi_{d+1}, \allowbreak
  \theta_{d \width + i + 1 },
  \ldots, \theta_{\fd_{ \width } } } ( x ) - f ( x ) ) ^2 \, \mu ( \d x ) \in \R
\end{multline}
 is differentiable at
 $( \theta_{ ( i - 1 ) d + 1 } , \ldots, \theta_{ id } , \theta_{\width d + i } )$
 shows that for all $j \in \cu{ ( i - 1 ) d + k \colon k \in \num{d} } \cup \cu{ \width d + i }$
it holds that
 \begin{equation}
 	\tfrac{\partial}{\partial \theta_{ j } } \rbr*{\tint_{ \nsign_{0} \cup \nsign_1 } (\realization{ \theta  } ( x ) - f ( x ) ) ^2 \, \mu ( \d x ) } = 0.
 \end{equation}
 This, \cref{lem:left:approx:partial:deriv:eq2},
 and
 \cite[Proposition 2.2]{JentzenRiekert2023_Flow}
 ensure that
 \begin{equation}
 \textstyle 
  \tfrac{\partial}{\partial \theta_{\width d + i } } \cL_\width ( \theta ) = \theta_{ ( d + 1 ) \width + i }  \int_{\nsign_1} ( \realization{\theta} ( x ) - f ( x ) ) \, \mu( \d x ) = \cG_{\width d + i }^\width ( \theta )
\end{equation}
\begin{equation}
\textstyle 
\text{and}\qquad
  \forall \, j \in \num{d} \colon 
 \tfrac{\partial}{\partial \theta_{ ( i - 1 ) d + j } } \cL ( \theta )
 =  \theta_{ ( d + 1 ) \width + i }  \int_{\nsign_1} x_j ( \realization{\theta} ( x ) - f ( x ) ) \, \mu( \d x )
 = \cG^\width _ { ( i - 1 ) d + j } ( \theta ) .
 \end{equation}
\end{cproof}

\section{A priori bounds for GF and GD in the training of deep ANNs}
\label{sec:a_priori_bounds}

In \cref{sec:locmin,sec:nonc} we examined the risk levels of local minima for shallow ANNs and, in particular, established the existence of multiple risk levels that are bounded above by the best approximation error of an ANN with constant realization function.
We also investigated the behavior of gradient processes in this regime.
To complement these findings, in this final section we will instead consider the training process for the times where the risk value is larger than the error of the best constant approximation.
The results in this section again hold for deep ANNs with an arbitrary number of hidden layers and the ReLU activation function.

\subsection{Description of deep ANNs with ReLU activation}

	We first introduce our notation for deep ReLU ANN, the corresponding risk function and its generalized gradient, which is inspired by Hutzenthaler et al.~\cite{HutzenthalerJentzenPohlRiekertScarpa2021}.

\begin{setting}
	\label{setting:dnn}
	Let 
	$L, \fd\in \N $,
	$ ( \ell_k)_{k\in \N_0 } \subseteq \N $,
	$ a, \mathbf{a} \in \R $, $ b \in (a, \infty) $, 
	$ \scrA \in (0,\infty) $, 
	$ \scrB \in ( \scrA, \infty) $ 
	satisfy $ \fd = \sum_{ k = 1 }^L \ell_k( \ell_{k-1} + 1) $
	and $ \mathbf{a} =\max\{\abs{a}, \abs{b},1\} $,
	for every $ \theta = ( \theta_1, \ldots, \theta_{ \fd }) \in \R^{ \fd } $, $k \in \N$
	let $ \fw^{ k, \theta } = ( \fw^{k, \theta }_{ i,j} )_{ (i,j) \in \num{ \ell_k } \times \num{ \ell_{k-1} } }
	\in \R^{ \ell_k \times \ell_{k-1} } $
	and $ \fb^{k, \theta } = ( \fb^{  k, \theta }_i )_{i \in \num{\ell_k } }
	\in \R^{ \ell_k} $
	satisfy for all
	$ i \in \num{ \ell_k } $,
	$ j \in \num{ \ell_{k-1} } $ that
	\begin{equation}
		\fw^{ k, \theta }_{ i, j } 
		=
		\theta_{ (i-1)\ell_{k-1} + j+\sum_{h= 1 }^{ k-1} \ell_h( \ell_{h-1} + 1)}
		\qqandqq
		\fb^{ k, \theta }_i =
		\theta_{\ell_k\ell_{k-1} + i+\sum_{h= 1 }^{ k-1} \ell_h( \ell_{h-1} + 1)} 
		\, ,
	\end{equation}
	for every $ k \in \N $, $ \theta \in \R^{ \fd } $
	let $ \cA_k^\theta = ( \cA_{k,1}^\theta, \ldots, \cA_{k, \ell_k}^\theta)
	\colon \R^{ \ell_{k-1} } \to \R^{ \ell_k} $
	satisfy 
	for all $ x \in \R^{ \ell_{ k - 1 } } $ that 
	$ 
	\cA_k^{ \theta }( x ) 
	=
	\fb^{ k, \theta } + \fw^{ k, \theta } x 
	$,
	let 
	$ \Rect_r \colon \R \to \R $, 
	$ r \in [1, \infty] $,
	satisfy for all 
	$ r \in [1, \infty) $, 
	$ x \in ( - \infty, \scrA r^{ - 1 } ] $, 
	$ 
	y \in \R
	$, 
	$
	z \in [ \scrB r^{ - 1 }, \infty ) 
	$
	that 
	\begin{equation}
	\Rect_r \in C^1( \R, \R ) ,
	\quad
	\Rect_r ( x ) = 0,
	\quad 
	0 \leq \Rect_r ( y ) \leq \Rect_{ \infty }( y ) 
	= 
	\max\{ y, 0 \}
	,
	\qandq
	\Rect_r(z) = z
	,
	\end{equation}
	assume 
	$
	\sup_{ r \in [1, \infty) }
	\sup_{ x \in \R } 
	| ( \Rect_r )'( x ) | < \infty 
	$,
	let 
	$
	\mathfrak{M}_r 
	\colon( \cup_{ n \in \N } \R^n ) \to ( \cup_{ n \in \N } \R^n )
	$, 
	$ r \in [1, \infty]
	$, 
	satisfy for all $ r \in [1, \infty] $, $ n \in \N $,
	$ x = (x_1, \ldots, x_n ) \in \R^n $
	that 
	\begin{equation}
	\mathfrak{M}_r(x)= ( \Rect_r(x_1), \ldots, \Rect_r(x_n ) ),
	\end{equation}
	for every $k \in \N$,
	$i \in \num{\ell_k}$,
	 $ \theta \in \R^{ \fd } $,
	$r \in [1 , \infty ]$
	let $ \realization{ k, \theta }_r = ( \realization{ k, \theta }_{r,1}, \ldots, \realization{ k, \theta }_{r, \ell_k} )
	\colon \R^{ \ell_0 } \to \R^{ \ell_k} $
	satisfy 
	for all $ x \in \R^{ \ell_0 } $
	that
	\begin{equation}
		\realization{ 1, \theta }_r(x) = \cA^\theta_1(x) 
		\qqandqq
		\realization{ k + 1, \theta }_r( x ) 
		=
		\cA_{ k + 1}^{ \theta }(
		\mathfrak{M}_{ r^{ 1 / k } }(
		\realization{ k, \theta }_r( x )
		)
		),
	\end{equation}
	let $ \mu \colon \mathcal{B} ( [a,b]^{ \ell_0 } ) \to [0, \infty] $ be a finite measure,
	let $ f = (f_1, \ldots, f_{\ell_L}) \colon [a,b]^{ \ell_0 } \to \R^{ \ell_L} $ be measurable,
	for every $ r \in [1, \infty] $ 
	let $ \cL_r\colon \R^{ \fd }\to \R $ satisfy 
	for all $ \theta \in \R^{ \fd } $
	that 
	\begin{equation}
		\cL_r( \theta)=\int_{[a,b]^{ \ell_0 } }
		\norm{ \realization{L, \theta }_r ( x )- f ( x ) }^2\, \mu( \d x),
	\end{equation}
	let $ \cG= ( \cG_1, \ldots, \cG_{ \fd })\colon \R^{ \fd }\to\R^{ \fd } $ satisfy
	for all $ \theta \in \{\vartheta \in \R^{ \fd }\colon(( \nabla \cL_r)( \vartheta))_{ r \in [1,\infty) }
	\text{ is convergent} \} $ that
	$ \cG ( \theta)=\lim_{r\to\infty}( \nabla \cL_r)( \theta) $, and
	for every $\xi \in \R^{\ell_L}$ 
	let $V_\xi \colon \R^{ \fd }\to\R $ satisfy for all $ \theta \in \R^{ \fd } $  that
	\begin{equation}
		V_\xi ( \theta)=\br*{ \ssum_{k= 1 }^L \rbr[\big]{ 
		k \norm{ \fb^{ k, \theta } } ^2 +
		\ssum_{ i= 1 }^{ \ell_k} \ssum_{ j = 1 }^{ \ell_{k-1} }
		|\fw^{ k, \theta }_{ i, j}|^2 } }
		- 2L\langle \xi, \fb^{L, \theta } \rangle .
	\end{equation}
\end{setting}

\subsection{Lyapunov functions in the training of ANNs}

We next review two properties of the functions $V_\xi$ introduced in \cref{setting:dnn}, which have been established in 
Hutzenthaler et al.~\cite{HutzenthalerJentzenPohlRiekertScarpa2021}.

\begin{lemma} \label{lem:lyap:est}
	Assume \cref{setting:dnn}
	and let
	$\theta \in \R^\fd$,
	$\xi \in \R^{\ell_L}$.
	Then
	\begin{equation}
	\label{eq:lyap:est}
	\tfrac{1}{2} \norm{\theta} ^2 - 2 L^2 \norm{\xi}^2 \le V_\xi ( \theta ) \le 2 L \norm{\theta} ^2 + L \norm{\xi} ^2 .
	\end{equation}
\end{lemma}

\begin{cproof}{lem:lyap:est}
\Nobs that 
Hutzenthaler et al.~\cite[Proposition 3.1]{HutzenthalerJentzenPohlRiekertScarpa2021} 
establishes \cref{eq:lyap:est}.
\end{cproof}

\begin{lemma} \label{lem:lyapunov}
	 Assume \cref{setting:dnn} and let $ \theta \in \R^{ \fd } $, $\xi \in \R^{ \ell_L}$.
	Then
	\begin{equation}
	\label{eq_V}
	\spro{ ( \nabla V_\xi )( \theta) , \cG( \theta)} =
	4L \int_{[a,b]^{ \ell_0 } } \spro{\realization{L, \theta }_\infty(x)-f(x ) ,
		\realization{L, \theta }_\infty(x)- \xi } \, \mu( \d x) .
	\end{equation}
\end{lemma}

\begin{cproof}{lem:lyapunov}
\Nobs that Hutzenthaler et al.~\cite[Proposition 3.2]{HutzenthalerJentzenPohlRiekertScarpa2021} 
establishes \cref{eq_V}.
\end{cproof}

\subsection{A priori bounds for gradient flow processes}

The next result generalizes the findings in \cite{JentzenRiekert2023_Flow} to the case of deep neural networks.
It can also be viewed as a generalization of \cite[Theorem 3.8]{HutzenthalerJentzenPohlRiekertScarpa2021}, which deals with the case that the considered target function $f$ is constant.
In this case, the optimal error of a constant approximation is $\nu = 0$, whence it follows that $\limsup_{t \to \infty} \cL _\infty ( \Theta_t ) = 0$.

\begin{prop} \label{prop:gf:lyap}
	Assume \cref{setting:dnn},
	let $\Theta \in C ( [ 0, \infty ) , \R^{\fd } )$ satisfy for all $t \in [0, \infty ) $ that
	$\Theta_t = \Theta_0 - \int_0 \cG ( \Theta_s ) \, \d s$,
	and let $\xi \in \R^{\ell_L}$, $\nu \in \R$
	satisfy $\nu = \int_{[a,b]^{\ell_0 } } \norm{ \realization{L, \theta } _\infty ( x ) - \xi } ^2 \, \mu ( \d x )$.
	Then
	\begin{enumerate}[label = (\roman*)]
		\item \label{prop:gf:lyap:item2} it holds that
		\begin{equation}
			\sup\nolimits _{ t \in [0, \infty), \, \cL _\infty ( \Theta_t ) \geq \nu \indicator{(0, \infty ) } ( t ) } \norm{\Theta_t } 
			\le 2 V_\xi ( \Theta_0 ) + 4 L ^2 \norm{\xi } ^2 < \infty
		\end{equation}
		and
		\item \label{prop:gf:lyap:item1} it holds that $\limsup_{t \to \infty} \cL _\infty ( \Theta_t ) \le \nu$.
	\end{enumerate}
\end{prop}

\begin{cproof}{prop:gf:lyap}
	\Nobs that \cref{lem:lyapunov},
	the Cauchy-Schwarz inequality,
	and the fact that for all $u,v \in \R$ it holds that $uv \le \frac{u ^2 + v ^2 }{2}$
	ensure that for all $\theta \in \R^\fd$ we have that
	\begin{equation}
	\begin{split}
	&\spro{ ( \nabla V_\xi)( \theta) , \cG( \theta)} \\
	&= 4 L \cL_\infty ( \theta ) + 4 L \int_{[a,b]^{\ell_0 } } \spro{\realization{L, \theta }_\infty(x)-f(x ) ,
	f ( x ) - \xi } \, \mu( \d x) \\
& \ge 4 L \cL_\infty ( \theta ) - 4 L \br*{\int_{[a,b]^{\ell_0 } } \norm{ \realization{L, \theta } _\infty ( x ) - f(x) } ^2 \, \mu ( \d x )}^{1/2} \br*{\int_{[a,b]^{\ell_0 } } \norm{ \realization{L, \theta } _\infty ( x ) - \xi } ^2 \, \mu ( \d x )}^{1/2} \\
& = 4 L \cL_\infty ( \theta ) - 4 L \sqrt{\cL_\infty ( \theta ) } \sqrt {\nu}
\ge 2 L \cL_\infty ( \theta ) - 2 L \nu .
	\end{split}
	\end{equation}
	Combining this with Cheridito et al.~\cite[Lemma 3.1]{CheriditoJentzenRiekert2022} (see also \cite[Proposition 3.4]{HutzenthalerJentzenPohlRiekertScarpa2021})
	demonstrates for all $t \in [0, \infty )$ that
	\begin{equation} \label{prop:gf:lyap:eq:1}
	V_\xi ( \Theta_t ) = V_\xi ( \Theta_0 ) - \int_0^t \spro{ ( \nabla V_\xi)( \Theta_s ) , \cG( \Theta_s ) } \, \d s \le V_\xi ( \Theta_0 ) - 2 L \int_0^t ( \cL_\infty ( \Theta_s ) - \nu ) \, \d s .
	\end{equation}
	In the following let $\bfm \in \R$ satisfy $\bfm = \inf_{t \in [0, \infty ) } \cL_\infty( \Theta_t )$. \Nobs that the fact that $[0, \infty ) \ni t \mapsto \cL_\infty ( \Theta_t ) \in \R$ is non-increasing implies that $\limsup_{t \to \infty} \cL_\infty ( \Theta_t ) = \bfm$. In addition, \nobs that \cref{prop:gf:lyap:eq:1} and \cref{lem:lyap:est} show for all $t \in [0, \infty )$ that
	\begin{equation}
	2 L t ( \bfm - \nu ) \le 2 L \int_0^t ( \cL_\infty ( \Theta_s ) - \nu ) \, \d s
	\le V _\xi ( \Theta_0 ) - V_\xi ( \Theta_t ) \le V_\xi ( \Theta_0 ) + 2 L ^2 \norm{\xi } ^2 .
	\end{equation}
	This proves that $\bfm \le \nu$, which establishes \cref{prop:gf:lyap:item1}.
	Next, \nobs that the fact that $[0, \infty ) \ni t \mapsto \cL_\infty ( \Theta_t ) \in \R$ is non-increasing ensures for all $t \in [0, \infty )$ with $\cL_\infty ( \Theta_t ) \ge \nu$ that 
	\begin{equation}
	  \inf_{s \in [0,t]} \cL_\infty ( \Theta_s ) \ge \nu
	  .
	\end{equation}
    Combining this with \cref{prop:gf:lyap:eq:1} and \cref{lem:lyap:est} demonstrates for all
		 $t \in [0, \infty )$ with $\cL_\infty ( \Theta_t ) \ge \nu$ that
	\begin{equation}
	\begin{split}
	\norm{\Theta_t } ^2 
	&\le 2 V_\xi ( \Theta_t ) + 4 L ^2 \norm{\xi } ^2
	\le 2 V_\xi ( \Theta_0 ) + 4 L ^2 \norm{\xi } ^2 - 4 L \int_0^t ( \cL_\infty ( \Theta_s ) - \nu ) \, \d s \\
	&\le 2 V_\xi ( \Theta_0 ) + 4 L ^2 \norm{\xi } ^2.
	\end{split}
	\end{equation}
	This establishes \cref{prop:gf:lyap:item2}.
\end{cproof}

\subsection{A priori bounds for gradient descent processes}

In \cref{prop:gd:limsup} we show an analogous result to \cref{prop:gf:lyap} for discrete-time gradient descent processes, provided that the learning rates are sufficiently small.
Due to the error that occurs in the discrete numerical approximation of the gradient flow, we can only show that in the limit the risk values are bounded above by $\nu + \varepsilon$, where $\nu$ is the optimal error of a constant approximation of the target function.

\begin{prop} \label{prop:gd:limsup}
	Assume \cref{setting:dnn},
	assume $\mu ( [ a ,b ] ^{ \ell_0 } ) > 0 $,
	let $\theta \in \R^\fd$, $\varepsilon \in (0, \infty )$,
	$\xi \in \R^{\ell_L}$, $\nu \in \R$
	satisfy $\nu = \int_{[a,b]^{\ell_0 } } \norm{ \realization{L, \theta }_\infty ( x ) - \xi } ^2 \, \mu ( \d x )$,
	let $P \colon \R \to \R$
	satisfy for all
	$y \in \R$
	that
	$P(y) = L \bfa ^2 \mu ( [ a,b]^d ) \prod_{p=0}^L ( \ell_p + 1 ) (2 y + 4 L^2 \norm{\xi}^2 + 1 ) ^{ L - 1 }$,
	let $(\gamma_n)_{n \in \N_0} \subseteq [0 , \infty )$ satisfy $\sum_{n=0}^\infty \gamma_n = \infty $ and
	\begin{equation} \label{prop:gd:limsup:eq:gamma}
	\sup\nolimits_{n \in \N_0} \gamma_n < \tfrac{  \varepsilon  }{ 2(\nu + \varepsilon )   P ( V_ \xi ( \theta )) } ,
	\end{equation}
	 let $\Theta \colon \N_0 \to \R^\fd$
	satisfy for all 
	$n \in \N_0$
	that
	$\Theta_0 = \theta$
	and
	$\Theta_{n+1} = \Theta_n - \gamma _n \cG ( \Theta_n )$,
	and let $T \in \N_0 \cup \cu{\infty}$ satisfy $T = \inf( \cu{n \in \N_0 \colon \cL_\infty ( \Theta_n ) \le \nu + \varepsilon } \cup \cu{\infty} )$.
	Then
	\begin{enumerate} [label = (\roman*)]
	    \item \label{prop:gd:limsup:item1} it holds that $\sup_{n \in \N_0 \cap [0, T ] } \norm{\Theta_n } \le 2 V_\xi ( \Theta_0 ) + 4 L ^2 \norm{\xi } ^2 < \infty $,
	    \item \label{prop:gd:limsup:item2} it holds that $\inf_{n \in \N _0} \cL_\infty ( \Theta_n ) \le \nu + \varepsilon$, and
	    \item \label{prop:gd:limsup:item3} it holds that $\limsup_{n \to \infty} \cL_\infty ( \Theta_{\min \cu{n, T } } ) \le \nu + \varepsilon$.
	\end{enumerate}
\end{prop}

\begin{cproof}{prop:gd:limsup}
	Throughout this proof assume without loss of generality that $\nu >0$ (cf.~\cite{HutzenthalerJentzenPohlRiekertScarpa2021}) and $T>0$.
	\Nobs that \cite[Corollary 4.4]{HutzenthalerJentzenPohlRiekertScarpa2021}, the Cauchy-Schwarz inequality,
		and the fact that for all $u,v \in \R$ it holds that $uv \le \frac{u ^2 + v ^2 }{2}$ demonstrate that for all $n \in \N_0$ we have that
	\begin{equation} \label{prop:gd:eq:help1}
	\begin{split}
	& V_\xi ( \Theta_{n+1} ) 
	 - V _ \xi ( \Theta_n ) \\
	 & \le 4 \gamma_n ^2 L P ( V_ \xi ( \Theta_n )) \cL_\infty ( \Theta_n ) - 4 \gamma _n L \int_{[a,b]^{ \ell_0 } } \spro{\realization{L, \theta }_\infty(x)-f(x ) ,
		\realization{L, \theta }_\infty(x)- \xi } \, \mu( \d x) \\
	& =  4 \gamma_n^2 L P ( V_ \xi ( \Theta_n )) \cL_\infty ( \Theta_n ) - 4 \gamma  L \cL_\infty ( \Theta_n ) \\
	& \quad - 4 \gamma_n  L \int_{[a,b]^{ \ell_0 } } \spro{\realization{L, \theta }_\infty(x)-f(x ) ,
		f ( x ) - \xi } \, \mu( \d x) \\
	& \le  4 \gamma_n^2 L P ( V_ \xi ( \Theta_n )) \cL_\infty ( \Theta_n ) - 4 \gamma_n  L \cL_\infty ( \Theta_n ) + 4 \gamma_n L \sqrt{ \cL_\infty ( \Theta_n ) } \sqrt{ \nu } \\
	& \le 4 \gamma_n^2 L P ( V_\xi ( \Theta_n ) ) \cL_\infty ( \Theta_n ) - 4 \gamma_n L \cL_\infty ( \Theta_n ) + 2 \gamma _n L ( \cL_\infty ( \Theta_n ) + \nu ) \\
	&=  2 \gamma _n L \cL _\infty ( \Theta_n ) 
	\br*{2 \gamma _n P ( V_ \xi ( \Theta_n ) ) - 1 } + 2 \gamma_n L \nu .
	\end{split}
	\end{equation}
	Next, we prove by induction on $n $ that for all $n \in \N_0 \cap [0, T )$ we have that 
	\begin{equation} \label{prop:gd:eq:indclaim}
	V_\xi ( \Theta_{n+1} ) \le V_\xi ( \Theta_n ) .
	\end{equation} \Nobs that the fact that $T>0$,  \cref{prop:gd:limsup:eq:gamma}, and \cref{prop:gd:eq:help1} ensure that
	\begin{equation}
	\begin{split}
	V_\xi ( \Theta_1 ) - V_\xi ( \Theta_0 )
	& \le  2 \gamma _0 L \rbr*{ \cL _\infty ( \Theta_0 )  \br*{ 2 \gamma_0 P ( V_ \xi ( \theta ) ) - 1 } + \nu } \\
	& \le  2 \gamma _0 L \rbr*{ (\nu + \varepsilon ) \br*{1 - \varepsilon / (\nu + \varepsilon )} + \nu } = 0.
	\end{split}
	\end{equation}
	This establishes \cref{prop:gd:eq:indclaim} in the base case $n=0$.
	For the induction step let $n \in \N \cap [0, T)$ satisfy $V_\xi ( \Theta_0 ) \ge V_\xi ( \Theta_1 ) \ge \cdots \ge V_\xi ( \Theta_n ) $. 
	\Nobs that this implies that 
	$V_\xi ( \theta ) + 4 L  ^2 \norm{\xi} ^2 \ge V_\xi ( \Theta_n ) + 4 L  ^2 \norm{\xi} ^2  > 0$ and, hence,
	 $P ( V _\xi ( \theta )) \ge P ( V _\xi ( \Theta_n ))$.
	\cref{prop:gd:limsup:eq:gamma,prop:gd:eq:help1} therefore show that
	\begin{equation} \label{prop:gd:eq:help2}
	\begin{split}
V_\xi ( \Theta_{n+1} ) - V_\xi ( \Theta_n )
& \le  2 \gamma _n L \rbr*{ \cL _\infty ( \Theta_n )  \br*{ 2 \gamma _n P ( V_ \xi ( \theta ) ) - 1 } + \nu } \\
& \le  2 \gamma_n L \rbr*{ (\nu + \varepsilon ) \br*{1 - \varepsilon / (\nu + \varepsilon )} + \nu } = 0.
\end{split}
	\end{equation}
	This establishes \cref{prop:gd:eq:indclaim}. 
	\Cref{prop:gd:limsup:item1} thus follows from \cref{lem:lyap:est}.
	To prove \cref{prop:gd:limsup:item2,prop:gd:limsup:item3} we assume without loss of generality that $T = \infty $.
	\Nobs that \cref{prop:gd:eq:help2} ensures for all $N \in \N$ that
	\begin{equation}
	\begin{split}
	V _\xi ( \Theta_0 ) - V_\xi ( \Theta_N )
	&= \ssuml_{n=0}^{N - 1 } ( 	V _\xi ( \Theta_n ) - V_\xi ( \Theta_{n + 1 } ) ) \\
	& \ge 2 L \ssuml_{n=0}^{N - 1 } \rbr[\big]{ \gamma_n \cL _\infty ( \Theta_n )  \br*{ 1 - 2 \gamma _n P ( V_ \xi ( \theta ) ) } - \nu }
	\ge 2 L \nu \ssuml_{n=0}^{N - 1 } \gamma_n \rbr*{ \tfrac{\cL_\infty  ( \Theta_n ) }{ \nu + \varepsilon } - 1 }
	.
	\end{split}
	\end{equation}
	Combining this with \cref{lem:lyap:est} demonstrates that
	\begin{equation}
		\label{prop:gd:eq_sum_finite}
	2  L \nu \ssuml_{n=0}^{ \infty } \gamma _n \rbr*{ \tfrac{\cL_\infty  ( \Theta_n ) }{ \nu + \varepsilon } - 1 }
	\le V_\xi ( \Theta_0 ) - \inf_{N \in \N} V_\xi ( \Theta_N ) \le V_\xi ( \Theta_0 ) + 2 L ^2 \norm{\xi} ^2 < \infty .
	\end{equation}
	This, the assumption that $\sum_{n=0}^\infty \gamma_n = \infty$,
	 and the fact that for all $n \in \N_0$ it holds that $ \cL _\infty ( \Theta_n ) \ge \nu + \varepsilon > 0$ assure that
	\begin{equation}
	\liminf\nolimits_{n \to \infty} \rbr*{ \tfrac{\cL_\infty  ( \Theta_n ) }{ \nu + \varepsilon } - 1 }  = 0.
	\end{equation}
	Hence, we obtain that $\liminf_{n \to \infty} \cL_\infty ( \Theta_n ) = \nu + \varepsilon$, which establishes \cref{prop:gd:limsup:item2}.
	To complete the proof, assume for the sake of contradiction that $\limsup_{n \to \infty} \cL_\infty ( \Theta_n ) > \nu + \varepsilon$.
	\Nobs that this implies that there exist $\delta \in (0 , \infty )$ and
	$(m_k, n_k ) \in \N^2$, $k \in \N$,
	which satisfy for all $k \in \N$ that $m_k < n_k < m_{k+1}$,
	$\cL_\infty ( \Theta_{m_k} ) > \nu + \varepsilon + 2 \delta$,
	$\cL_\infty ( \Theta_{n_k} ) < \nu + \varepsilon + \delta$,
	and $\forall \, j \in \N \cap [m_k, n_k) \colon \cL_\infty ( \Theta_j ) \ge \nu + \varepsilon + \delta$.
	Combining this with \cref{prop:gd:eq_sum_finite} demonstrates that
	\begin{equation}
		\begin{split}
			\ssuml_{k=1}^\infty \ssuml_{j=m_k}^{n_k - 1 } \gamma_j 
			\le \tfrac{\nu + \varepsilon }{ \delta }
			\ssuml_{k=1}^\infty \ssuml_{j=m_k}^{n_k - 1 } \gamma_j \rbr*{ \tfrac{\cL_\infty  ( \Theta_j ) }{ \nu + \varepsilon } - 1 } < \infty.
		\end{split}
	\end{equation}
	Furthermore, \cref{prop:gd:limsup:item1} and the fact that $\cG$ is locally bounded (cf.~\cite[Corollary 2.12]{HutzenthalerJentzenPohlRiekertScarpa2021})
	imply that there exists $D \in \N$ which satisfies for all $n \in \N_0$ that $\norm{\cG ( \Theta_n ) } < D$.
	The triangle inequality hence shows that
	\begin{equation}
		\ssuml_{k=1}^\infty \norm{\Theta_{m_k} - \Theta_{n_k} }
		\le \ssuml_{k=1}^\infty \ssuml_{j=m_k}^{n_k - 1 } \norm{\Theta_{j+1} - \Theta_j}
		= \ssuml_{k=1}^\infty \ssuml_{j=m_k}^{n_k - 1 } \gamma_j \norm{\cG ( \Theta_j ) }
		\le D \ssuml_{k=1}^\infty \ssuml_{j=m_k}^{n_k - 1 } \gamma_j  < \infty .
	\end{equation}
	Therefore, we obtain that $\limsup_{k \to \infty } \norm{\Theta_{m_k} - \Theta_{n_k} } = 0$.
	Combining this with \cref{prop:gd:limsup:item1} and the fact that $\cL$ is locally Lipschitz continuous (cf.~\cite[Lemma 2.10]{HutzenthalerJentzenPohlRiekertScarpa2021})
	ensures that 
    \begin{equation}
    \textstyle 
	  \limsup_{k \to \infty} \abs{ \cL _\infty ( \Theta_{m_k} ) - \cL _\infty ( \Theta_{n_k} )} = 0
	  .
    \end{equation}
	However, this contradicts the fact that $\forall \, k \in \N \colon \cL_\infty ( \Theta_{m_k} ) - \cL_\infty ( \Theta_{n_k} ) > \delta$.
	Hence, we obtain that $\limsup_{n \to \infty} \cL_\infty ( \Theta_n ) = \nu + \varepsilon$.
	 This establishes \cref{prop:gd:limsup:item3}.
\end{cproof}

In \cref{prop:gd:limsup} we needed to assume that all learning rates are sufficiently small depending on the required error $\nu + \varepsilon$.
Under the condition that the GD iterates remain bounded, which is a standard assumption in optimization theory, we can strengthen this statement and establish that the risk value in the limit is upper bounded by $\nu$.

\begin{prop}
	\label{prop:gd:limsup:bounded}
	Assume \cref{setting:dnn},
	let
	$\xi \in \R^{\ell_L}$, $\nu \in \R$
	satisfy $\nu = \int_{[a,b]^{\ell_0 } } \norm{ \realization{L, \theta }_\infty ( x ) - \xi } ^2 \, \mu ( \d x )$,
	let $(\gamma_n)_{n \in \N_0} \subseteq [0 , \infty )$ satisfy
	\begin{equation} \label{prop:gd:limsup:bounded:eq_gamma}
		\ssuml_{n=0}^\infty ( \gamma_n )^2 < \ssuml_{n=0}^\infty \gamma_n = \infty ,
	\end{equation}
	let $\Theta \colon \N_0 \to \R^\fd$
	satisfy for all 
	$n \in \N_0$
	that
	$\Theta_{n+1} = \Theta_n - \gamma _n \cG ( \Theta_n )$,
	let $T \in \N_0 \cup \cu{\infty }$ satisfy
	$T = \inf(\cu{n \in \N_0 \colon \cL_\infty ( \Theta_n) \le \nu } \cup \cu{\infty} )$,
	and assume $\sup_{n \in \N_0} \norm{\Theta_n } < \infty $.
	Then 
	\begin{enumerate}[label=(\roman*)]
		\item \label{prop:gd:limsup:bounded:item1} it holds that
		$\liminf_{n \to \infty} \cL_\infty ( \Theta_n ) \le \nu $
		and
		\item \label{prop:gd:limsup:bounded:item2} it holds that
		$\limsup_{n \to \infty} \cL_\infty ( \Theta_{\min \cu{n , T } } ) \le \nu $.
	\end{enumerate}
\end{prop}

\begin{cproof}{prop:gd:limsup:bounded}
	Throughout this proof let $P \colon \R \to \R$
	satisfy for all
	$y \in \R$
	that
	$P(y) = L \bfa ^2 \mu ( [ a,b]^d ) \prod_{p=0}^L ( \ell_p + 1 ) (2 y + 4 L^2 \norm{\xi}^2 + 1 ) ^{ L - 1 }$.
	\Nobs that the assumption that $\sup_{n \in \N_0} \norm{\Theta_n } < \infty $
	implies that there exists $D \in \N$
	which satisfies for all $n \in \N_0$ that $\cL_\infty ( \Theta_n ) P ( V_ \xi ( \Theta_n ) ) < D$.
	Moreover, as in the proof of \cref{prop:gd:limsup}, we obtain for all $n \in \N_0$ that
	$V_\xi ( \Theta_{n+1} ) - V_\xi ( \Theta_n ) \le 2 \gamma _n L \cL _\infty ( \Theta_n ) 
	\br*{2 \gamma _n P ( V_ \xi ( \Theta_n ) ) - 1 } + 2 \gamma_n L \nu $.
	This shows for all $N \in \N$ that
	\begin{equation}
		\begin{split}
			V_\xi ( \Theta_N) - V_\xi ( \Theta_0 ) 
			&= \ssuml_{n=0}^{N-1} (V_\xi ( \Theta_{n+1} ) - V_\xi ( \Theta_n )  ) \\
			&\le 2 L \ssuml_{n=0}^{N-1} \gamma_n ( \nu - \cL_\infty ( \Theta_n ) )
			+ 4 L \ssuml_{n=0}^{N-1} \gamma_n^2 \cL_\infty ( \Theta_n ) P ( V_ \xi ( \Theta_n ) ).
		\end{split}
	\end{equation}
	Combining this with \cref{prop:gd:limsup:bounded:eq_gamma}
	and \cref{lem:lyap:est}
	demonstrates for all $N \in \N$ that
	\begin{equation}
		\label{prop:gd:limsup:bounded:eq_sum_finite}
		\begin{split}
			\ssuml_{n=0}^{N-1} \gamma_n ( \cL_\infty ( \Theta_n )  - \nu )
			& \le V_\xi ( \Theta_0 ) - V_\xi ( \Theta_N ) + 4 L D \ssuml_{n=0}^{N-1} \gamma_n^2 \\
			&\le V_\xi ( \Theta_0 ) + 2 L^2 \norm{\xi} ^2 + 4 L D \ssuml_{n=0}^{\infty} \gamma_n^2
			< \infty .
		\end{split}
	\end{equation}
	This and \cref{prop:gd:limsup:bounded:eq_gamma} imply that $\liminf_{n \to \infty } ( \cL_\infty ( \Theta_n )  - \nu ) \le 0$, which establishes \cref{prop:gd:limsup:bounded:item1}.
	
	To prove \cref{prop:gd:limsup:bounded:item2} assume without loss of generality that $T = \infty $,
	and assume for the sake of contradiction that $\limsup_{n \to \infty} \cL_\infty ( \Theta_n ) > \nu$.
	As in the proof of \cref{prop:gd:limsup:bounded} we can thus choose $\delta \in (0 , \infty )$ and
	$(m_k, n_k ) \in \N^2$, $k \in \N$,
	which satisfy for all $k \in \N$ that $m_k < n_k < m_{k+1}$,
	$\cL_\infty ( \Theta_{m_k} ) > \nu + 2 \delta$,
	$\cL_\infty ( \Theta_{n_k} ) < \nu + \delta$,
	and $\forall \, j \in \N \cap [m_k, n_k) \colon \cL_\infty ( \Theta_j ) \ge \nu + \delta$.
	Combining this with the fact that $\inf_{n \in \N_0} (\cL_\infty  ( \Theta_n ) - \nu ) \ge 0$ and \cref{prop:gd:limsup:bounded:eq_sum_finite} demonstrates that
	\begin{equation}
		\begin{split}
			\ssuml_{k=1}^\infty \ssuml_{j=m_k}^{n_k - 1 } \gamma_j 
			\le \tfrac{ 1}{ \delta }
			\ssuml_{k=1}^\infty \ssuml_{j=m_k}^{n_k - 1 } \gamma_j \rbr*{ \cL_\infty  ( \Theta_j ) - \nu }
			\le \tfrac{ 1}{ \delta } \ssuml_{n=0}^\infty \gamma_n ( \cL_\infty ( \Theta_n ) - \nu ) < \infty.
		\end{split}
	\end{equation}
	Furthermore, the fact that $\cG$ is locally bounded (cf.~\cite[Corollary 2.12]{HutzenthalerJentzenPohlRiekertScarpa2021})
	implies that there exists $D \in \N$ which satisfies for all $n \in \N_0$ that $\norm{\cG ( \Theta_n ) } < D$.
	The triangle inequality hence shows that
	\begin{equation}
		\ssuml_{k=1}^\infty \norm{\Theta_{m_k} - \Theta_{n_k} }
		\le \ssuml_{k=1}^\infty \ssuml_{j=m_k}^{n_k - 1 } \norm{\Theta_{j+1} - \Theta_j}
		= \ssuml_{k=1}^\infty \ssuml_{j=m_k}^{n_k - 1 } \gamma_j \norm{\cG ( \Theta_j ) }
		\le D \ssuml_{k=1}^\infty \ssuml_{j=m_k}^{n_k - 1 } \gamma_j  < \infty .
	\end{equation}
	Therefore, we obtain that $\limsup_{k \to \infty } \norm{\Theta_{m_k} - \Theta_{n_k} } = 0$.
	Combining this with the fact that $\cL$ is locally Lipschitz continuous (cf.~\cite[Lemma 2.10]{HutzenthalerJentzenPohlRiekertScarpa2021})
	ensures that $\limsup_{k \to \infty} \abs{ \cL _\infty ( \Theta_{m_k} ) - \cL _\infty ( \Theta_{n_k} )} = 0$.
	However, this contradicts the fact that $\forall \, k \in \N \colon \cL_\infty ( \Theta_{m_k} ) - \cL_\infty ( \Theta_{n_k} ) > \delta$.
	Hence, we obtain that $\limsup_{n \to \infty} \cL_\infty ( \Theta_n ) \le \nu$.
	This establishes \cref{prop:gd:limsup:bounded:item2}.
\end{cproof}

Finally, we establish a consequence of \cref{prop:gd:limsup} for all sufficiently small initial learning rates.

\begin{cor} \label{cor:gd:limsup}
	Assume \cref{setting:dnn},
	assume $\mu ( [ a , b ] ^{ \ell_0 } ) > 0$,
	let $\theta \in \R^\fd$, $\rho \in [0 , 1 ]$, $\xi \in \R^{ \ell_L }$, $\nu \in \R$ satisfy
	$\nu = \int_{[a,b] ^{ \ell_0 } } \norm{\realization{L, \theta} _\infty ( x ) - \xi } ^2 \, \mu ( \d x )$,
	and for every $\gamma \in (0 , \infty )$ let $\Theta^\gamma = (\Theta^\gamma_n)_{n \in \N_0 } \colon \N_0 \to \R^\fd$ satisfy $\Theta^\gamma_0 = \theta$ and $\forall \, n \in \N_0 \colon \Theta^\gamma_{n+1} = \Theta^\gamma_n - \gamma (n+1)^{ - \rho} \cG ( \Theta_n^\gamma )$.
	Then 
	\begin{equation}
		\label{cor:gd:limsup:eq_claim}
		\limsup\nolimits_{\gamma \searrow 0 } \inf\nolimits_{n \in \N_0 } \cL_\infty ( \Theta_n^\gamma ) \le \nu .
	\end{equation}
\end{cor}

\begin{cproof}{cor:gd:limsup}
	\Nobs that for all $\gamma \in (0 , \infty )$ it holds that $\sum_{n=0}^\infty \gamma (n+1)^{ - \rho} = \infty $.
	Combining this with \cref{prop:gd:limsup} demonstrates that for all $\varepsilon \in (0 , \infty )$ there exists $\eta \in (0 , \infty )$ which satisfies for all $\gamma \in (0 , \eta )$ that $\inf\nolimits_{n \in \N_0 } \cL_\infty ( \Theta_n^\gamma ) \le \nu + \varepsilon$.
	This establishes \cref{cor:gd:limsup:eq_claim}.
\end{cproof}

\subsection*{Acknowledgments}
This work has been partially funded by the European Union (ERC, MONTECARLO, 101045811). 
The views and the opinions expressed in this work are however those of the authors only 
and do not necessarily reflect those of the European Union or the European Research Council. 
Neither the European Union nor the granting authority can be held responsible for them.
In addition, this work has been partially funded by the Deutsche Forschungsgemeinschaft 
(DFG, German Research Foundation) under Germany’s Excellence Strategy EXC 2044-390685587, 
Mathematics M\"{u}nster: Dynamics-Geometry-Structure.

\begin{figure}[h]
        \centering
        \includegraphics[width=70mm]{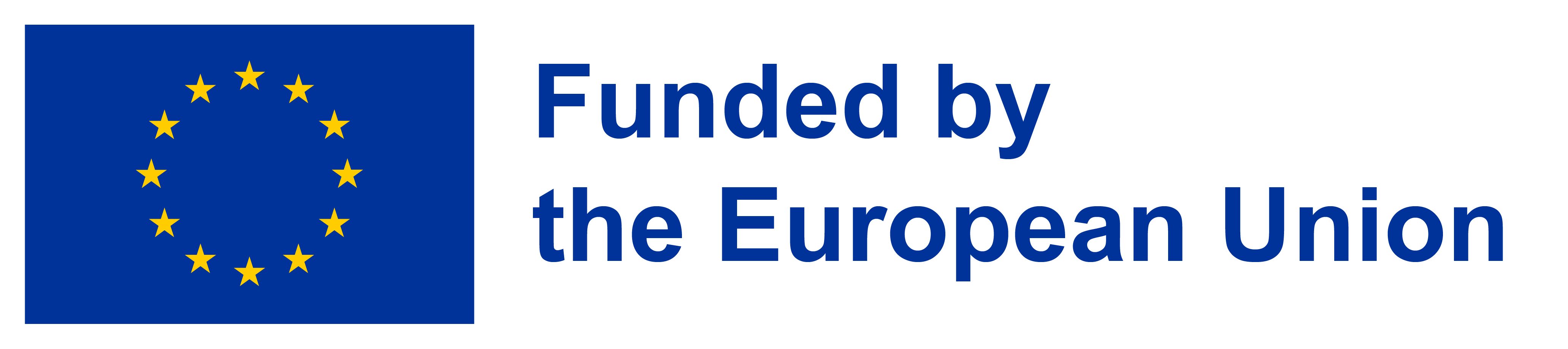}
        \label{fig:meine-grafik}
\end{figure}


\end{document}